\documentclass[11pt,twoside]{article}
\usepackage{amsmath}
\usepackage{amsthm}
\usepackage{amsfonts}
\usepackage{amssymb}
\usepackage{mathtools}
\usepackage[latin1]{inputenc}
\usepackage{color} 
\numberwithin{equation}{section}
\usepackage{hyperref}

\newtheoremstyle{thmstyle}
  {3pt}{3pt}{\itshape}{}{\bfseries}{}{.5em}
  {\thmname{#1}\thmnumber{ #2}\thmnote{ \textmd{(#3)}}}
\theoremstyle{thmstyle}

\newtheorem{theorem}{Theorem}[section]
\newtheorem{lemma}[theorem]{Lemma}
\newtheorem{coro}[theorem]{Corollary}
\newtheorem{prop}[theorem]{Proposition}
\newtheorem{definition}[theorem]{Definition}
\newtheorem{remark}[theorem]{Remark}
\newtheorem{example}[theorem]{Example}

\title{ title }
\author{ author }

\newcommand{\R}{\mathbb{R}}
\newcommand{\Rn}{\mathbb{R}^n}
\newcommand{\Rm}{\mathbb{R}^m}
\newcommand{\N}{\mathbb{N}}

\newcommand{\acspace}{AC_{L^p}([a,b],N)}
\newcommand{\acspaceg}{AC_{L^p}([a,b],G)}
\newcommand{\lp}{\mathcal{L}^p}

\newcommand{\funcion}{\Lambda_{\xi,\eta}}

\DeclareMathOperator*{\esssup}{ess\,sup}

\DeclareMathOperator{\Evol}{Evol}
\DeclareMathOperator{\pr}{pr}
\DeclareMathOperator*{\Diff}{Diff^r}
\DeclareMathOperator*{\Dif}{Diff_K^r}
\DeclareMathOperator*{\Difff}{Diff_c^\infty}
\DeclareMathOperator*{\id}{id}

\begin{document}
\begin{center}
\textbf{\Large Manifolds of absolutely continuous functions
with values in}\\[2mm]
\textbf{\Large an infinite-dimensional manifold and regularity }\\[2mm]
\textbf{\Large properties of half-Lie groups }\\[3mm]
\textbf{Matthieu F. Pinaud}
\end{center}

\begin{abstract}
\noindent
For $p\in [1,\infty]$, we define a smooth manifold structure on the set $AC_{L^p}([a,b],N)$ of absolutely continuous functions $\gamma\colon [a,b]\to N$ with $L^p$-derivatives for all real numbers $a<b$ and each smooth manifold $N$ modeled on a sequentially complete locally convex topological vector space, such that $N$ admits a local addition. Smoothness of natural mappings between spaces of absolutely continuous functions is discussed, like superposition operators $AC_{L^p}([a,b],N_1)\to AC_{L^p}([a,b],N_2)$, $\eta\mapsto f\circ \eta$, for a smooth map $f\colon N_1\to N_2$.
Let $G$ be a right half-Lie group. For $1\leq p <\infty$, we show certains properties of $L^p$-semiregularity of $G$ under certains condicions on $G$. For $r\in \mathbb{N}$, we prove that the right half-Lie groups  $\text{Diff}_K^r(\mathbb{R})$ and $\Diff(M)$ are $L^p$-semiregular. Here $K$ is a compact subset of $\mathbb{R}$ and $M$ is a compact smooth manifold (possible with boundary). An $L^p$-semiregular half-Lie group $G$ admits an evolution map $\text{Evol}:L^p([0,1],T_e G)\to AC_{L^p}([0,1],G)$, where $e$ is the neutral element of $G$. For the preceding examples, the evolution map $\text{Evol}$ is continuous.
\end{abstract}

\section{Introduction}
The study of mapping manifolds is an important branch of geometry and global analysis. For example, for $k\in \N\cup \{0\}$, a well-known case is the manifold of the set of $C^k$-maps $f:M\to N$, denoted by $C^k(M,N)$, for a compact manifold $M$ and a smooth manifold $N$ admitting local addition (see e.g. \cite{ASm, Els, Mor}). In \cite{Kri}, Krikorian shows how to contruct smooth manifold structures for set of functions with H\"older conditions and Sobolev functions. \\ \\
\noindent 
Our main focus will be the study of absolutely continuous functions with values in a manifold $N$ modeled on a sequentially complete locally convex space which admits a local addition. For a Hilbert manifold~$N$,
a smooth manifold structure
on the set $AC_{L^2}([0,1],N)$
of absolutely continuous
paths with $L^2$-derivatives in local charts
was used by Flaschel and Klingenberg
to study closed geodesics
in Riemannian manifolds (see \cite{FKl} and \cite{Krg},
cf.\ also \cite{Sm1}).
For Banach manifolds~$N$ admitting a smooth local
addition
and $p\in [1,\infty]$, Pierron and Trouv\'e show that a smooth manifold structure
on $AC_{L^p}([0,1],N)$ may also be obtained
using a method of Krikoian~\cite{Kri},
which is similar to Palais' use of
Banach section functors~\cite{Pls}.
A Lie group structure (and hence a smooth manifold structure)
on $AC_{L^p}([0,1],G)$
was obtained in \cite{Nin1} for each Lie group~$G$
modeled on a sequentially complete locally convex space
(as in \cite{Mil}),
generalizing the case of Fr\'{e}chet--Lie groups
treated in~\cite{GL2}. \\ \\
\noindent
We construct manifolds of absolutely continuous
functions in higher generality. To formulate
the main result, let us fix notation.

\begin{definition}\label{loc-add}
Let $N$ be a smooth manifold
modeled on real a locally convex space,
with tangent bundle $TN$ and its
bundle projection $\pi_{TN}\colon TN\to N$.
A \emph{local addition} on~$N$
is a map
\[
\Sigma\colon \Omega\to N,
\]
defined on an open set $\Omega\subseteq TN$
which contains the zero-vector $0_p\in T_pN$ for each $p\in N$,
such that $\Sigma(0_p)=p$ for all $p\in N$ and
\[
\theta\colon \Omega\to N\times N,\quad v\mapsto (\pi_{TN}(v),\Sigma(v))
\]
has open image
and is a $C^\infty$-diffeomorphism onto its
image.
\end{definition}
\noindent
Throughout the following, $a<b$ are real numbers
    and $p\in[1,\infty]$.
For a definition of absolutely
continuous functions with values in a sequentially
complete locally convex space~$E$ or a smooth manifold~$N$
defined thereon, the reader is referred to Definitions~\ref{defABE}
and \ref{defACM}, respectively (see also \cite{Nin1}).
For $\eta\in AC_{L^p}([a,b],N)$,
the pointwise operations make
\[
\Gamma_{AC}(\eta):=\{\tau\in AC_{L^p}([a,b],TN)\colon \pi_{TN}\circ \tau=\eta\}
\]
a vector space;
we endow it with a natural topology making it a locally
convex topological vector space (cf.\ Definition~\ref{deftop}).
We shall see that
\[
{\mathcal V}_\eta :=\{\tau\in \Gamma_{AC}(\eta)\colon \tau([a,b])\subseteq \Omega\}
\]
is an open $0$-neighborhood in $\Gamma_{AC}(\eta)$.
Setting
\[
{\mathcal U}_\eta:=\{\gamma\in AC_{L^p}([a,b],N)\colon (\eta (t),\gamma(t))\in \Omega'\;
\mbox{for all $t\in [a,b]$}\},
\]
the map
\[
\Psi_\eta\colon {\mathcal V}_\theta\to {\mathcal U}_\theta,\quad \tau\mapsto\Sigma\circ \tau
\]
is a bijection (see Remark \ref{recharts}). We show (see Theorem~\ref{teomanifold}):
\begin{theorem}\label{main-1}
For each smooth manifold~$N$
modeled on a sequentially complete
locally convex space and $p\in[1,\infty]$,
the set $AC_{L^p}([a,b],N)$
of all $AC_{L^p}$-maps $\gamma\colon [a,b]\to N$
admits a smooth manifold structure
such that for each local addition
$\Sigma\colon \Omega\to N$,
the sets ${\mathcal U}_\eta$ are open in $AC_{L^p}([a,b],N)$
for all $\eta\in AC_{L^p}([a,b],N)$
and $\Psi_\eta\colon {\mathcal V}_\eta\to {\mathcal U}_\eta$
is a $C^\infty$-diffeomorphism.
\end{theorem}
\noindent
Using the smooth manifold structures just described,
we find:
\begin{theorem}\label{intro fon}
Let $f\colon N_1\to N_2$
be a smooth map between smooth manifolds
$N_1$ and $N_2$ modeled on sequentially
complete locally convex spaces such that
$N_1$ and $N_2$ admit a local addition and $p\in [1,\infty]$.
We then have $f\circ\gamma\in AC_{L^p}([a,b], N_2)$
for all $\gamma\in AC_{L^p}([a,b],N_1)$;
the map
\begin{equation*}\label{ac-nonlin}
AC_{L^p}([a,b],f)
\colon AC_{L^p}([a,b],N_1)\to AC_{L^p}([a,b],N_2),\quad \gamma\mapsto f\circ \gamma
\end{equation*}
is smooth.
\end{theorem}
\noindent
More generally,
$AC_{L^p}([a,b],f)$ is $C^r$ for $r\in \N\cup\{0,\infty\}$
whenever $f$ is $C^{r+2}$ (see Proposition~\ref{propfon}).\medskip

\noindent
If, in the situation of Definition~\ref{loc-add},
$N$ is a $\mathbb{K}$-analytic manifold
modeled on a locally convex
topological $\mathbb{K}$-vector space
over $\mathbb{K}\in \{\mathbb{R},\mathbb{C}\}$
and $\theta\colon \Omega\to \Omega'$
is a diffeomorphism of
$\mathbb{K}$-analytic manifolds,
then $\Sigma$ is called a \emph{$\mathbb{K}$-analytic local addition}.
In this case, $AC_{L^p}([a,b],N)$
can be given a $\mathbb{K}$-analytic manifold structure
modeled on the locally convex topological
$\mathbb{K}$-vector spaces $\Gamma_{AC}(\eta)$
with properties as in Theorem~\ref{main-1},
replacing $C^\infty$-diffeomorphisms
with diffeomorphisms of $\mathbb{K}$-analytic
manifolds there (see Corollary \ref{manalytic}).
If $N_j$ is a $\mathbb{K}$-analytic manifold
modeled on a sequentially
complete locally convex topological $\mathbb{K}$-vector
space such that $N_j$
admits a $\mathbb{K}$-analytic local addition
for $j\in \{1,2\}$,
then the map
$AC_{L^p}([a,b],f)$ described in Theorem \ref{intro fon}
is $\mathbb{K}$-analytic for all $p\in [1,\infty]$ (see Corollary \ref{fanalytic}).\medskip

\noindent
If $\pi_{TN}:TN\to N$ denotes the tangent bundle of $N$. Followings results on more general manifolds of mappings (see e.g. \cite{AGS}), we can conclude that the map
\[ AC_{L^p}([a,b],\pi_{TN}):AC_{L^p}([a,b],TN)\to AC_{L^p}([a,b],N),\quad \tau \mapsto \pi_{TN}\circ \tau\]
defines a vector bundle on $\acspace$. The proofs are analogous to the reference given, but we will show them for the sake of completeness. \\
\noindent
Manifolds of absolutely continuous paths
occur in the regularity theory of infinite-dimensional
Lie groups.
Consider a Lie group $G$ modeled on a sequentially
complete locally convex space, with
tangent space ${\mathfrak g}:=T_eG$
at the neutral element $e\in G$.
For $g\in G$, let $\rho_g\colon G\to G$, $x\mapsto xg$ be the right
multiplication with~$g$. Then
\[
TG\times G\to TG,\quad (v,g)\mapsto T\rho_g(v)=:v.g
\]
is a smooth map and a right action of $G$ on $TG$.
The following concept strengthens
``regularity'' in the sense of Milnor~\cite{Mil}.
\begin{definition}
Following \cite{Nin1} (cf.\ also \cite{GL2}),
$G$ is called \emph{$L^p$-semiregular}
if, for each $\gamma\in {\mathcal L}^p([0,1],\mathfrak{g})$, the initial value problem
\begin{align}\label{ode-analog}
\dot{\eta}(t)&=\gamma(t).\eta(t), \quad t\in [0,1]\\ 
\eta(0)&=e
\end{align}
has an
$AC_{L^p}$-solution $\eta\colon [0,1]\to G$.
Then $\eta$ is necessarily unique;
we call $\eta$ the \emph{evolution} of~$\gamma$
and write $\Evol([\gamma]):=\eta$.
If $G$ is $L^p$-semiregular and $\Evol\colon L^p([0,1],\mathfrak{g})
\to AC_{L^p}([0,1],G)$
is smooth, then $G$ is called \emph{$L^p$-regular}.
\end{definition}
\begin{remark}
If $G$ is modeled on a Fr\'{e}chet
space, (\ref{ode-analog})
is required to hold for almost all $t\in [0,1]$
with respect to Lebesgue measure.
In the general case, $\eta$ is required to
be a Carath\'{e}odory solution to~(\ref{ode-analog}),
i.e., it solves the corresponding integral equation
piecewise in local charts.
We mention that a Lie group $G$
is $L^p$-regular if and only
if $G$ is $L^p$-semiregular and $\Evol$
is smooth as map
\[
L^p([0,1],\mathfrak{g})\to C([0,1],G)
\]
(see \cite{Nin1}).
The latter holds whenever
$\Evol\colon L^p([0,1],\mathfrak{g})\to C([0,1],G)$
is continuous at~$0$ (see \cite[Theorem~E]{GL3}).
\end{remark}

\noindent
Now consider a half-Lie group~$G$
modeled on a sequentially complete
locally convex space~$E$.
Thus~$G$ is a group, endowed
with a smooth manifold structure modeled
on~$E$ which makes $G$
a topological group and turns
the right translation $\rho_g\colon G\to G$
into a smooth mapping for each $g\in G$
(cf.\ \cite{BHM,MNb}).
Let us use notation as in the case of Lie groups.
\begin{definition}
We say that a half-Lie group $G$ is \emph{$L^p$-semiregular}
if the differential equation
\begin{equation}
    \dot{y}(t)=\gamma(t).y(t),\quad t\in [0,1]
\end{equation}
satisfies local uniqueness of Carath\'{e}odory
solutions for each $\gamma\in {\mathcal L}^p([0,1],\mathfrak{g})$
(in the sense of \cite{GHi}) and the initial
value problem (\ref{ode-analog})
has a Carath\'{e}odory solution $\Evol(\gamma):=\eta\colon [0,1]\to G$.
\end{definition}
\noindent
The Lie group $\text{Diff}(M)$
of $C^\infty$-diffeomorphisms of a compact smooth
manifold~$M$ without boundary
is known to be $L^1$-regular,
and also the Lie group $\text{Diff}_K(\mathbb{R}^n)$
of all $C^\infty$-diffeomorphisms
$\phi\colon \mathbb{R}^n\to \mathbb{R}^n$
such that $\phi(x)=x$ for all $x\in \mathbb{R}^n\setminus K$,
for each compact subset $K\subseteq \mathbb{R}^n$
(see \cite{GL2}). For each positive integer~$r$,
the following analogues are obtained
(see Theorems ~\ref{text-main-4}
and \ref{text-main-3}):
\begin{theorem}\label{main-3}
Let $1\leq p <\infty$. For each compact smooth manifold~$M$
without boundary and $r\in \N$, 
the half-Lie group $G:=\mbox{\emph{Diff}}^r(M)$
of all $C^r$-diffeomorphisms $\phi\colon M\to M$
is $L^p$-semiregular.
Moreover, its evolution map 
\[
\mbox{\emph{Evol}}\colon L^p([0,1],\mathfrak{g})\to AC_{L^p}([0,1],G)
\]
is continuous.
\end{theorem}
\noindent
Here $\mathfrak{g}$ is the Banach space of
$C^r$-vector fields on~$M$.
\begin{theorem}\label{main-4}
Let $1\leq p <\infty$ and $r\in \N$. For each positive integer~$n$
and compact subset $K$ of $\mathbb{R}^n$,
the half-Lie group $G:=\mbox{\emph{Diff}}^r_K(\mathbb{R}^n)$
of all $C^r$-diffeomorphisms $\phi\colon \mathbb{R}^n\to \mathbb{R}^n$
with \emph{$\phi|_{\mathbb{R}^n\setminus K}=\text{id}_{\mathbb{R}^n\setminus K}$}
is $L^p$-semiregular. Moreover, its evolution map 
\[
\mbox{\emph{Evol}}\colon L^p([0,1],\mathfrak{g})\to AC_{L^p}([0,1],G)
\]
is continuous.\\ 
If we replace $L^p$ with $L_{rc}^\infty$ (see Definition \ref{Linfinito}) the preceding theorem remains valid.
\end{theorem}
\noindent
Here $\mathfrak{g}$ is the Banach space of
all $C^r$-vector fields on~$\mathbb{R}^n$
which vanish outside~$K$.\medskip

\noindent
For an $L^p$-semiregular
half-Lie group $G$ admitting a local addition with $1\leq p <\infty$,
the smooth manifold structure on $AC_{L^p}([0,1],G)$
provided by Theorem~\ref{main-1}
makes it possible to discuss
continuity properties and differentiability
properties of the evolution map as a mapping
\[
\Evol\colon L^p([0,1],\mathfrak{g})\to AC_{L^p}([0,1],G).
\]
So far, we have one positive result in this regard:

\begin{theorem}
    Let $G$ be a right half-Lie group modeled in a sequentially complete locally convex space space $E$ which admits a local addition and $1\leq p < \infty$. Let $G$ be $L^p$-semiregular with continuous evolution map      
    \[\Evol_C:L^p([0,1],T_eG),\to C([0,1],G),\quad \gamma\mapsto \Evol_C(\gamma).\]
    If the restriction of the right action
    \[ \tau : T_eG\times G \to TG,\quad (v,g)\mapsto v.g \]
    is continuous, then the evolution map
    \[ \Evol:L^p([0,1],T_eG),\to AC_{L^p}([0,1],G),\quad \gamma\mapsto \Evol(\gamma)\]
    is continuous. If $G$ is a right half-Lie group modeled in an integral complete locally convex space $E$, then if we replace $L^p$ with $L_{rc}^\infty$ the result remains valid.
\end{theorem}

\noindent
So far, $C^0$-regularity has been
investigated for the half-Lie group $\Diff(M)$
in a suitable sense (see the sketch in \cite{MCM}). Independently, related questions of regularity have been considered by Pierron and Trouv\'e (see \cite{PaT}).

\section{Preliminaries}
\begin{definition}
Let $E$ and $F$ be real locally convex spaces, $U\subseteq E$ be open and $f:U\to F$ be a map. We say that $f$ is $C^0$ if it is continuous. We say that $f$ is $C^1$ if $f$ is continuous, the directional derivative
\[ df(x,y):= (D_yf)(x):=\lim_{t\to 0} \frac{1}{t}(f(x+ty)-f(x)) \]
(with $t\neq 0$) exists in $F$ for all $(x,y)\in U\times E$, and $df:U\times E\to F$ is continuous. Recursively, for $k\in \N$ we say that $f$ is $C^k$ if $f$ is $C^1$ and $df:U\times E\to F$ is $C^{k-1}$. We say that $f$ is $C^\infty$ (or smooth) if $f$ is $C^k$ for each $k\in \N$.
\end{definition}

\begin{definition}
Let $E_1, E_2$ and $F$ be real locally convex spaces, $U\subseteq E_1$ and $V\subseteq E_2$ be open subsets, $r, s\in \N\cup\{0,\infty\}$ and $f:U\times V\to F$ be a map. If the iterated directional derivatives 
\[ d^{(i,j)} f ((x,a),y_1,...,y_i,b_1,...,b_j):= ( D_{(y_1,0)}...D_{(y_i,0)}D_{(0,b_1)}...D_{(0,b_j)} f)(x,a) \]
exist for all $i, j \in \N\cup\{0\}$ such that $i\leq r$ and $j\leq s$, and all $y_1,...,y_i \in E_1$ and $b_1,...,b_j \in E_2$, and, we assume that the mappings
\[ d^{(i,j)}f:U\times V \times E_1^i \times E_2^j \to F \]
are continuous, then $f$ is called a $C^{r,s}$-map.
\end{definition}

\begin{definition}
Let $X$ be a locally compact topological space, endowed with a measure \\
$\mu:\mathcal{B}(X)\to [0,\infty]$ on its $\sigma$-algebra of Borel sets and let $Y$ be a topological space. A function $\gamma:X\to Y$ is called Lusin $\mu$-measurable (or $\mu$-measurable) if for each compact subset $K\subseteq X$, there exists a sequence $(K_n)_{n\in \N}$ of compact subsets $K_n\subseteq K$ such that each restriction $\gamma|_{K_n}$ is continuous and $\mu\left( K \backslash \cup_{n\in \N} K_n\right)=0$.
\end{definition}
\noindent
For details of the construction of Lebesgue spaces, we refer the reader to see \cite{Nin1}.

\begin{definition}
Let $E$ be a locally convex space, $a<b$ be real numbers, $1\leq p <\infty$ and $\lambda:\mathcal{B}([a,b])\to [0,\infty)$ be the restriction of the Lebesgue-Borel measure on $\R$. We define the set $\lp([a,b],E)$ as the set of all Lusin $\lambda$-measurable functions $\gamma : [a,b]\to E$ such that for each continuous seminorm $q$ on $E$ we have 
\[q\circ \gamma \in \lp([a,b],\R).\]
And we endow it with the locally convex topology defined by the family of seminorms 
\[ \lVert \cdot \lVert_{\lp,q} : \lp([a,b],E) \to [0,\infty[,\quad \lVert \gamma \lVert_{\lp,q}:= \lVert q\circ \gamma\lVert_{\lp}.\]
Let $\gamma \sim \eta$ if and only if $\gamma(t) = \eta(t)$ for almost all $t\in [a,b]$ and write $[\gamma]$ for the equivalence class of $\gamma$. We define the Hausdorff locally convex space
\[ L^p([a,b],E):= \lp([a,b],E) / [0] \]
with seminorms 
\[ \lVert \cdot \lVert_{L^p,q} : L^p([a,b],E) \to [0,\infty[,\quad \lVert [\gamma] \lVert_{L^p,q}:= \lVert \gamma\lVert_{\lp,q}.\]
\end{definition}

\begin{definition}\label{Linfinito} Let $E$ be a locally convex space, $a<b$ be real numbers and $\lambda:\mathcal{B}([a,b])\to [0,\infty)$ be the restriction of the Lebesgue measure on $\R$. We define the set $\mathcal{L}^{\infty}([a,b],E)$ of all Lusin $\lambda$-measurable, essentially bounded functions $\gamma : [a,b]\to E$. For each continuous seminorm $q:E\to \Rn$ we define the seminorm 
\[ \lVert \gamma \lVert_{\mathcal{L}^{\infty},q} :=\esssup_{t\in [a,b]} q\circ \gamma(t). \]
We endow $\mathcal{L}^{\infty}([a,b],E)$ with the Hausdorff locally convex topology given by these seminorms.\\
Let $\gamma \sim \eta$ if and only if $\gamma = \eta$ a.e. We define the Hausdorff locally convex space
\[ L^{\infty} ([a,b],E):= \mathcal{L}^{\infty}([a,b],E) /[0] \]
with seminorms 
\[ \lVert \cdot \lVert_{L^{\infty},q} : L^{\infty}([a,b],E) \to [0,\infty[,\quad \lVert [\gamma] \lVert_{L^{\infty},q}:= \lVert \gamma\lVert_{\mathcal{L}^{\infty},q}.\]
We define the vector space $\mathcal{L}_{rc}^\infty([a,b],E)$ of all Borel measurable functions $\gamma:[a,b]\to E$ such that the image $\text{Im}(\gamma)$ has compact and metrizable closure, endowed with the topology induced by $\mathcal{L}^{\infty}([a,b],E)$. Thus
\[ L_{rc}^{\infty} ([a,b],E):= \mathcal{L}_{rc}^{\infty}([a,b],E) /[0] \]
is a Hausdorff locally convex space.
\end{definition}

\begin{remark}
Let $E$ be a Frechet space and $p\in [1,\infty]$. If $L_\mathcal{B}^p ([0,1],E)$ denotes the Lebesgue space constructed with the set of Borel measurable functions (see \cite{GL2}), then $L_\mathcal{B}^p ([0,1],E)$ coincides with $L^p([0,1],E)$ (\cite[Proposition 2.10]{Nin1}).
\end{remark}

\begin{remark}
Let $1\leq q\leq p < \infty$, then 
\[\mathcal{L}^\infty([a,b],E) \subseteq \mathcal{L}^p([a,b],E)\subseteq \mathcal{L}^q([a,b],E)\subseteq \mathcal{L}^1([a,b],E)\] 
as a consequence of H\"{o}lder's inequality.
\end{remark}

\begin{definition}
Let $E$ be a locally convex space and ${\gamma: [a,b]\to E}$ be such that $\alpha \circ \gamma \in \mathcal{L}^1 ([a,b],\R)$ for each $\alpha \in E'$. An element $z\in E$ is called the \textit{weak integral of $\gamma$} if 
\[ \quad \alpha(z) = \int_a^b (\alpha \circ \gamma)(s) ds,\]
for each $\forall \alpha \in E'$. Then $z$ is called the weak integral of $\gamma$ from $a$ to $b$, and we write $z=:\int_a^b \gamma(s)ds$.
\end{definition}

\begin{definition}
Let $E$ be a locally convex space. We say that a sequence $(x_n)_n \subset E$ is a Cauchy sequence if for each $\varepsilon > 0$ and each continuous seminorm $q$ of $E$, there exists an integer $N_{\varepsilon,q} \in \N$, such that for all $m,n\geq N_{\varepsilon,q}$ we have
\[ q(x_m-x_n)<\varepsilon. \]
We say that $E$ is sequentially complete if every Cauchy sequence converge in $E$.
\end{definition}
\noindent
The following lemma \cite[Proposition 2.26]{Nin1} allows us to define absolute continuous functions with vector values.
\begin{lemma}
If $E$ is sequentially complete locally convex space, then for each $\gamma \in \mathcal{L}^1([a,b],E)$, the weak integral $\int_a^b \gamma(s)ds$ exists and the map
\[ \eta:[a,b]\to E,\quad \eta(t)=\int_a^t \gamma(s)ds \]
is continuous. 
\end{lemma}

\noindent
A related important result of weak integrals is the Mean Value Theorem (see e.g. \cite{GL1}).
\begin{theorem}
Let $E$ and $F$ be locally convex spaces, $U\subseteq E$ be an open subset, $f:U\to F$ a $C^1$-map and $x,y\in U$ such that the line segment $\{ tx+(1-t)y \in E : t\in [0,1]\}$ is contained in $U$. Then
\[ f(y)-f(x)=\int_0^1 df(x+t(y-x),y-x)dt.\]
\end{theorem}

\begin{definition}\label{defABE}
Let $E$ be a sequentially complete locally convex space and $p\in[1,\infty]$. For $t_0 \in [a,b]$, we say that a function $\eta:[a,b]\to E$ is $L^p$-absolutely continuous (or just absolutely continuous if there is no confusion) if there exists a $[\gamma]\in L^p([a,b],E)$ such that 
\begin{equation}
    \eta(t)=\eta(t_0)+\int_{t_0}^t \gamma(s)ds, \quad t\in [a,b].
\end{equation}
We denote the space of all $L^p$-absolutely continuous functions by $AC_{L^p}([a,b],E)$.
Let $t_0\in [a,b]$ be fixed, since $\eta':=[\gamma]$ is necessarily unique (see \cite[Lemma 2.28]{Nin}), the map
\begin{equation}
    \Phi:AC_{L^p}([a,b],E)\to E\times L^p([a,b],E),\quad \eta\mapsto (\eta(t_0),\eta')
\end{equation}
is an isomorphism of vector spaces. We endow $AC_{L^p} ([a,b],E)$ with the Hausdorff locally convex vector topology which makes $\Phi$ an isomorphism of topological vector space (see \cite[Definition 3.1]{Nin}).
\end{definition}
\noindent
The following result will allow us to study differentiability of functions with values in $AC_{L^p}([a,b],E)$.
\begin{lemma}\label{actocl}
Let $E$ be a sequentially complete locally convex space and $p\in[1,\infty]$. Then the map 
\begin{equation}
    \Psi: AC_{L^p} ([a,b],E) \to C([a,b],E)\times {L^p} ([a,b],E),\quad \eta\mapsto (\eta, \eta')
\end{equation}
is a linear topological embedding with closed image.
\end{lemma}
\begin{proof}
We let $I:E\times L^p([a,b],E)\to C([a,b],E)$ be the continuous map given by
\[ I((x,[\gamma])(t) := x+\int_a^t \gamma(s)ds,\quad t\in [a,b]\]
for each $x\in E$ and $[\gamma]\in L^p([a,b],E)$. Let $\Phi:AC_{L^p}([a,b],E)\to E\times L^p([a,b],E)$, with $t_0:=a$, be the isomorphism of topological vector spaces as above. We consider the map
\[ \Theta:E\times L^p([a,b],E)\to C([a,b],E)\times {L^p} ([a,b],E),\quad  (x,[\gamma])\mapsto \big(I(x,[\gamma]),[\gamma]\big)\]
which is continuous.\\
Moreover, since the evaluation map $\varepsilon_a :C([a,b],E)\to E$, $\eta\mapsto \eta(a)$ is continuous, the map $\left(\Theta|^{\text{Im}(\Theta)}\right)^{-1}=(\varepsilon_a,\text{id}_{L^p})$ is also continuous. Hence $\Psi$ is a topological embedding.\\
Let $(\eta_\alpha, \eta_\alpha^\prime)_\alpha$ be a net in $\text{Im}(\Theta)$ that converges to $(\eta,[\gamma])\in C([a,b],E)\times {L^p} ([a,b],E)$. By continuity of $\varepsilon_a$, the net $(\eta_\alpha(a))_\alpha$ converges to $\eta(a)\in E$, and by continuity of $\Theta$, the net $\left(\Theta( \eta_\alpha(a),\eta_\alpha^\prime)\right)_\alpha$ converges to $\left( I(\eta(a),[\gamma]),[\gamma]\right) \in \text{Im}(\Theta)$. Since the net $\left(\Theta( \eta_\alpha(a),\eta_\alpha^\prime)\right)_\alpha$ also converges to $(\eta,[\gamma])$, we have
\[ I(\eta(a),[\gamma]) = \eta.\]
Therefore $\eta'=[\gamma]$ and $(\eta,[\gamma])\in \text{Im}(\Theta)$.
\end{proof}
\begin{remark}
Let $p\in[1,\infty]$. Since the inclusion map ${ AC_{L^p}([a,b],E)\to C([a,b],E)}$ is continuous \cite[Lemma 3.2]{Nin1}, the topology on $AC_{L^p}([a,b],E)$ is independent of the choice of $t_0$ and finer than the compact-open topology. Hence the sets
\[ AC_{L^p}([a,b],V):=\{ \eta \in AC_{L^p}([a,b],E) : \eta ([a,b])\subseteq V \}\]
are open on $AC_{L^p}([a,b],E)$, for each open subset $V\subseteq E$.
\end{remark}
\noindent
For maps between absolute continuous function spaces, we have the following results (see \cite{Nin1}).

\begin{lemma}\label{lemmanog} Let $E$ be a sequentially complete locally convex space and $p\in[1,\infty]$. For $c,d\in \R$ with $c<d$ we define a map $g:[c,d]\to [a,b]$ via
\[ g(t)=a+\frac{t-c}{d-c}(b-a),\quad t\in [c,d].\]
Then $\eta\circ g\in AC_{L^p}([c,d],E)$ for each $\eta \in AC_{L^p}([a,b],E)$ and the map 
\[ AC_{L^p}(g,E):AC_{L^p}([a,b],E) \to AC_{L^p}([c,d],E), \quad \eta \mapsto \eta\circ g\]
is continuous linear. 
\end{lemma}

\begin{lemma}\label{Acts0} Let $E$ and $F$ be sequentially complete locally convex spaces, $p\in[1,\infty]$, $V\subseteq E$ be open subset and ${f:V\to F}$ be a $C^1$-map. Then $f\circ \eta \in AC_{L^p}([a,b],F)$ for each $\eta\in AC_{L^p}([a,b],V)$.
\end{lemma}

\begin{lemma}\label{Acts}
Let $E$ and $F$ be sequentially complete locally convex spaces, $p\in[1,\infty]$ and $k\in \N \cup \{0,\infty\}$. Let $V\subseteq E$ be open subset and ${f:V\to F}$ be a $C^{k+2}$-map, then the map
\[ f_*:=AC_{L^p} ([a,b],f) : AC_{L^p}([a,b],V)\to AC_{L^p}([a,b],F),\quad \eta \mapsto f\circ \eta \]
is $C^{k}$. Moreover, we have 
\[ d(f_*)(\eta,\eta_1) = df \circ (\eta,\eta_1)\]
for all $(\eta,\eta_1)\in AC_{L^p}([a,b],V)\times AC_{L^p}([a,b],E)$.
\end{lemma}

\begin{remark}\label{times}
For sequentially complete locally convex spaces $E$ and $F$ we have
\[ AC_{L^p}([a,b],E\times F) \cong AC_{L^p}([a,b],E)\times AC_{L^p}([a,b],F). \]
\end{remark}
\noindent

\begin{definition}
Let $E$ and $F$ be complex topological vector spaces, where is $F$ is locally convex, $U\subseteq E$ be an open subset and $f:U\to F$ be a mapping. We say that $f$ is complex analytic if it is continuous and, for each $x\in U$, there exists a $0$-neighbouhood $V\subseteq E$ such that $x+V\subseteq U$ and certain continuous homogeneous polynomials $\beta_n:E\to F$ of degree $n$, such that $f$ admits the expansion: $f(x+y)=\sum_{n=0}^\infty \beta_n (y)$, for all $y\in V$. 
\end{definition}
\noindent
For our context, we present an application of \cite[Proposition 7.7]{Ber} to our particular case.
\begin{lemma}\label{GL}
Let $E$ and $F$ be complex locally convex spaces, and $f : U \to F$
be a mapping defined on an open subset of $E$. Then $f$ is complex analytic if and only if $f$ is smooth and the mapping $df (x, \cdot) : E \to F$ is complex linear for each $x\in U$.
\end{lemma}

\begin{definition}
Let $E$ and $F$ be real locally convex spaces, $U\subseteq E$ be an open subset and $ f: U \to F$ be a map. We say that $f$ is real analytic if it extends to an analytic map $f:V\to F_{\mathbb{C}}$ on some open neighborhood $V$ of $U$ in $E_\mathbb{C}$, where $E_\mathbb{C}$ and $F_\mathbb{C}$ denotes the complexification of $E$ and $F$, respectively.
\end{definition}

\begin{lemma}\label{Analy}
Let $E$ and $F$ be sequentially complete locally convex spaces over $\mathbb{K}\in\{\R,\mathbb{C}\}$, $p\in[1,\infty]$, $V\subseteq E$ be an open set and $f:V\to F$ be a $\mathbb{K}$-analytic map. Then the map
\[ f_*:=AC_{L^p} ([a,b],f): AC_{L^p}([a,b],V)\to AC_{L^p} ([a,b],F), \quad \eta \mapsto f\circ \eta \]
is $\mathbb{K}$-analytic.
\end{lemma}
\begin{proof}
First we consider the case $\mathbb{K}=\mathbb{C}$. By Lemma \ref{Acts} the map $f_*$ is smooth and $d(f_*)$ is complex linear in the second variable, hence by Lemma \ref{GL} the map $f_*$ is complex analytic.\\
If $\mathbb{K}=\mathbb{R}$, then by definition the map $f$ has a complex analytic extension $\tilde{f}$, hence $(\tilde{f})_*$ is the complex analytic extension of $f_*$, whence $f_*$ is real analytic.
\end{proof}

\begin{definition}\label{defACM}
Let $N$ be a smooth manifold modeled with on a sequentially complete locally convex space  $E$ and $p\in[1,\infty]$. We say that a function $\eta:[a,b]\to N$ is $L^p$-absolutely continuous if it is continuous and there exists a partition $\{t_0,...,t_n\}$ of $[a,b]$ such that for each $j\in \{1,...,n\}$, there exists a chart $\varphi_j:U_j\to V_j$ that verifies
\begin{itemize}
\item[i)]$\eta([t_{j-1},t_j])\subseteq U_j $.
\item[ii)]$\varphi_j \circ \eta|_{[t_{j-1},t_j]}\in AC_{L^p}([t_{j-1},t_j],E)$.
\end{itemize}  
In this case, we say that these charts verify the definition of $L^p$-absolute continuity for $\eta$. If there is no confusion, we simple call $\eta$ absolutely continuous.
\end{definition}
\noindent
For the case of absolutely continuous functions with values in a manifold, the following facts are available (see \cite{Nin1}).

\begin{lemma} 
Let $N$ be a smooth manifold modeled on a sequentially complete locally convex space $E$ and $p\in[1,\infty]$. If $\eta\in\acspace$, then 
\[ \varphi\circ \eta|_{[\alpha,\beta]}\in AC_{L^p}([\alpha,\beta],E),\]
for each chart $\varphi:U\to V$ of $N$ and each subinterval $[\alpha,\beta]\subseteq [a,b]$ such that $\eta([\alpha,\beta])\subseteq U$.
\end{lemma}

\begin{lemma} \label{pre1}
Let $M$ and $N$ be smooth manifolds modeled on sequentially complete locally convex spaces and $p\in[1,\infty]$. If $f:M\to N$ is a $C^1$-map, then $f\circ\eta \in AC_{L^p}([a,b],N)$ for each $\eta\in AC_{L^p}([a,b],M)$.
\end{lemma}

\section{The space of absolutely continuous sections}
\begin{definition}\label{deftop}
Let $p\in [1,\infty]$ and $N$ be a smooth manifold modeled on a sequentially complete locally convex space $E$. We denote by $\pi_{TN}:TN\to N$ its tangent bundle. For each $\eta \in AC_{L^p}([a,b],N)$ we define the set
\begin{equation}
    \Gamma_{AC}(\eta) := \{ \sigma \in AC_{L^p}([a,b], TN) : \pi_{TN} \circ \sigma = \eta \}
\end{equation}
and we endow it with the pointwise operations, making it a vector space.\\ 
Consider a partition $\{t_0,...,t_n\}$ of $[a,b]$ and charts $\{(\varphi_i,U_i): i\in \{1,...,n\}\}$ of $N$ that verify the definition of absolute continuity for $\eta$. Let $1\leq i\leq n$, since $\eta([t_{i-1},t_i])\subseteq U_i$, we have $\sigma ([t_{i-1},t_i])\subseteq TU_i$, so we endow $\Gamma_{AC}(\eta)$ with the Hausdorff locally convex vector topology which is the initial topology with respect to the linear mappings
\begin{equation}
     h_i:\Gamma_{AC}(\eta) \to  AC_{L^p}([t_{i-1},t_i],E),\quad \sigma \mapsto h_i(\sigma)= d\varphi_i \circ \sigma|_{[t_{i-1},t_i]}.
\end{equation}
In Remark \ref{independencepartition} we will show that this topology does not depend on the chosen partition.
\end{definition}

\begin{prop}\label{prop} Let $p\in [1,\infty]$ and $N$ be a smooth manifold modeled on a sequentially complete locally convex space $E$ and $\eta\in AC_{L^p}([a,b],N)$. If $P=\{t_0,...,t_n\}$ is a partition of $[a,b]$ and $\{(\varphi_i,U_i): i\in \{1,...,n\}\}$ are charts of $N$ that verify the definition of absolute continuity for $\eta$, then the map
\begin{equation}
    \Phi_{\eta,P}:\Gamma_{AC}(\eta)  \to \prod_{i=1}^n AC_{L^p}([t_{i-1},t_i],E),\quad \sigma \mapsto \left( d\varphi_i \circ \sigma|_{[t_{i-1},t_i]} \right)_{i=1}^n
\end{equation}
is a linear topological embedding with closed image. Its image $\text{Im}(\Phi_{\eta,P})$ is given by the set of all elements $(\tau_i)_{i=1}^n$ that verify
\[ \tau_{i} (t_i) = d\varphi_i \circ (T\varphi_{i+1})^{-1} \Big(\varphi_{i+1}\circ\eta(t_{i}),\tau_{i+1}(t_{i})\Big),\quad \text{ for all }i\in\{1,...,n-1\}.\]
\end{prop}

\begin{proof}
The linear map $\Phi_{\eta,P}$ is continuous by the previous definition and it is injective.\\ 
Let $i\in \{1,...,n\}$. If $W_i= \varphi_{i+1}(U_i\cap U_{i+1})$, then the map
\[ g_i:W_i\times E \to E,\quad (x,y)\mapsto d\varphi_i \circ (T\varphi_{i+1})^{-1} (x,y),\]
is continuous and linear in the second component. This enables us to define the closed vector subspace $K$ given by the elements 
\[(\tau_i)_{i=1}^n \in \prod_{i=1}^n AC_{L^p}([t_{i-1},t_i],E)\] 
which verify
\[ \tau_{i} (t_i) = d\varphi_i \circ (T\varphi_{i+1})^{-1} \Big(\varphi_{i+1}\circ\eta(t_{i}),\tau_{i+1}(t_{i})\Big),\quad \text{ for all }i\in\{1,...,n-1\}.\]
We will show that the image of $\Phi_{\eta,P}$, denoted by $\text{Im}(\Phi_{\eta,P})$, coincides with the closed subspace $K$. Indeed, the space $\text{Im}(\Phi_{\eta,P})$ is contained in $K$ by definition of $\Phi_{\eta,P}$. Let us consider now $\tau=(\tau_i)_{i=1}^n\in K$. We define the maps
\[\sigma_i:[t_{i-1},t_i]\to TN,\quad s\mapsto (T\varphi_{i})^{-1} \Big(\varphi_{i}\circ\eta(s),\tau_{i}(s)\Big)\]
and
\[ \sigma_{\tau}:[a,b]\to TN, \quad t\mapsto \sigma_i(t), \text{ for }t\in [t_{i-1},t_i].\]
By the Glueing Lemma, the map $ \sigma_{\tau}$ is continuous. 
Moreover, by Lemma \ref{pre1} each function $\sigma_i$ is absolutely continuous, hence $\sigma_\tau$ is too. Since ${\pi_{TN}\circ  \sigma_{\tau} = \eta}$ we have that $ \sigma_{\tau}\in \Gamma_{AC}(\eta)$ and $\Phi_{\eta,P}( \sigma_{\tau})=(\tau_i)_{i=1}^n$. Therefore $\text{Im}(\Phi_{\eta,P})=K$.\\ 
It remains to show that the inverse map
\[ \Phi_{\eta,P}^{-1}:\text{Im}(\Phi_{\eta,P})\to \Gamma_{AC}(\eta), \quad (\tau_i)_{i=1}^n\mapsto  \sigma_{\tau}\]
is continuous. If $h_i$ are the functions that define the topology (see Definition \ref{deftop}), then for each $i\in \{1,...,n\}$ we have $h_i \circ \Phi_{\eta,P}^{-1} = q_i$, where $q_i$ is the continuous linear map
\[ q_i: \prod_{j=1}^n AC_{L^p}([t_{j-1},t_j],E) \to AC_{L^p}([t_{i-1},t_i],E),\quad (\eta_j)_{j=1}^n\mapsto \eta_i. \]
Hence $\Phi_{\eta,P}^{-1}$ is continuous.
\end{proof}

\begin{remark}
From now on, we identify $\Phi_{\eta,P}$ with the homeomorphism 
\[\Phi_{\eta,P}:\Gamma_{AC}(\eta)\to \text{Im}(\Phi_{\eta,P}).\]
\end{remark}

\begin{coro}
Let $p\in [1,\infty]$ and $N$ be a smooth manifold modeled on a Banach space $E$ (resp. Fr\'echet space). If $\eta \in AC_{L^p}([a,b],N)$, then the vector space $\Gamma_{AC}(\eta)$ is a Banach space (resp. Fr\'echet space).
\end{coro}
\begin{proof}
    This follows from the fact that each vector space $AC_{L^p}([t_{i-1},t_i],E)$ is a Banach space (resp. Fr\'echet space).
\end{proof}

\begin{prop}\label{fsection}
Let $p\in [1,\infty]$. Let $M$ and $N$ be smooth manifolds modeled on sequentially complete locally convex spaces and $f:M\to N$ a $C^{2}$-map. If $\eta\in AC_{L^p}([a,b],M)$, then $Tf\circ \sigma \in \Gamma_{AC}(f\circ \eta)$ for each $\sigma \in \Gamma_{AC}(\eta)$. Moreover, the map
\[ \widetilde{f}:\Gamma_{AC}(\eta)\to  \Gamma_{AC}(f\circ \eta),\quad \sigma\mapsto Tf\circ \sigma,\]
is continuous linear.
\end{prop}
\begin{proof}
Let  $E_M$ and $E_N$ be the modeling spaces of $M$ and $N$ respectively.\\ 
By Lemma \ref{Acts0}, $f\circ \eta \in AC_{L^p}([a,b],N)$ and for each $\sigma\in \Gamma_{AC}(\eta)$, we have $Tf\circ \sigma \in AC_{L^p}([a,b],TN)$. Since $T_{\eta(t)}f \left(\sigma(t)\right) \in T_{f\circ \eta(t)} N$ for each $t\in [a,b]$, we have
\[ \pi_{TN}\circ (Tf\circ \sigma) = f\circ \eta. \]
Thus $\widetilde{f}(\sigma)\in \Gamma_{AC}(f\circ \eta)$ for each $\sigma\in \Gamma_{AC}(\eta)$. The linearity of $\widetilde{f}$ is clear.\\
We may choose a partition $P=\{t_0,...,t_n\}$ of $[a,b]$ such that there exist  families of charts $\{(\varphi_i,U_i):i\in \{1,...,n\}\}$ and $\{(\phi_i,V_i):i\in \{1,...,n\}\}$ that verify the definition of absolute continuity for $\eta$  and $f\circ\eta$ respectively, such that
\[ f(U_i) \subseteq V_i, \quad \text{for each }i\in\{1,...,n\}. \]
Let $i\in\{1,...,n\}$. For each $\sigma\in \Gamma_{AC}(\eta)$ we denote $\eta_i=\eta|_{[t_{i-1},t_i]}$ and $\sigma_i=\sigma|_{[t_{i-1},t_i]}$, and we define the maps
\[ F_i: AC_{L^p}([t_{i-1},t_i],E_M) \to AC_{L^p}([t_{i-1},t_i],E_N),\quad \tau \mapsto \left(d\phi_i \circ Tf \circ  T\varphi_i^{-1} \right)\circ (\varphi_i\circ \eta_i, \tau), \]
and 
\[ F: \prod_{i=1}^n AC_{L^p}([t_{i-1},t_i],E_M) \to \prod_{i=1}^n  AC_{L^p}([t_{i-1},t_i],E_N),\quad \left(\tau_i\right)_{i=1}^n \mapsto \Big( F_i(\tau_i) \Big)_{i=1}^n, \]
which are continuous by Lemma \ref{Acts}. We will show that $F\left( \text{Im}\left(\Phi_{\eta,P}\right)\right) \subseteq \text{Im}\left(\Phi_{f\circ\eta,P}\right)$.\\
Let $i\in\{1,..,n-1\}$. We denote
\begin{align*}
\tau_i:&=F_i(d\varphi_i\circ \sigma_i) \\
&= (d\phi_i \circ Tf \circ d\varphi_i^{-1})(\varphi_i\circ \eta_i, d\varphi_i\circ \sigma_i)\\
&= d\phi_i \circ Tf \circ \sigma_i.
\end{align*}
Then, since $\sigma_{i+1}(t_i)=\sigma_{i}(t_i)$ (by continuity of $\sigma)$, we have
\begin{align*}
d\phi_i \circ (T\phi_{i+1})^{-1} \Big(\phi_{i+1}\circ(f\circ\eta)(t_{i}),\tau_{i+1}(t_{i})\Big) &=
d\phi_i \circ (T\phi_{i+1})^{-1} \Big(\phi_{i+1}\circ(f\circ\eta)(t_{i}),d\phi_{i+1} \circ Tf \circ \sigma_{i+1}(t_{i})\Big) \\
&= d\phi_i\circ Tf \circ \sigma_{i+1}(t_{i}) \\
&= d\phi_i\circ Tf \circ \sigma_{i}(t_{i}) \\
&= \tau_i (t_i).
\end{align*}
Hence $F\circ \Phi_{\eta,P}(\sigma) \in \text{Im}(\Phi_{f\circ \eta,P})$ and consequently
\[ \widetilde{f}=\Phi_{f\circ\eta,P}^{-1}\circ F \circ \Phi_{\eta,P}. \]
Thus $\widetilde{f}$ is continuous.
\end{proof}

\begin{remark}\label{independencepartition}
The topology of $\Gamma_{AC}(\eta)$ does not depend on the partition or charts chosen. Indeed, let us consider topologies $\tau_1$ and $\tau_2$ defined by the partition $P_1$ and $P_2$ of $[a,b]$, respectively. Since the identity map $\text{id}_M:M\to M$ is smooth and $T\text{id}_M=\text{id}_{TM}$, by the previous proposition, the map
\[ \widetilde{\text{id}_M}:\left(\Gamma_{AC}(\eta),\tau_1\right) \to \left(\Gamma_{AC}(\eta),\tau_2\right), \quad \sigma \mapsto \sigma, \]
is continuous linear (hence smooth) regardless of the partition or charts chosen.
\end{remark}

\begin{remark}\label{re}
Let us consider the set $C([a,b],TN)$ endowed with the compact-open topology. For $\eta\in C([a,b],N)$ we consider the vector space
\[\Gamma_C (\eta)=\{\sigma\in C([a,b],TN): \pi_{TN}\circ \sigma=\eta \},\]
endowed with the subspace topology inherited from $C([a,b],TN)$. Since for each partition of $[a,b]$, the inclusions 
\[AC_{L^p}([t_{i-1},t_i],E)\to C([t_{i-1},t_i],E),\quad \gamma\mapsto \gamma,\] 
are continuous \cite[Lemma 3.2]{Nin1}, following the idea of the proof of Proposition \ref{prop}, the inclusion map $J_\Gamma:\Gamma_{AC} (\eta)\to \Gamma_C(\eta)$ is also continuous. This implies that for each open subset $V\subseteq TN$, the set
\[\mathcal{V}:=\{ \sigma \in \Gamma_{AC}(\eta): \sigma([a,b])\subseteq V \} \]
is open in $\Gamma_{AC}(\eta)$.
\end{remark}

\begin{prop}
Let $p\in [1,\infty]$. Let $N_1$ and $N_2$ be smooth manifolds modeled on sequentially complete locally convex spaces  and $\pr_i:N_1\times N_2\to N_i$ be the $i$-projection for $i\in \{1,2\}$. If $\eta_1\in AC_{L^p}([a,b],N_1)$ and $\eta_2\in AC_{L^p}([a,b],N_2)$, then the map
\[ \mathcal{P}:\Gamma_{AC}(\eta_1, \eta_2)\to \Gamma_{AC}(\eta_1)\times \Gamma_{AC}(\eta_2),\quad \sigma \mapsto (T\pr_1,T\pr_2)(\sigma) \]
is a linear homeomorphism.
\end{prop}
\begin{proof}
By Proposition \ref{fsection}, the map $\mathcal{P}$ is continuous and clearly linear. 
Let $P=\{t_0,..., t_n\}$ be a
partition of $[a,b]$ and let
$\{(\phi_{1,i},U_{1,i})\colon i\in \{1,\ldots,n\}\}$
and $\{(\phi_{2,i},U_{2,i})\colon i\in \{1,\ldots,n\}\}$
be families of charts of $N_1$ and $N_2$, respectively,
that verify the definition
of absolute continuity for $\eta_1$ and $\eta_2$,
respectively. Then $\eta:=(\eta_1,\eta_2)\colon [a,b]\to N_1\times N_2$
is $L^p$-absolutely continuous and it is clear that
the charts $\{(\phi_{1,i}\times \phi_{2,i},U_{1,i}\times U_{2,i})\colon
i\in \{1,\ldots, n\}\}$ satisfy the condition of absolute continuity for
$\eta$. For $j\in \{1,2\}$, consider the linear topological embedding
\[
\Phi_{\eta_j,P}\colon \Gamma_{\eta_j}\to\prod_{i=1}^n AC_{L^p}([t_{i-1},t_i],E_j),\;\;
\tau\mapsto (d\phi_{j,i}\circ \tau|_{[t_{i-1},t_i]})_{i=1}^n
\]
where $E_j$ is the modeling space of $N_j$. Identifying $T_{(x_1,x_2)}(N_1\times N_2)$ with $T_{x_1}N_1\times T_{x_2}N_2$, we define the map
\[
\Phi_{\eta,P}\colon \Gamma_{\eta}\to\prod_{i=1}^n AC_{L^p}([t_{i-1},t_i],E_1\times E_2),\;\;
\tau\mapsto ((d\phi_{1,i}\times d\phi_{2,i})\circ \tau|_{[t_{i-1},t_i]})_{i=1}^n
\]
which is a linear topological embedding. For $i\in \{1,\ldots, n\}$, let
\[
\alpha_i\colon  AC_{L^p}([t_{i-1}, t_i],E_1)\times  AC_{L^p}([t_{i-1},t_i],E_2)\to
 AC_{L^p}([t_{i-1}, t_i], E_1\times E_2)
\]
be the map taking a pair $(f_1,f_2)$ of functions to the function
$t\mapsto(f_1(t),f_2(t))$; we know that $\alpha_i$ is an isomorphism
of topological vector spaces. Then also
\[
\alpha\colon \left(\prod_{i=1}^n  AC_{L^p}([t_{i-1},t_i],E_1)\right)\times
\left(\prod_{i=1}^n  AC_{L^p}([t_{i-1},t_i],E_2)\right)\to
\prod_{i=1}^n  AC_{L^p}([t_{i-1},t_i],E_1\times E_2),
\]
\[
((f_i)_{i=1}^n,(g_i)_{i=1}^n)\mapsto(\alpha(f_i,g_i))_{i=1}^n
\]
is an isomorphism of topological vector spaces.
If $(f_i)_{i=1}^n$ is in the image of $\Phi_{\eta_1,P}$
and $(g_i)_{i=1}^n$ is in the image of $\Phi_{\eta_2,P}$,
then $\alpha((f_i)_{i=1}^n,(g_i)_{i=1}^n)$
is in the image of $\Phi_{\eta,P}$,
as the compatibility at the endpoints can be checked
by considering the components in $E_1$ and $E_2$.
We can therefore define a function
\[
\Theta:=\Phi_{\eta,P}^{-1}\circ
\alpha\circ (\Phi_{\eta_1,P}\times \Phi_{\eta_2,P})
\colon \Gamma_{AC}(\eta_1)\times\Gamma_{AC}(\eta_2)\to\Gamma_{AC}(\eta),
\]
which is continuous and linear. We readily check that
$\mathcal{P}(\Theta(\sigma,\tau))=(\sigma,\tau)$
for all $\sigma\in \Gamma_{AC}(\eta_1)$ and $\tau\in \Gamma_{AC}(\eta_2)$.
Hence $\mathcal{P}$ is surjective and thus bijective, with
$\mathcal{P}^{-1}=\Theta$ a continuous map.
\end{proof}

\begin{prop}\label{proprestric}
Let $p\in [1,\infty]$. Let $N$ be a smooth manifold modeled on a sequentially complete locally convex space $E$. If $\eta\in AC_{L^p}([a,b],N)$ and $P=\{t_0,...,t_n\}$ is a partition of $[a,b]$, then the map
\[ \rho: \Gamma_{AC}(\eta)\to \prod_{i=1}^n \Gamma_{AC}\left(\eta|_{[t_{i-1},t_i]}\right),\quad \sigma\mapsto \left(\sigma|_{[t_{i-1},t_i]}\right)_{i=1}^n \]
is a linear topological embedding with closed image.
\end{prop}
\begin{proof}
We see that $\rho$ is clearly linear. For $i\in \{1,...,n\}$, we denote $P_i=\{t_{i-1},t_i\}$. Since $\eta_i:=\eta|_{[t_{i-1},t_i]} \in AC_{L^p}([t_{i-1},t_i],N)$, we have
\[ \rho= \left(\prod_{i=1}^n \Phi_{\eta_i,P_i}^{-1}\right) \circ \Phi_{\eta,P}.\]
The image is given by the closed subspace
\[ \text{Im}(\rho)= \left\{ (\tau_i)_{i=1}^n \in \prod_{i=1}^n \Gamma_{AC}(\eta|_{[t_{i-1},t_i]}) : \tau_{i}(t_i)=\tau_{i+1}(t_i) \text{ for all }i\in\{1,...,n-1\} \right\}. \]
Thus $(\rho|^{\text{Im}(\rho)})^{-1}: \text{Im}(\rho) \to \Gamma_{AC}(\eta)$ is well defined and 
\[ \left(\rho|^{\text{Im}(\rho)}\right)^{-1} = \Phi_{\eta,P}^{-1} \circ \left( \prod_{i=1}^n \Phi_{\eta_i,P_i}\right). \]
Therefore, $\rho$ is a topological embedding with closed image.
\end{proof}

\begin{prop}\label{propnog}
Let $p\in [1,\infty]$. Let $N$ be a smooth manifold modeled on a sequentially complete locally convex space $E$ and let $g:[c,d]\to [a,b]$ be the map given by
\[ g(t)=a+\frac{t-c}{d-c}(b-a),\quad t\in[c,d].\]
Then $\eta\circ g\in AC_{L^p}([c,d],N)$ for each $\eta\in AC_{L^p}([a,b],N)$. Moreover, if $\eta\in AC_{L^p}([a,b],N)$, then the map
\[ L_g:\Gamma_{AC}(\eta) \to \Gamma_{AC} (\eta\circ g), \quad \sigma \mapsto \sigma\circ g\]
is continuous linear.
\end{prop}
\begin{proof}
Let $\eta \in AC_{L^p}([a,b],N)$. We first show that $\eta\circ g\in AC_{L^p}([c,d],N)$. Let $P=\{t_0,...,t_n\}$ be a partition of $[a,b]$ and $\{(\varphi_i,U_i): i\in \{1,...,n\}\}$ charts of $N$ that verify the definition of absolute continuity for $\eta$. Since $g$ is a strictly increasing function, we can define a partition $Q=\{s_0,...,s_n\}$ of $[c,d]$ such that $g(s_i)=t_i$, for each $i\in \{1,...,n\}$. Since
\[ \varphi_i\circ\left(\eta\circ g\right)|_{[s_{i-1},s_i]}=\left(\varphi_i\circ\eta|_{[t_{i-1},t_i]}\right) \circ g|_{[s_{i-1},s_i]} \]
we have $\eta \circ g \in AC_{L^p}([c,d], N)$ by Lemma \ref{lemmanog}. By the same argument, we have that $\sigma\circ g \in AC_{L^p}([c,d],TN)$ for each $\sigma\in\Gamma_{AC}(\eta)$ and
\[ \pi_{TN}\circ (\sigma\circ g) = \eta\circ g.\] 
Hence $L_g(\sigma) \in \Gamma_{AC}(\eta\circ g)$. It remains to show that $L_g$ is continuous. For each $i\in\{1,...,n\}$ we define the maps
\[ G_i : AC_{L^p}([t_{i-1},t_i], E)\to AC_{L^p}([s_{i-1},s_i], E) ,\quad  \tau\mapsto \tau \circ g|_{[s_{i-1},s_i]}\]
which are linear and continuous by Lemma \ref{lemmanog}. Using the topological embeddings $\Phi_{\eta,P}$ and $\Phi_{\eta\circ g,Q}$ (as in Proposition \ref{prop}), if $(\tau_i)_{i=1}^n \in \text{Im}(\Phi_{\eta,P})$ we have
\begin{align*}
    \tau_i \circ g (s_i) &= \tau_i (t_i) \\
    &= d\varphi_i \circ (T\varphi_{i+1})^{-1} \Big(\varphi_{i+1}\circ\eta(t_{i}),\tau_{i+1}(t_{i})\Big) \\
    &= d\varphi_i \circ (T\varphi_{i+1})^{-1} \Big(\varphi_{i+1}\circ\eta\circ g(s_{i}),\tau_{i+1}\circ g(s_{i})\Big)
\end{align*}
for each $i\in \{1,...,n-1\}$. If $G=(G_1\times ...\times G_n)$, then $G$ is continuous linear and 
\[ \text{Im}\left(G \circ \Phi_{\eta,P}\right)\subseteq \text{Im}(\Phi_{\eta\circ g,Q}).\] 
Therefore
\[ L_g = \Phi_{\eta\circ g,Q}^{-1}\circ G \circ \Phi_{\eta,P}. \]
is continuous linear.
\end{proof}

\begin{prop}\label{gammaeval}
Let $p\in [1,\infty]$ and $N$ be a smooth manifold modeled on a sequentially complete locally convex space $E$. If $\eta\in AC_{L^p}([a,b],N)$, then the evaluation map
\[ \epsilon:\Gamma_{AC}(\eta)\times [a,b]\to TN,\quad(\sigma,t)\mapsto \sigma(t) \]
is continuous and linear in the first argument. Moreover, for each $t \in [a,b]$ fixed, the point evaluation map
\[ \epsilon_{t}:\Gamma_{AC}(\eta) \to TN,\quad \sigma \mapsto \sigma(t) \]
is smooth.
\end{prop}
\begin{proof}
The evaluation map
\[  \tilde{\epsilon}: \Gamma_C(\eta)\times [a,b]\to TN, \quad (\sigma, t)\mapsto \sigma(t) \]
is continuous and the evaluation map $ \tilde{\epsilon}_t : \Gamma_C (\eta)\to TN$, $\sigma\mapsto \sigma(t)$ is smooth for each $t\in [a,b]$ (see \cite{AGS}). Then $\epsilon = \tilde{\epsilon} \circ (J_\Gamma\times\text{Id}_\R)$ and $\epsilon_t = \tilde{\epsilon}_t \circ J_\Gamma$, where $J_\Gamma:\Gamma_{AC}(\eta) \to \Gamma_C (\eta)$ is the inclusion map, which is continuous linear by Remark \ref{re}.
\end{proof}

\section{Manifolds of absolutely continuous functions}
\begin{definition}\label{localaddition}
Let $N$ be a smooth manifold and $\pi_{TN}:TN\to N$ its tangent bundle. A \textit{local addition} is a smooth map ${\Sigma:\Omega\to N}$ defined on an open neighborhood $\Omega \subseteq TN$ of the zero-section ${0_N := \{ 0_p \in T_p N : p\in N\}}$ such that
\begin{itemize}
\item[a)] $\Sigma(0_p)=p$ for all $p\in N$.
\item[b)] The image $\Omega^\prime:= \big( \pi_{TN},\Sigma\big) (\Omega)$ is open in $N\times N$ and the map 
\begin{equation*}
    \theta_N:\Omega\to \Omega^\prime,\quad v\mapsto \big( \pi_{TN}(v),\Sigma(v) \big)
\end{equation*}
is a $C^\infty$-diffeomorphism.
\end{itemize}
Moreover, if $T_{0_p}(\Sigma |_{T_p N})=id_{T_p N}$ for all $p\in N$, we say that the local addition $\Sigma$ is \textit{normalized}. We denote the local addition as the pair  $(\Omega, \Sigma)$.\\
If $\theta_N:\Omega\to \Omega^\prime$ is a diffeomorphism of $\mathbb{K}$-analytic manifolds, we call ${\Sigma:\Omega\to N}$ a $\mathbb{K}$-analytic local addition.
\end{definition}

\begin{remark} \label{locallemma}
Let $N$ be a smooth manifold with local addition $(\Omega,\Sigma)$. We denote the tangent bundle of $TN$ as $\pi_{T(TN)}:T(TN)\to TN$. If $\kappa:T(TN)\to T(TN)$ is the canonical flip, then $T\Sigma\circ \kappa: T(T\Omega)\to TN$ it is a local addition on $TN$ \cite[Lemma A.11]{AGS}. Moreover, each manifold which admits a local addition also admits a normalized local addition \cite[Lemma A.14]{AGS}.\\ 
From now on, we assume that each local addition is normalized.
\end{remark}

\begin{remark}\label{recharts}
Let $N$ be a smooth manifold modeled on a sequentially complete locally convex space $E$, $p\in [1,\infty]$ and $\eta \in AC_{L^p}([a,b],N)$. We define the sets
\begin{equation*}
    \mathcal{V}_\eta := \{ \sigma \in \Gamma_{AC}(\eta) : \sigma([a,b])\subseteq \Omega \}
\end{equation*}
which is open in $\Gamma_{AC}(\eta)$ by Remark \ref{re} and 
\begin{equation*}
    \mathcal{U}_\eta := \{ \gamma \in AC_{L^p}([a,b],N) : (\eta,\gamma)([a,b]) \subseteq \Omega' \}.
\end{equation*}
Lemma \ref{pre1} enables us to define the map
\begin{equation*}
    \Psi_\eta: \mathcal{V}_\eta  \to \mathcal{U}_\eta,\quad \sigma  \mapsto \Sigma\circ \sigma.
\end{equation*}
Sith inverse given by
\begin{equation*}
    \Psi_\eta^{-1} :\mathcal{U}_\eta \to \mathcal{V}_\eta,\quad \gamma \mapsto \theta_N^{-1}\circ (\eta, \gamma).
\end{equation*}
\end{remark}

\noindent
The following lemma is just an application of \cite[Lemma 10.1]{Ber} to our particular case.
\begin{lemma}\label{Fclosed}
Let $E$ and $F$ sequentially complete locally convex spaces, $U\subseteq E$ open and $f:U\to F$ a map. If $F_0\subseteq F$ is a closed vector subspace and $f(U)\subseteq F_0$, then $f:U\to F$ is smooth if and only if $f|^{F_0}:U\to F_0$ is smooth.
\end{lemma}

\begin{theorem} \label{teomanifold}
Let $p\in [1,\infty]$. For each smooth manifold~$N$ 
modeled on a sequentially complete
locally convex space $E$ which admits a local addition, the set $AC_{L^p}([a,b],N)$ 
admits a smooth manifold structure
such that the sets ${\mathcal U}_\eta$ are open in $AC_{L^p}([a,b],N)$
for all $\eta\in AC_{L^p}([a,b],N)$
and $\Psi_\eta\colon {\mathcal V}_\eta\to {\mathcal U}_\eta$
is a $C^\infty$-diffeomorphism.
\end{theorem}

\begin{proof}
We endow $AC_{L^p}([a,b],N)$ with the final topology with respect to the family of maps 
\[ \Psi_\eta\colon {\mathcal V}_\eta\to {\mathcal U}_\eta,\quad \text{for each }\eta\in AC_{L^p}([a,b],N).\]
Analogously, for each $\eta\in C([a,b],N)$, we define the maps $\Psi_\eta^C\colon {\mathcal V}_\eta^C\to {\mathcal U}_\eta^C$ on the space of continuous functions $C([a,b],N)$ as in Remark \ref{recharts}. Then the final topology on $C([a,b],N)$ with respect to these maps coincides with its compact-open topology and the inclusion map
\[ J: AC_{L^p}([a,b],N) \to C([a,b],N), \quad \gamma \mapsto \gamma\]
is continuous. Moreover, for each $\eta \in AC_{L^p}([a,b],N)$ the set 
\[ \mathcal{U}_{\eta}^C := \{ \gamma \in C([a,b],N) : (\eta,\gamma)([a,b]) \subseteq \Omega' \}\]
is open in $C([a,b],N)$, whence
\[ \mathcal{U}_\eta = \mathcal{U}_\eta^C \cap AC_{L^p}([a,b],N) \]
is open in $AC_{L^p}([a,b],N)$. Consequently, $AC_{L^p}([a,b],N)$ is Hausdorff.\\
The goal is to make the family $\{ (\Psi_\eta^{-1},\mathcal{U}_\eta) : \eta \in AC_{L^p}([a,b],N) \}$ an atlas for $AC_{L^p}([a,b],N)$ for a smooth manifold structure on $AC_{L^p}([a,b],N)$. We need to show that the charts are compatible, i.e., that the smoothness of each transition map 
\begin{equation*}\label{compatiblecharts}
\Lambda_{\xi,\eta} := \Psi_\xi^{-1} \circ \Psi_\eta : \Psi_\eta^{-1}(\mathcal{U}_\eta \cap \mathcal{U}_\xi)\subseteq \Gamma_{AC}(\eta) \to \Gamma_{AC}(\xi) 
    , \quad \sigma \mapsto \theta_N^{-1} \circ (\xi, \Sigma \circ \sigma ),
\end{equation*}
for each $\eta, \xi \in AC_{L^p}([a,b],N)$ such that the open set 
\[ \Psi_\eta^{-1}(\mathcal{U}_\eta \cap \mathcal{U}_\xi) = \left(\Psi_\eta^{C}\right)^{-1}(\mathcal{U}_\eta^C \cap \mathcal{U}_\xi^C) \cap \Gamma_{AC}(\eta)\]
is not empty. \\
Let $R=\{t_0,...,t_n\}$ be a partition of $[a,b]$, let $\{(\varphi_i,U_{\varphi_i}): i\in\{1,...,n\}\}$ and $\{(\phi_i,U_{\phi_i}): i\in\{1,...,n\}\}$ be charts that verify the definition of absolute continuity for $\eta$ and $\xi$, respectively. For each $\sigma\in AC_{L^p}([a,b],TN)$, we denote $\sigma_i := \sigma|_{[t_{i-1},t_i]}$ and 
\[ \funcion(\sigma_i):= \funcion(\sigma)\big|_{[t_{i-1},t_i]}.  \]
By Lemma \ref{Fclosed}, the smoothness of $\funcion$ is equivalent to the smoothness of the composition
\begin{align*}
     \Phi_{\xi,R} \circ \funcion: \Psi_\eta^{-1}(\mathcal{U}_\eta \cap \mathcal{U}_\xi) &\to \text{Im}(\Phi_{\xi,R})\subset \prod_{i=1}^n AC_{L^p}([t_{i-1},t_i],E) \\
     \sigma &\mapsto \left( d\phi_i \circ \funcion (\sigma_i ) \right)_{i=1}^n.
\end{align*}
For each $i\in\{1,...,n\}$, we denote $\eta_i:=\eta|_{[t_{i-1},t_i]}$ and $\xi_i:=\xi|_{[t_{i-1},t_i]}$ and we have
\begin{align*}
    d{\phi}_i \circ \funcion|_{[t_{i-1},t_i]} (\sigma_i ) &= d\phi_i \circ \theta_N^{-1} \circ (\xi_i, \Sigma \circ \sigma_i ) \\ 
    &= d\phi_i \circ \theta_N^{-1} \circ (\phi_i^{-1} \left( \phi_i\circ\xi_i \right), \Sigma\circ T\varphi_i^{-1} \circ T\varphi_i \circ \sigma_i )\\
    &= d\phi_i \circ \theta_N^{-1} \circ (\phi_i^{-1} \left( \phi_i\circ\xi_i \right), \Sigma\circ T\varphi_i^{-1} \left( \varphi_i \circ \eta_i , d\varphi_i \circ \sigma_i \right) ).
 \end{align*}
Because all of the functions involved are continuous and have an open domain, also the composition
\begin{equation}
H_i(x,y,z) := d\phi_i \circ \theta_N^{-1} \circ (\phi_i^{-1} \left( x \right), \Sigma\circ T\varphi_i^{-1} ( y , z ) ) 
\end{equation}  
has an open domain $\mathcal{O}_i$. Hence the map $H_i: \mathcal{O}_i \to E$ is smooth. By Lemma \ref{Acts}, the map
\[  AC_{L^p}([t_{i-1},t_i],H_i) :AC_{L^p}([t_{i-1},t_i],  \mathcal{O}_i) \to AC_{L^p}([t_{i-1},t_i],E) \quad \alpha \mapsto H_i \circ \alpha \]
is smooth. Using the identification of products of $AC_{L^p}$ spaces (Remark \ref{times}), if we fix the functions
$\phi_i\circ\xi_i$ and $\varphi_i \circ \eta_i$, we have the continuous linear map
\[ AC_{L^p}([t_{i-1},t_i],E)\to AC_{L^p}([t_{i-1},t_i],E\times E\times E),\quad \tau\mapsto (\phi_i\circ\xi_i, \varphi_i \circ \eta_i,\tau).\]
We write $W_i$ for the preimage of $AC_{L^p}([t_{i-1},t_i],\mathcal{O}_i)$ under this map. Then the map
\[ \Theta_i :W_i \to AC_{L^p}([t_{i-1},t_i],E),\quad \tau \mapsto H_i\circ (\phi_i\circ\xi_i, \varphi_i \circ \eta_i ,  \tau ) \]
is also smooth. Since the maps $h_i:\Gamma_{AC}(\eta)\to AC_{L^p}([t_{i-1},t_i],E), \sigma \mapsto d\varphi_i \circ \sigma_i$ are continuous by definition of the topology, rewriting we have
\[ \Phi_{\xi,R} \circ \funcion (\sigma) = \left( \Theta_i \circ h_i (\sigma) \right)_{i=1}^n \]
for each $\sigma \in \Psi_\eta^{-1}(\mathcal{U}_\eta \cap \mathcal{U}_\xi)$. Since each $\Theta_i$ is smooth and $h_i$ is continuous linear, the map $\Phi_{\xi,R} \circ \funcion$ is smooth an by Lemma \ref{Fclosed}, so is $\funcion$.
\end{proof}

\noindent
Proceeding in the same way, using the fact that compositions of $\mathbb{K}$-analytic maps are $\mathbb{K}$-analytic and the analytic version of Lemma \ref{Fclosed} (see \cite{GL1}), we obtain the following analytic analogue of Theorem \ref{teomanifold}.
\begin{coro} \label{manalytic}
For each $\mathbb{K}$-analytic manifold~$N$
modeled on a sequentially complete locally convex space $E$ which admits a $\mathbb{K}$-analytic local addition and $p\in [1,\infty]$, the set $AC_{L^p}([a,b],N)$ 
admits a $\mathbb{K}$-analytic manifold structure 
such that the sets ${\mathcal U}_\eta$ are open in $AC_{L^p}([a,b],N)$
for all $\eta\in AC_{L^p}([a,b],N)$
and $\Psi_\eta\colon {\mathcal V}_\eta\to {\mathcal U}_\eta$
is a $\mathbb{K}$-analytic diffeomorphism.
\end{coro}

\begin{prop}
Let $1\leq q\leq p \leq \infty$ and $N$ be a smooth manifold modeled on a sequentially complete locally convex space $E$ which admits a local addition. Then 
\[ AC_{L^\infty}([a,b],N)\subseteq AC_{L^p}([a,b],N)\subseteq AC_{L^q}([a,b],N)\subseteq AC_{L^1}([a,b],N) \]
with smooth inclusion maps.
\end{prop}
\begin{proof}
Let $\eta \in AC_{L^p}([a,b],N)$. Let $\{t_0,...,t_n\}$ be a partition of $[a,b]$ and $\{(\varphi_i,U_i): i\in \{1,...,n\}\}$ be charts of $N$ that verify the definition of absolute continuity for $\eta$. \\
By \cite[Remark 3.2]{Nin1}, we know that each inclusion map $AC_{L^p}([t_{i-1},t_i],E)\to AC_{L^q}([t_{i-1},t_i],E)$ is continuous linear, hence $AC_{L^p}([a,b],N)\subseteq AC_{L^q}([a,b],N)$. Since the map
\[ \prod_{i=1}^n AC_{L^p}([t_{i-1},t_i],E)\to \prod_{i=1}^n AC_{L^q}([t_{i-1},t_i],E),\quad (\gamma_i)_{i=1}^n \mapsto (\gamma_i)_{i=1}^n. \]
is continuous linear, by Proposition \ref{prop}, the inclusion $AC_{L^p}([a,b],N)\to AC_{L^q}([a,b],N)$ is smooth. Applying the same argument to each consecutive inclusion in the chain gives the result.
\end{proof}

\begin{theorem}\label{propfon}
Let $p\in [1,\infty]$ and $k\in \N_0\cup\{\infty\}$. Let $M$ and $N$ be smooth manifolds modeled on sequentially complete locally convex spaces which admit a local addition. If $f:M\to N$ is a $C^{k+2}$-map, then the map
\[ AC_{L^p}([a,b],f):AC_{L^p}([a,b],M)\to AC_{L^p}([a,b],N),\quad \eta \mapsto f\circ \eta \]
is $C^k$.
\end{theorem}
\begin{proof}
The map is well defined by Lemma \ref{pre1}. Let $E_M$ and $E_N$ be the modeling spaces and $(\Omega_M,\Sigma_M)$ and $(\Omega_N,\Sigma_N)$ be local additions on $M$ and $N$ respectively.\\
Let $\eta \in AC_{L^p}([a,b],M)$. We consider charts $(\mathcal{U}_\eta, \Psi_\eta^{-1})$ and $(\mathcal{U}_{f\circ\eta}, \Psi_{f\circ\eta}^{-1})$ around $\eta \in AC_{L^p}([a,b],M)$ and $f\circ \eta \in AC_{L^p}([a,b],N)$, respectively. Following the notation of Theorem \ref{teomanifold}, we see that
\[ \Psi_\eta^{-1}\left( \mathcal{U}_\eta \cap AC_{L^p}([a,b],f)^{-1}(\mathcal{U}_{f\circ\eta})\right) = \Gamma_{AC}(\eta) \cap \left(\Psi_\eta^C\right)^{-1}\left( \mathcal{U}_\eta^C \cap AC_{L^p}([a,b],f)^{-1}(\mathcal{U}_{f\circ\eta}^C)\right) \]
is open in $\Gamma_{AC}(\eta)$. If $\theta_N=(\pi_{TN},\Sigma_N)$, then we define
\[  F(\sigma) := \Psi_{f\circ\eta}^{-1} \circ AC_{L^p}([a,b],f) \circ \Psi_\eta (\sigma) = \theta_N^{-1} \circ \Big((f\circ \eta) , (f\circ \Sigma_M) \circ \sigma \Big)\in \Gamma_{AC}(f\circ \eta), \]
for all $\sigma \in \Psi_\eta^{-1}\left( \mathcal{U}_\eta \cap AC_{L^p}([a,b],f)^{-1}(\mathcal{U}_{f\circ\eta})\right)$.\\ 
Proceeding as in the proof of Theorem \ref{teomanifold}, choosing the corresponding partition $P=\{t_0,...,t_n\}$ of $[a,b]$ and the families of charts $\{(\varphi_i,U_{\varphi_i}): i\in \{1,...,n\}\}$ and $\{({\phi}_i,U_{\phi_i}): i\in \{1,...,n\}\}$ that verify the definition of absolute continuity for $\eta$ and $f\circ\eta$ respectively.\\ 
We denote $\sigma_i:=\sigma|_{[t_{i-1},t_i]}$, $\eta_i:=\eta|_{[t_{i-1},t_i]}$ and
\[ F(\sigma)|_{[t_{i-1},t_i]} := F(\sigma_i). \]
We will study the $C^k$-regularity of the map
\[
   \Phi_{f\circ \eta}\circ F: \Psi_\eta^{-1}\left( \mathcal{U}_\eta \cap AC_{L^p}([a,b],f)^{-1}(\mathcal{U}_{f\circ\eta})\right) \to \text{Im}(\Phi_{f\circ \eta}),\quad \sigma \mapsto \left( d\phi_i \circ  F (\sigma_i ) \right)_{i=1}^n,  \]
where $\Phi_{f\circ \eta}$ is the linear topological embedding as in Proposition \ref{prop}. For each $i\in\{1,...,n\}$ and $\sigma \in \Psi_\eta^{-1}\left( \mathcal{U}_\eta \cap AC_{L^p}([a,b],f)^{-1}(\mathcal{U}_{f\circ\eta})\right)$ we have
\begin{align*}
d\phi_i\circ F (\sigma)|_{[t_{i-1},t_i]} &= d\phi_i\circ\theta_N^{-1} \circ \Big(f\circ \eta_i , f\circ \Sigma_M \circ \sigma_i \Big) \\
&= d\phi_i\circ\theta_N^{-1} \circ \Big(\phi_i^{-1}\circ \phi_i \circ f\circ \eta_i , f\circ \Sigma_M \circ T\varphi_i^{-1} \left( \varphi_i \circ \eta_i , d\varphi_i \circ \sigma_i \right) \Big).
\end{align*}
Because all of the functions involved are continuous and have an open domain, also the composition
\begin{equation}
H_i(x,y,z) := d\phi_i\circ\theta_N^{-1} \circ \Big(\phi_i^{-1} (x) , f\circ \Sigma_M \circ T\varphi_i^{-1} (y,z)\Big) 
\end{equation}  
has an open domain $\mathcal{O}_i$ in $E_N\times E_M\times E_M$. Hence the map $H_i: \mathcal{O}_i \to E_N$ is a $C^{k+2}$-map. For the absolutely continuous functions $\phi_i \circ f\circ \eta_i$ and $\varphi_i\circ\eta_i$, we consider the continuous linear map
\[ AC_{L^p}([t_{i-1},t_i],E_M)\to AC_{L^p}([t_{i-1},t_i],E_N\times E_M\times E_M),\quad \tau \mapsto (\phi_i \circ f\circ \eta_i, \varphi_i\circ\eta_i,\tau)\]
and we write $W_i$ for the preimage of $AC_{L^p}([t_{i-1},t_i],\mathcal{O}_i)$ under this map. Then the map
\[ \Theta_i :W_i\to AC_{L^p}([t_{i-1},t_i],E_N) \quad
\tau \mapsto H_i\circ  (\phi_i \circ f\circ \eta_i, \varphi_i\circ\eta_i, \tau ),\]
is $C^k$ by Lemma \ref{Acts}. Since the maps $h_i:\Gamma_{AC}(\eta) \to AC_{L^p}([t_{i-1},t_i],E_M)$, $\sigma\mapsto d\varphi_i\circ \sigma_i$ are continuous linear by definition, rewriting we have
\[ \Phi_{f\circ \eta}\circ F (\sigma) = \left( \Theta_i \circ h_i(\sigma) \right)_{i=1}^n \]
for each $\sigma \in \Psi_\eta^{-1}\left( \mathcal{U}_\eta \cap AC_{L^p}([a,b],f)^{-1}(\mathcal{U}_{f\circ\eta})\right)$. Since each $\Theta_i$ is $C^k$ and $h_i$ is continuous linear, by Lemma \ref{Fclosed} the map $F$ is $C^k$, and therefore, $AC_{L^p}([a,b],f)$ is $C^k$.
\end{proof}

\noindent
Proceeding in the same way, using the analytic version of Lemma \ref{Fclosed} (see \cite{GL1}), we obtain the following analytic analogue of Theorem \ref{propfon}.

\begin{coro}\label{fanalytic}
Let $p\in [1,\infty]$. Let $M$ and $N$ be smooth manifolds modeled on sequentially complete locally convex spaces which admit a local addition. If $f:M\to N$ be a $\mathbb{K}$-analytic map, then the map
\[ AC_{L^p}([a,b],f):AC_{L^p}([a,b],M)\to AC_{L^p}([a,b],N),\quad \eta \mapsto f\circ \eta \]
is $\mathbb{K}$-analytic.
\end{coro}

\begin{remark}\label{sametopology}
The manifold structures for $AC_{L^p}([a,b],N)$ given by different local additions coincide. Indeed, let us consider the topologies $\tau_1$ and $\tau_2$ given by two different local addition of $N$. Since the identity map $\text{id}_N:N\to N$ is smooth, the map
\[ AC_{L^p}([a,b],\text{id}_N):(AC_{L^p}([a,b],N),\tau_1) \to (AC_{L^p}([a,b],N),\tau_2), \quad \eta \mapsto \text{id}_N \circ\eta \]
is smooth.
\end{remark}

\begin{remark}\label{remark} The inclusion map 
\[ J:AC_{L^p}([a,b],N) \to C([a,b],N),\quad \eta\mapsto \eta\] 
is smooth. Indeed, let $(\mathcal{U}_\eta, \Psi_\eta^{-1})$ and $(\mathcal{U}_{\eta}^C, \left(\Psi_{\eta}^C\right)^{-1})$ be charts around $\eta\in AC_{L^p}([a,b],N)$ and $\eta\in C([a,b],N)$, respectively. Then 
\[ (\Psi_{\eta}^C)^{-1} \circ J\circ \Psi_\eta: \Psi_\eta^{-1}\left(\mathcal{U}_\eta\cap J^{-1}(\mathcal{U}_{\eta}^C)\right)\subseteq \Gamma_{AC}(\eta)\to \Gamma_C (\eta),\quad \sigma\mapsto \sigma\] 
is a restriction of the inclusion map $\Gamma_{AC}(\eta) \to \Gamma_C (\eta)$, which is continuous linear. Therefore $J$ is smooth. Moreover, if $U\subseteq N$ is an open subset, then the manifold structure induced by $AC_{L^p}([a,b],N)$ on the open subset 
\[ AC_{L^p}([a,b],U):=\{ \eta\in AC_{L^p}([a,b],N) : \eta([a,b])\subseteq U \} \]
coincides with the manifold structure on $AC_{L^p}([a,b],U)$ given by the local addition restricted to $U$.
\end{remark} 

\begin{prop}
Let $p\in [1,\infty]$. Let $N_1$ and $N_2$ be smooth manifolds modeled on sequentially complete locally convex spaces which admit a local addition. If $\pr_i:N_1\times N_2\to N_i$ denotes the i-th projection for $i\in \{1,2\}$, then the map
\[ \mathcal{P}:  AC_{L^p}([a,b],N_1 \times N_2)\to AC_{L^p}([a,b],N_1)\times AC_{L^p}([a,b],N_2),\quad \eta \mapsto (\pr_1,\pr_2)\circ \eta \]
is a diffeomorphism.
\end{prop}
\begin{proof}
By Remark \ref{sametopology}, if $(\Omega_1,\Sigma_1)$ and $(\Omega_2,\Sigma_2)$ are the local additions on $N_1$ and $N_2$ respectively, then we can assume that the local addition on $N_1\times N_2$ is 
\[ \Sigma:=\Sigma_1\times\Sigma_2: \Omega_1\times \Omega_2 \to N_1\times N_2\]
where $\Omega_1\times \Omega_2 \subseteq TN_1 \times TN_2 \cong T(N_1\times N_2)$.
The map $\mathcal{P}$ is smooth as a consequence of the smoothness of the maps
\[AC_{L^p}([a,b],\pr_i):AC_{L^p}([a,b],N_1 \times N_2)\to AC_{L^p}([a,b],N_i),\]
for each ${i\in\{1,2\}}$. Let $(\mathcal{U}_{\eta_1}\times \mathcal{U}_{\eta_2}, \Psi_{\eta_1}^{-1}\times \Psi_{\eta_2}^{-1})$ and $(\mathcal{U}_\eta, \Psi_\eta^{-1})$ be charts in  $(\eta_1,\eta_2) \in AC_{L^p}([a,b],N_1)\times AC_{L^p}([a,b],N_2)$ and $\mathcal{P}^{-1}(\eta_1,\eta_2)=\eta \in AC_{L^p}([a,b],N_1\times N_2)$  respectively. Since the map
\[ \mathcal{Q}: \Gamma_{AC}(\eta)\to \Gamma_{AC}(\eta_1)\times \Gamma_{AC}(\eta_2),\quad \tau \mapsto (\text{q}_1,\text{q}_2)\circ \tau \]
is an isomorphism of topological vector spaces, where for $i\in\{1,2\}$, the map  
\[ q_i:\Gamma_{AC}(\eta)\to \Gamma_{AC}(\eta_i),\quad \sigma \mapsto \text{pr}_i\circ \sigma\] 
is the $i$-projection, we have
\begin{align*}
  \Psi_\eta^{-1} \circ \mathcal{P}^{-1}\circ  (\Psi_{\eta_1}\times \Psi_{\eta_2}) (\sigma_1,\sigma_2) 
  &= (\pi_{T(N_1\times N_2)}, \Sigma)^{-1}\circ \left( \eta, \mathcal{P}^{-1} \circ (\Sigma_1\times \Sigma_2)(\sigma_1,\sigma_2)\right)\\
  &= (\pi_{T(N_1\times N_2)}, \Sigma)^{-1}\circ \left( \eta, \Sigma\circ \mathcal{Q}^{-1}(\sigma_1,\sigma_2) \right)\\
  &= \mathcal{Q}^{-1}(\sigma_1,\sigma_2)
\end{align*}
for all  $(\sigma_1,\sigma_2) \in (\Psi_{\eta_1}^{-1}\times \Psi_{\eta_2}^{-1}) \left( \mathcal{U}_{\eta_1}\times \mathcal{U}_{\eta_2} \cap \mathcal{P}(\mathcal{U}_{\eta})\right)$. Hence $\mathcal{P}^{-1}$ is smooth.
\end{proof}

\begin{prop}\label{GRembedding}
Let $p\in [1,\infty]$ and $N$ be a smooth manifold modeled on a sequentially complete locally convex space $E$ which admits a local addition. If $P=\{t_0,...,t_n\}$ is a partition of $[a,b]$, then the map
\[ T:AC_{L^p}([a,b],N) \to \prod_{i=1}^n AC_{L^p}([t_{i-1},t_i],N),\quad \eta\mapsto \left( \eta|_{[t_{i-1},t_i]}\right)_{i=1}^n \]
is smooth. Moreover, it is a smooth diffeomorphism onto a submanifold of $\displaystyle\prod_{i=1}^n AC_{L^p}([t_{i-1},t_i],N)$.
\end{prop}
\begin{proof}
It is clear that the map $T$ is well defined and injective. Let $\text{Im}(T)$ be the image of the map $T$. Then
\[ \text{Im}(T)= \{ (\gamma_i)_{i=1}^n \in \prod_{i=1}^n AC_{L^p}([t_{i-1},t_i],N) : \gamma_{i}(t_i)=\gamma_{i+1}(t_i) \text{ for all }i\in\{1,...,n-1\} \}. \]
Let $\gamma:=(\gamma_i)_{i=1}^n \in \prod_{i=1}^n AC_{L^p}([t_{i-1},t_i],N)$. For each $i\in \{1,...,n\}$ let $\Psi_{\gamma_i}^{-1}:\mathcal{U}_i \to  \mathcal{V}_i$ be charts around $\gamma_i$. Then the map
\[ \Psi_\gamma^{-1}:=\prod_{i=1}^n \Psi_{\gamma_i}^{-1} : \prod_{i=1}^n \mathcal{U}_i \to \prod_{i=1}^n \mathcal{V}_i \]
is a chart around $\gamma$. Let $\eta\in AC_{L^p}([a,b],N)$ and $\tilde{\eta}=T(\eta)$, then $\Psi_{\tilde{\eta}}^{-1}\circ T \circ \Psi_\eta$ is just restriction of the product of the restrictions of the smooth maps
\[ \Gamma_{AC}(\eta) \to \prod_{i=1}^n \Gamma_{AC}(\eta_i) ,\quad \sigma\mapsto \left( \sigma|_{[t_{i-1},t_i]}\right)_{i=1}^n \]
thus $T$ is smooth. It remains to show that $\text{Im}(T)$ is a submanifold.\\ 
Let $\gamma =(\gamma_i)_{i=1}^n \in \text{Im}(T)$ with charts as before, then for each  $i\in \{1,...,n-1\}$ and\\ $
\xi=\left( \xi_j\right)_{j=1}^n \in \text{Im}(T)\cap \prod_{j=1}^n \mathcal{U}_j$ we have
\begin{align*}
\Psi_{\gamma_i}^{-1}(\xi_i)(t_i) &= \theta_N^{-1} \circ (\gamma_i,\xi_i)(t_i) \\
&= \theta_N^{-1} \circ (\gamma_i(t_i),\xi_i(t_i))\\
&= \theta_N^{-1} \circ (\gamma_{i+1}(t_i),\xi_{i+1}(t_i))\\
&= \Psi_{\gamma_{i+1}}^{-1}(\xi_{i+1})(t_i).
\end{align*}
This implies that if $K$ denotes the vector space
\[ K:=\{ \left( \sigma_i \right)_{i=1}^n \in \prod_{i=1}^n \Gamma_{AC}(\gamma_i) : \sigma_{i}(t_i)=\sigma_{i+1}(t_i) \text{ for all }i\in\{1,...,n-1\}\}. \]
Then $\Psi_{\gamma}^{-1}|_{\mathcal{U}_\gamma \cap \text{Im}(T)}: \text{Im}(T) \cap \prod_{i=1}^n \mathcal{U}_i  \to K\cap \prod_{i=1}^n \mathcal{V}_i$ is a chart of $\text{Im}(T)$, making $\text{Im}(T)$ a smooth submanifold and the map $\tilde{T}:AC_{L^p}([a,b],N)\to \text{Im}(T)$, $\eta\mapsto T(\eta)$ a diffeomorphism.
\end{proof}

\begin{remark}
Let $N$ be a $\mathbb{K}$-analytic manifold modeled on a sequentially complete locally convex space which admits a $\mathbb{K}$-analytic local addition and $p\in [1,\infty]$. Since every continuous linear operator is analytic, the isomorphism
\[ \Gamma_{AC}(\eta) \to K,\quad \sigma\mapsto \left( \sigma|_{[t_{i-1},t_i]}\right)_{i=1}^n \]
is $\mathbb{K}$-analytic, which implies that $T$  in Proposition \ref{GRembedding} is a $\mathbb{K}$-analytic diffeomorphism onto the submanifold $\text{Im}(T)$.
\end{remark}

\begin{prop}
Let $p\in [1,\infty]$. Let $N$  be a smooth manifold modeled on a sequentially complete locally convex space $E$ which admits a local addition and let $g:[c,d]\to [a,b]$ be the map given by 
\[ g(t)=a+\frac{t-c}{d-c}(b-a),\quad t\in [c,d].\]
Then the map
\[ AC_{L^p}(g,N):AC_{L^p}([a,b],N)\to AC_{L^p}([c,d],N), \quad \eta \mapsto \eta\circ g\]
is smooth.
\end{prop}
\begin{proof}
By Proposition \ref{propnog} we know that the map is well defined. Let $(\mathcal{U}_\eta, \Psi_\eta^{-1})$ and $(\mathcal{U}_{\eta\circ g}, \Psi_{\eta\circ g}^{-1})$ be charts around $\eta \in AC_{L^p}([a,b],N)$ and $\eta\circ g \in AC_{L^p}([c,d],N)$ respectively. Then we have
\[ \Psi_{\eta\circ g}^{-1} \circ AC_{L^p}(g,N) \circ \Psi_\eta (\sigma) = 
 \theta_N^{-1}\circ (\eta\circ g, \Sigma \circ (\sigma\circ g )) \]
for all $\sigma \in \Psi_\eta^{-1}\left( \mathcal{U}_\eta \cap AC_{L^p}(g,N)^{-1}(\mathcal{U}_{\eta\circ g})\right)$. This set coincides with 
\[ \Psi_\eta^{-1}\left( \mathcal{U}_\eta \cap C(g,N)^{-1}(\mathcal{U}_{\eta\circ g}^C)\right) \]
which is open by the continuity of the map $C(g,N)$.\\
Let $\alpha=\eta\circ g :[c,d]\to N$ and $\tau=\sigma\circ g:[c,d]\to TN$. Then both are absolutely continuous, with $\pi_{TN} \circ \tau =\alpha$, hence $\tau \in \Gamma_{AC}(\alpha)$. Moreover, we have
\begin{align*}
\Psi_{\eta\circ g}^{-1} \circ AC_{L^p}(g,N) \circ \Psi_\eta (\sigma) &= 
 \theta_N^{-1}\circ (\alpha, \Sigma \circ \tau ) \\
 &= \Psi_{\alpha}^{-1}\circ \Psi_\alpha (\tau) \\
 &= \tau \\
 &= \sigma \circ g.
\end{align*}
Hence, $\Psi_{\eta\circ g}^{-1} \circ AC_{L^p}(g,N) \circ \Psi_\eta$ is a restriction of the map 
\[ {\Gamma_{AC}(\eta)}\to {\Gamma_{AC}(\eta\circ g)},\quad \sigma\mapsto \sigma\circ g\] 
which is continuous linear by Proposition \ref{propnog}. Therefore, $AC_{L^p}(g,N)$ is smooth.
\end{proof}

\begin{prop}
Let $p\in [1,\infty]$ and $k\in \N_0\cup\{\infty\}$. Let $M$, $N$ and $L$ be smooth manifolds modeled on sequentially complete locally convex spaces which admit a local addition. If $f:L\times M \to N$ is a $C^{k+2}$-map and $\gamma\in AC_{L^p}([a,b],L)$ is fixed, then
\[ f_* : AC_{L^p}([a,b],M) \to AC_{L^p}([a,b],N),\quad 	\eta\mapsto f\circ (\gamma,\eta)\]
is a $C^k$-map.
\end{prop}
\begin{proof}
Define the smooth map
\[ C_\gamma: AC_{L^p}([a,b],N) \to AC_{L^p}([a,b],L) \times AC_{L^p}([a,b],N), \quad \eta\mapsto (\gamma,\eta).\]
Identifying  $AC_{L^p}([a,b],L) \times AC_{L^p}([a,b],M)$ with $AC_{L^p}([a,b],L\times M)$, we have 
\[f_* = AC_{L^p}([a,b],f)\circ C_\gamma.\] 
Hence, since $AC_{L^p}([a,b],f)$ is $C^k$, $f_*$ is $C^k$.
\end{proof}

\begin{prop}
Let $p\in [1,\infty]$. Let $N$ be a smooth manifold modeled on a sequentially complete locally convex space $E$ which admits a local addition. Then the evaluation map
\[ \varepsilon:AC_{L^p}([a,b],N)\times [a,b]\to N,\quad (\eta,t)\mapsto \eta(t)\]
is continuous and for $t\in [a,b]$ fixed, the point evaluation map 
\[\varepsilon_t:AC_{L^p}([a,b],N) \to N,\quad \eta\mapsto \eta(t)\] 
is smooth.
\end{prop}
\begin{proof}
The evaluation map
\[ \varepsilon_c:C([a,b],N)\times [a,b]\to N, \quad (\eta,t)\mapsto \eta(t)\] 
is continuous and the point evaluation $(\varepsilon_c)_t:C([a,b],N)\to N$, $\eta\mapsto \eta (t)$ is smooth for each $t\in[a,b]$ (see \cite{ASm}). Since the inclusion map $J:AC_{L^p}([a,b],N)\to C([a,b],N)$ is smooth, the result follows from the observation that $\varepsilon=\varepsilon_c\circ (J\times\text{Id}_{[a,b]})$ and $\varepsilon_t=(\varepsilon_c)_t \circ J$ for each $t\in [a,b]$.
\end{proof}


\begin{prop}
Let $p\in [1,\infty]$. Let $N$ be a smooth manifold modeled on a sequentially complete locally convex space $E$ which admits a local addition. For each $q\in N$, the function $\zeta_q:[a,b]\to N$, $t\mapsto q$ is absolutely continuous and the map
\[ \zeta:N\to AC_{L^p}([a,b],N), \quad q\mapsto \zeta_q \]
is a smooth topological embedding.
\end{prop}
\begin{proof}
Consider the local addition $\Sigma:\Omega\to N$ and $\theta_N:\Omega \to \Omega'$ as in Definition \ref{localaddition}. Let $(U,\varphi)$ be a chart around $q\in N$ such that $\{q\}\times U\subseteq \Omega'$ and $(\mathcal{U}_{\zeta_q},\Psi_{\zeta_q}^{-1})$ be a chart in $\zeta_q\in AC_{L^p}([a,b],N)$. If $x\in \varphi \left(U\cap \zeta^{-1}(\mathcal{U}_{\zeta_q})\right)$, then for each $t\in [a,b]$ we have
\begin{align*}
\Psi_{\zeta_q}^{-1}\circ \zeta\circ \varphi^{-1}(x)(t) &= \theta_N^{-1} \left(\zeta_q(t), \zeta_{\varphi^{-1}(x)}(t)\right) \\
&= \theta_N^{-1}\left(q,\varphi^{-1}(x)\right) \\
&= \theta_N^{-1} \circ (q,\varphi^{-1}\circ \Tilde{\zeta}_x(t)) 
\end{align*}
where $\Tilde{\zeta}_x:[a,b]\to E$, $t\mapsto x$ is the constant function. Since the map 
\[\Tilde{\zeta}:E\to AC_{L^p}([a,b],E),\quad x\mapsto \Tilde{\zeta}_x\]  is continuous linear, setting the smooth map
\[h:\varphi(U)\to TN,\quad z\mapsto \theta_N^{-1} \circ(q,\varphi^{-1}(z))\]
we have 
\[  \Psi_{\zeta_q}^{-1}\circ \zeta\circ \varphi^{-1} = AC_{L^p}([a,b],h) \circ \Tilde{\zeta}|_{\varphi(U)}.  \]
Hence $\zeta$ is smooth. It remains to show that $\zeta$ is a topological embedding. Let $t\in [a,b]$ be fixed, then the map 
\[ \varepsilon_t :\text{Im}(\zeta)\to N,\quad \zeta_q \mapsto \zeta_q(t)\] 
verifies that $\varepsilon_t \circ \zeta :N\to N$, $q\mapsto q$. In consequence, $\left(\zeta|^{\text{Im}(\zeta)}\right)^{-1}=\varepsilon_t$ which is continuous since $\varepsilon_t$ is continuous. Therefore, $\zeta$ is a topological embedding.
\end{proof}

\noindent
Following more general cases of manifolds of mappings, such as the case of $C^\ell$-maps (with $\ell \geq 0$) from a compact manifold (possibly with rough boundary) to a smooth manifold which admits a local addition (see e.g. \cite{AGS, Kri}), we will study the tangent bundle of the manifold of absolutely continuous functions. 

\begin{remark}
Let $p\in [1,\infty]$ and $N$ be a smooth manifold modeled on a sequentially complete locally convex space $E$ which admits a local addition. We denote the tangent bundle of $AC_{L^p}([a,b],N)$ as $TAC_{L^p}([a,b],N)$. Since for each $t\in [a,b]$, the point evaluation map $\varepsilon_t:AC_{L^p}([a,b],N)\to N$ is smooth, we have
\[ T\varepsilon_t : TAC_{L^p}([a,b],N) \to TN.\] 
For each $v\in TAC_{L^p}([a,b],N)$ we define the function
\[ \Theta_N(v) : [a,b] \to TN,\quad \Theta_N(v)(t) = T\varepsilon_t (v).\]
\end{remark}

\begin{prop}
    Let $p\in [1,\infty]$ and $N$ be a smooth manifold modeled on a sequentially complete locally convex space $E$ which admits a local addition. If $\eta\in AC_{L^p}([a,b],N)$, then $\Theta_N(v)\in \Gamma_{AC}(\eta)$ for each $v\in T_\eta AC_{L^p}([a,b],N)$ and the map
    \[ \Theta_\eta : T_\eta AC_{L^p}([a,b],N)\to \Gamma_{AC}(\eta),\quad v\mapsto \Theta_\eta(v):=\Theta_N(v) \]
    is an isomorphism of topological vector spaces.
\end{prop}
\begin{proof}
    Let $\Sigma:\Omega \to N$ be a normalized local addition of $N$. Since $\Gamma_{AC}(\eta)$ is a vector space, we identify its tangent bundle with $\Gamma_{AC}(\eta)\times \Gamma_{AC}(\eta)$. Let $\Psi_\eta :\mathcal{V}_\eta \to \mathcal{U}_\eta$ be a chart around $\eta$ such that $\Psi_\eta (0) = \eta$, then
    \[ T\Psi_\eta : T\mathcal{V}_\eta  \simeq \mathcal{V}_\eta\times \Gamma_{AC}(\eta) \to T AC_{L^p}([a,b],N) \]
    is a diffeomorphism onto its image. Moreover, 
    \[ T\Psi_\eta : \{0\}\times \Gamma_{AC}(\eta) \to T_\eta AC_{L^p}([a,b],N) \]
    is an isomorphism of topological vector spaces. We will show that 
    \[ \Theta_\eta \circ T\Psi_\eta (0,\sigma) = \sigma \]
    for each $\sigma\in \Gamma_{AC}(\eta)$. This is equivalent to showing that
    \[ T\varepsilon_t \circ T\Psi_\eta (0,\sigma) = \sigma(t) \quad \text{ for all }t\in [a,b].\]
    Using the geometric description of tangent vectors as equivalence classes of curves, we see that $(0,\sigma)$ is equivalent to the curve $[s\mapsto s\sigma]$. Hence, for each $t\in [a,b]$ we have
    \begin{align*}
        T\varepsilon_t \circ T\Psi_\eta (0,\sigma) &= T\varepsilon_t \circ T\Psi_\eta ([s\mapsto s\sigma ]) \\
        &= T\varepsilon_t ([s\mapsto \Psi_\eta(s\sigma)]) \\
        &= T\varepsilon_t ([s\mapsto \Sigma (s\sigma)])\\
        &=[s\mapsto \Sigma|_{T_{\eta(t)}N} (s\sigma(t))] \\
        &= T_0\Sigma|_{T_{\eta(t)}N} ([s\mapsto s\sigma(t)]).
    \end{align*}
Since $\Sigma$ is normalized we have $T_0\Sigma|_{T_{\eta(t)}N}=\text{id}_{T_{\eta(t)}N}$ and
\[ T\varepsilon_t \circ T\Psi_\eta (0,\sigma) = \sigma (t). \]
In consequence, 
\[ \Theta_\eta \circ T\Psi_\eta (0,\sigma) = \sigma. \]
Since $T\Psi_\eta (0,\cdot)$ is an isomorphism, the map $\Theta_\eta$ is injective. Furthermore, for each $\sigma\in \Gamma_{AC}(\eta)$, there exists a $v\in T_\eta AC_{L^p}([a,b],N)$ with $v=T\Psi_\eta (0,\sigma)$ such that
\[ \Theta_\eta (v) = \sigma.\] Moreover, the function
\[ \Theta_\eta(v):[a,b]\to TN,\quad t\mapsto  \Theta_N(v)(t) = \sigma(t)\in T_{\eta(t)}N \]
is in $AC_{L^p}([a,b],TN)$ and verifies $\pi_{TN}\circ \Theta_\eta(v) =\eta$, making the map $\Theta_\eta$ an isomorphism of topological vector spaces. 
\end{proof}

\begin{remark}
By Theorem \ref{propfon}, the map
\[ AC_{L^p}([a,b],\pi_{TN}):AC_{L^p}([a,b],TN)\to AC_{L^p}([a,b],N),\quad \tau \mapsto \pi_{TN}\circ \tau\]
is smooth. Moreover, for $\eta \in AC_{L^p}([a,b],N)$, we have
\[ AC_{L^p}([a,b],\pi_{TN})^{-1}(\{\eta\}) = \Gamma_{AC}(\eta).\]
Then $AC_{L^p}([a,b],\pi_{TN})$ defines a vector bundle on $AC_{L^p}([a,b],N)$, with the local trivializations provided by the charts $\Psi_\eta$. By Theorem \ref{propfon}, we can apply \cite[Theorem A.12]{AGS} to obtain the following.
\end{remark}

\begin{prop}
   Let $p\in [1,\infty]$ and $N$ be a smooth manifold modeled on a sequentially complete locally convex space $E$ which admits a local addition. Then the map
\[ \Theta_N:TAC_{L^p}([a,b],N)\to AC_{L^p}([a,b],TN), \quad v \mapsto \Theta_N(v) \]
is an isomorphism of vector bundles.
\end{prop}

\begin{prop}
    Let $p\in [1,\infty]$. Let $M$ and $N$ be smooth manifolds modeled on sequentially complete locally convex spaces which admits local addition. If $f:M\to N$ is a smooth map, then the tangent map of the map
    \[ AC_{L^p}([a,b],f):AC_{L^p}([a,b],M)\to AC_{L^p}([a,b],N),\quad \eta\mapsto f\circ \eta \]
    is given by 
    \[ TAC_{L^p}([a,b],f) = \Theta_N^{-1}\circ AC_{L^p}([a,b],Tf) \circ \Theta_M.\]
\end{prop}
\begin{proof}
By Theorem \ref{propfon}, we know that $AC_{L^p}([a,b],f)$ is smooth, thus $TAC_{L^p}([a,b],f)$ makes sense. Let us consider the local addition $\Sigma_M:\Omega_M\to M$ and $\eta \in AC_{L^p}([a,b],M)$.\\ 
Let $\Psi_\eta:\mathcal{V}_\eta\to \mathcal{U}_\eta$ be a chart around $\eta$ such that $\Psi_\eta (0)= \eta$. We consider the isomorphism of topological vector spaces
    \[ T\Psi_\eta : \{0\}\times \Gamma_{AC}(\eta) \to T_\eta  AC_{L^p}([a,b],M). \]
For $t\in [a,b]$ we denote the point evaluation in $M$ and $N$ as $\varepsilon_t^M$ and $\varepsilon_t^N$ respectively, then for each $\sigma \in \Gamma_{AC}(\eta)$ we have
\begin{align*}
    T\varepsilon_t^N\circ TAC_{L^p}([a,b], f) \circ T\Psi_\eta (0,\sigma) &= T\varepsilon_t^N\circ TAC_{L^p}([a,b],f) \circ T\Psi_\eta([s\mapsto s\sigma]) \\
    &= T\varepsilon_t^N\circ TAC_{L^p}([a,b],f) ([s\mapsto \Sigma_M (s\sigma)])\\
    &= T\varepsilon_t^N ([s\mapsto f \circ \Sigma_M (s\sigma)])\\
    &= [s\mapsto \varepsilon_t\left(f \circ \Sigma_M (s\sigma)\right)] \\
    &= [s\mapsto f \circ \Sigma_M (s\sigma(t))]\\
    &= Tf \circ T_0\Sigma_M|_{T_{\eta(t)}M} ([s\mapsto s\sigma(t)]) \\
    &= Tf([s\mapsto s\sigma(t)])\\
    &= Tf(\sigma(t))\\
    &= AC_{L^p}([a,b],Tf)\circ T\varepsilon_t^M \circ T\Psi_\eta (0,\sigma).
\end{align*}
Hence
\[ \Theta_N\circ TAC_{L^p}([a,b],f)  =AC_{L^p}([a,b],Tf)\circ \Theta_M.\]
And applying $\Theta_N^{-1}$ on the left gives the result.
\end{proof}

\begin{example}\label{exlie}
Let $p\in [1,\infty]$. If $G$ is a Lie group modeled on a sequentially locally convex space $E$, then we already know that the set $AC_{L^p}([a,b],G)$ is a Lie group (see \cite{GL2, Nin1}).
We will give an alternative proof of this. Let $e\in G$ be the neutral element, let $L_g:G\to G$, $h\mapsto gh$ be the left translation by $g\in G$ and the left action 
\[ G\times TG\to TG,\quad (g,v_h)\mapsto g.v := TL_{g} (v).\]
If $\varphi:U\subseteq G\to V\subseteq T_e G$ is a chart in $e$ such that $\varphi(e)=0$, then the set \[\Omega_\varphi := \bigcup\limits_{g\in G}g.V \subseteq TG\] 
and the map
\[ \Sigma_\varphi :\Omega_\varphi \to G, \quad v\mapsto \pi_{TG}(v)\left( \varphi^{-1}( \pi_{TG}(v)^{-1}.v)\right) \]
defines a local addition for $G$ (see e.g. \cite{KMr}); hence $AC_{L^p}([a,b],G)$ is a smooth manifold with charts constructed with the local addition $(\Omega_\varphi,\Sigma_\varphi)$. Let $\mu_G:G\times G\to G$ and $\lambda_G:G\to G$ be the multiplication map and inversion maps on $G$, respectively. We define the multiplication map $ \mu_{AC}$ and the inversion map $\lambda_{AC}$ on $AC_{L^p}([a,b],G)$ as
\[ \mu_{AC}  := AC_{L^p}([a,b],\mu_G):AC_{L^p}([a,b],G)\times AC_{L^p}([a,b],G)\to AC_{L^p}([a,b],G)\]
and
\[ \lambda_{AC} := AC_{L^p}([a,b],\lambda_G):AC_{L^p}([a,b],G) \to AC_{L^p}([a,b],G).\]
By Lemma \ref{Acts} and Theorem \ref{propfon}, both maps are smooth. \\
We observe that for the neutral element $\zeta_e:[a,b]\to G$, $t\mapsto e$ of $AC_{L^p}([a,b],G)$ we have 
\[ \Gamma_{AC}(\zeta_e) = AC_{L^p} ([a,b],T_e G). \]
If $\Psi_{\zeta_e}^{-1}:\mathcal{U}_{\zeta_e}\to \mathcal{V}_{\zeta_e}$ is a chart in ${\zeta_e}\in AC_{L^p}([a,b],G)$, then we have $\mathcal{U}_{\zeta_e} = AC_{L^p}([a,b],U)$ and ${\mathcal{V}_{\zeta_{e}} = AC([a,b],V)}$. Moreover, we see that
\begin{align*}
\Psi_{\zeta_e} \circ AC_{L^p}([a,b],\varphi) (\eta) &= \Sigma_\varphi \circ (\varphi\circ\eta) \\
&= \pi_{TG}(\varphi\circ\eta)\left( \varphi^{-1}( \pi_{TG}(\varphi\circ\eta)^{-1}.\varphi\circ\eta)	 \right) \\
&= \varphi^{-1} (e.\varphi\circ \eta) \\
&= \eta.
\end{align*}
This enables us to say that for the neutral element ${\zeta_e}\in AC_{L^p}([a,b],G)$ we have a chart given by
\[ AC_{L^p}([a,b],\varphi):AC_{L^p} ([a,b],U)\to AC_{L^p}([a,b],V),\quad \eta \mapsto \varphi \circ \eta. \]
\end{example}

\section{Semiregularity of right half-Lie groups.}
\noindent
    Let $G$ be a topological group with neutral element $e\in G$. For $g\in G$, we denote the left translation of $g$ by $\ell_g:G\to G$, $h\mapsto gh$ and the right translation of $g$ by $\rho_g:G\to G$, $h\mapsto hg$.
    
\begin{definition}
    A sequentially complete locally convex right half-Lie groups $G$ is a topological group with a smooth manifold structure with modeled space in a sequentially complete Hausdorff locally convex space $E$ such that its right translations are smooth. 
\end{definition}

The following proposition is direct application of Proposition \ref{propfon}.

\begin{prop}\label{rightact}
Let $p\in[1,\infty]$ and $G$ be a sequentially complete locally convex right half-Lie group which admits local addition. Let $\eta \in AC_{L^p}([a,b],G)$ and $g\in G$. We define the function
\[ \eta.g(t):= \eta(t)g,\quad \text{for all }t\in [a,b].\]
Then $\eta.g\in AC_{L^p}([a,b],G)$ for each $g\in G$ and the map
\[ AC_{L^p} ([a,b], \rho_g):AC_{L^p}([a,b],G)\to AC_{L^p}([a,b],G), \quad \eta\mapsto \eta.g \]
is smooth.
\end{prop}

\begin{remark}\label{multipli}
The smoothness of the map
\[ \mathcal{R}:AC_{L^p}([a,b],G)\times G\to AC_{L^p}([a,b],G),\quad (\eta,g)\mapsto  \eta.g \]
would imply the smoothness of the multiplication map on $G$. Indeed, since the point evaluation map $\varepsilon_a:AC_{L^p}([a,b],G)\to G$, $\eta\mapsto \eta(a)$ and the map $\zeta:G\to AC_{L^p}([a,b],G)$, $g\mapsto [t\mapsto g]$, are smooth (see \cite{Pin}), the multiplication map on $G$ would be smooth as it coincides with the composition
\[ \varepsilon_a\circ \mathcal{R}\circ (\zeta,\text{id}_G) : G\times G \to G, \quad (h,g)\mapsto hg. \]
\end{remark}

\begin{definition}
Let  $p\in[1,\infty]$ and $G$ be a sequentially complete locally convex right half-Lie group and $\eta\in AC_{L^p}([0,1],G)$. Let $\{t_0,...,t_n\}$ be a partition of $[0,1]$, for each $ i\in \{1,...,n\}$, denoting $\eta_i = \eta|_{[t_{i-1},t_i]}$, we define the function
\[ \dot\eta:[0,1]\to TG,\quad t\mapsto \dot\eta(t)=\left\{\begin{array}{ll}
    T\varphi_i^{-1}\left( (\varphi_i\circ \eta_i)(t),(\varphi_i\circ\eta_i)'(t) \right)& \text{, if $t\in [t_{i-1},t_i)$} \\
    T\varphi_n^{-1}((\varphi_n\circ \eta)(b), (\varphi_n\circ \eta)'(b))& \text{, if $t=b$.}
\end{array} \right. \]
\end{definition}


\begin{definition}
 Let  $p\in[1,\infty]$ and $G$ be a sequentially complete locally convex right half-Lie group. If $\eta\in AC_{L^p}([0,1],G)$ and $\gamma\in L^p([0,1],T_e G)$, we consider the function
    \[ \eta.\gamma:[0,1]\to TG,\quad t\mapsto \sigma(\eta(t),\beta(t))=(\eta(t),T_e\rho_{\eta(t)}(\gamma(t)).\] 
    We say that $G$ is $L^p$-semiregular  if for each $\gamma\in L^p([0,1],T_eG)$, there exists an $AC_{L^p}$-Carath\'eodory solution $\eta_\gamma\in AC_{L^p}([0,1],G)$ of the equation 
    \[ \dot\eta=\eta.\gamma,\quad \eta(0)=e.\]
such that the differential equation satisfies local uniqueness of Carath\'eodory solutions in the sense of \cite{GHi}. In this case, we define the evolution map of $G$ by
\begin{equation*}
    \Evol: L^p([0,1],T_eG) \to \acspaceg,\quad  \gamma\mapsto \Evol(\gamma):=\eta_\gamma.
\end{equation*}
Additionally, if $G$ admits a local addition, we say that $G$ is $L^p$-regular if $G$ is $L^p$-semiregular and if the evolution map is smooth.
The definition for the case of $L_{rc}^{\infty}$-semiregularity and $L_{rc}^{\infty}$-regularity is analogous.
\end{definition}

We recall subdivision property \cite[Lemma 2.17]{Nin1}.
\begin{lemma}\label{subdiv}
Let $p\in [1,\infty]$, $E$ be a Hasudorff locally convex space and $\gamma\in L^p([0,1],E)$. For $n\in \N$ and $k\in\{0,1,...,n-1\}$ we define the function
\[ \gamma_{n,k}:[0,1]\to E,\quad \gamma_{n,k}(t):= \frac{1}{n} \gamma\left(\frac{k+t}{n}\right).\]
Then $\gamma_{n,k}\in L^p([0,1],E)$. Moreover, if $q$ is a continuous seminorm on $E$, then
\[ \sup_{k\in\{0,...,n-1\}} q(\gamma_{n,k})\to 0,\quad \text{as }n\to \infty.\]
\end{lemma}

\begin{lemma}\label{resemiregular}
Let $p\in[1,\infty]$ and $G$ be a right half-Lie group modeled on a sequentially complete locally convex space which admits local addition. If $B\subseteq L^p([0,1],T_eG)$ is a $0$-neighborhood such that for each $\gamma\in B$, there exists $\eta_\gamma=\Evol(\gamma)\in AC_{L^p}([0,1],G)$ such that
    \[ \dot\eta_\gamma=\eta_\gamma.\gamma,\quad \eta_\gamma(0)=e.\]
Then $G$ is $L^p$-semiregular.
\end{lemma}
\begin{proof}
Let $\gamma\in L^p([0,1],T_e G)$. Without loss of generality we assume there exists a continuous seminorm $q$ on $T_e G$ such that
\[ B =\{ \beta\in L^p([0,1],T_e G) : \lVert \beta \lVert_{L^p,q} < 1 \}. \]
By Lemma \ref{subdiv}, there exists $n=n_q\in \N$ such that for each  $k\in\{0,1,...,n-1\}$ we have $\gamma_{n_q,k}\in B$ and the linear map
\[ \alpha_k :L^p([0,1],T_eG)\to L^p([0,1],T_eG),\quad \beta\mapsto \beta_{n,k} \]
is continuous. Therefore, there is exists an open $\gamma$-neighborhood $W_{\gamma}$ such that
\[ \alpha_k ( W_{\gamma})\subseteq B,\quad \forall k\in \{0,...,n-1\}.\]
Let $\beta\in W_{\gamma}$. Considering the partition of $[0,1]$
\[ \left\{ \left[\frac{k}{n},\frac{k+1}{n}\right] : k\in\{0,....,n-1\} \right\},\]
we denote $\eta_k(1)=\Evol(\beta_{n,k})(1)\in G$ and we define the absolutely continuous function (see Proposition \ref{rightact})
\[ \eta_\beta:[0,1]\to G,\quad \left\{ 
\begin{array}{ll}
    \Evol(\beta_{n,0})(nt), & \text{ if $t\in [0,1/n]$,} \\
    \Evol(\beta_{n,k})(nt-k).\Big(\eta_{k-1}(1)\hdots\eta_0(1)\Big), & \text{ if $t\in  [k/n, (k+1)/n]$.}
\end{array}\right.\]
Let $t\in [0,1/n]$, then $\eta_\beta(0)=e$ and
\begin{align*}
    \dot{\eta_\beta}(t) & = n\beta_{n,0}(nt). \Evol(\beta_{n,0})(nt) \\
   &= n\left(\frac{1}{n}\beta\left(\frac{0+nt}{n}\right)\right) . \Evol(\beta_{n,0})(nt)\\
   &= \beta(t).\eta_\beta (t).
\end{align*}
Let $k\in\{1,..,n-1\}$ and $t\in [k/n,(k+1)/n]$, then 
\begin{align*}
     \eta_\beta\left(\frac{k}{n}\right) &= \Evol(\beta_{n,k})\left(n\frac{k}{n}-k\right).\Big(\eta_{k-1}(1)\hdots\eta_0(1)\Big) \\
     &= \eta_{k-1}(1)\hdots\eta_0(1),
\end{align*}
and
\begin{align*}
    \dot{\eta_\beta}(t) &= n \beta_{n,k}(nt-k) \Big(\eta_{k-1}(1)....\eta_0(1)\Big).\Evol(\beta_{n,k})(nt-k) \\
    &= n\left(\frac{1}{n}\beta\left(\frac{k-(nt-k)}{n}\right)\right) \Big(\eta_{k-1}(1)....\eta_0(1)\Big).\Evol(\beta_{n,k})(nt-k) \\
    &= n\left(\frac{1}{n}\beta\left(\frac{k-(nt-k)}{n}\right)\right).\Evol(\beta_{n,k})(nt-k) \Big(\eta_{k-1}(1)....\eta_0(1)\Big) \\
    &= \beta(t). \eta_\beta(t).
\end{align*}
Thus $\eta_\beta=\Evol(\beta)$ exists. In particular, $\eta_\gamma=\Evol(\gamma)$. Therefore, $G$ is $L^p$-semiregular.
\end{proof}

\begin{definition}
    Let $p\in[1,\infty]$ and $G$ be a sequentially complete locally convex right half-Lie group. We consider the set $C([0,1],G)$ of all continuous functions $\gamma:[0,1]\to G$, endowed with the compact-open topology. If $G$ is $L^p$-semiregular, we consider the evolution map with continuous values
    \[ \Evol_C: L^p([0,1],T_eG) \to C([0,1],G),\quad  \gamma\mapsto (J\circ \Evol)(\gamma), \]
    where $J:AC_{L^p}([0,1],G)\to C([0,1],G)$ is the inclusion map.
\end{definition}

The following lemma is just an application of \cite[Proposition 5.25]{GL2} or \cite[Proposition 4.11]{Nin} to our case.
\begin{lemma}\label{regularity}
 Let $p\in [1,\infty]$ and $G$ be a sequentially complete locally convex right half-Lie group which admit a local addition. If $G$ is $L^p$-semiregular and there exists a $0$-neighborhood $B\subseteq L^p([0,1],T_eG)$ such that the restricted evolution map 
 \[ \Evol_C|_{B}:B\to C([0,1],G),\quad \gamma\mapsto \Evol_C(\gamma),\] 
 is continuous, then $\Evol_C$ is continuous. 
\end{lemma}
\begin{proof}
Since $G$ is a topological group, the map 
\[ C([0,1],G)\times C([0,1],G)\to C([0,1],G),\quad (\eta,\xi)\mapsto \eta.(\xi(1)) \]
is continuous. Let $\gamma\in L^p([0,1],T_eG)$, by Lemma \ref{resemiregular}, there exists a $\gamma$-neighborhood $W_\gamma$ such that for each $\beta\in W_\gamma$ we have
\[ \Evol_C(\beta):[0,1]\to G,\quad \left\{ 
\begin{array}{ll}
    \Evol(\beta_{n,0})(nt), & \text{ if $t\in [0,1/n]$,} \\
    \Evol(\beta_{n,k})(nt-k).\Big(\eta_{k-1}(1)\hdots\eta_0(1)\Big), & \text{ if $t\in  [k/n, (k+1)/n]$.}
\end{array}\right.\]
This construction implies that the map 
\[ \left(\text{Evol}_C\right)|_{W_\gamma}:W_\gamma\to C([0,1],G),\]
is the product of composition of continuous maps, hence $\Evol_C$ is continuous.
\end{proof}

\begin{lemma}\label{lcfunction}
Let $r\in \N_0\cup\{\infty\}$ and $1\leq p <\infty$. Let $E_1$, $E_2$ and $F$ be Hausdorff locally convex spaces and $U\subseteq E_2$ an open subset. If $f:E_1\times U\to F$ is a $C^r$-map such that for each $y\in U$, the map $f(\cdot,y):E_1\to F$, $x\mapsto f(x,y)$ is linear, then the map
\[ \tilde{f}:L^p([0,1],E_1)\times C([0,1],U)\to L^p([0,1], F),\quad (\gamma,\eta)\mapsto f\circ (\gamma,\eta)\]
is $C^r$.
\end{lemma}
\begin{proof}
Let $(\gamma,\eta)\in L^p([0,1],E_1)\times C([0,1],U)$, then $\widetilde{f}(\gamma,\eta)\in L^p([0,1],F)$. Indeed, let $\beta$ be a continuous seminorm on $F$. We consider the compact subset $K:=\eta_0([0,1])$ of $U$. Let $y\in K$. By continuity of $f$ and $f(0,y)=0$, there exists a $y$-neighborhood $V_y \subseteq U$ and a continuous seminorm $\kappa_y$ on $E_1$ such that
\[ f\left( B_1^{\kappa_y} (0)\times V_y \right) \subseteq B_1^\beta (0). \]
Therefore, by linearity of $f$ in the first term, we have
\[ \beta(f(x,z))\leq \kappa_y(x),\quad \forall x\in E_1, \forall z\in V_y.\]
By compactness of $K$, there exists a finite numbers of $y_1,...,y_n\in K$ such that $K\subseteq V:=\cup_{i=1}^n V_{y_i}$. Denoting $\kappa:=\kappa_{y_1}+...+\kappa_{y_n}$, we have
\[ \beta\big( f(x,z)\big) \leq \kappa (x),\quad\forall x\in E_1, \forall z\in V. \]
Moreover, the set $C([0,1],V)$ is open in $C([0,1],U)$. Hence, we have
\begin{align*}
\lVert \tilde{f}(\gamma,\eta)\lVert_{L^p,\beta}  &:= \left(\int_0^1 \beta\big( \tilde{f}\circ (\gamma(t),\eta(t))\big)dt\right)^{\frac{1}{p}} \\
& \leq  \left(\int_0^1 \kappa(\gamma(t))dt\right)^{\frac{1}{p}} \\
& \leq \lVert \gamma \lVert_{L^p,\kappa} .
\end{align*} 
Thus $\tilde{f}(\gamma,\eta)\in L^p([0,1], F)$. Let $\varepsilon>0$. By \cite[Theorem 3.7]{Pin2}, the space $C([0,1],E_1)$ is dense in $L^p([0,1],E_1)$, i.e., there exists $\gamma_c \in C([0,1],E_1)$ such that
\[ \lVert \gamma_c - \gamma_0\lVert_{L^p,\kappa} \leq \varepsilon/4. \]
Since the map
\[ C([0,1],U)\to C([0,1],F),\quad \eta \mapsto f\circ (\gamma_c, \eta) \]
and inclusion map $C([0,1],F) \to L^p ([0,1],F)$ are continuous, the map
\[ C([0,1],V)\to L^p([0,1],F),\quad \eta \mapsto f\circ (\gamma_c, \eta) \]
is continuous. Therefore, there exists an open $\eta$-neighborhood $W_\eta$ in $C([0,1],V)$ such that
\[ \lVert \tilde{f}(\gamma_c, \eta_1)- \tilde{f}(\gamma_c ,\eta ) \lVert_{L^p,\beta} \leq \varepsilon/4, \quad\text{for all }\eta_1 \in W_\eta. \]
Let $\eta_1 \in W_\eta$ and $\gamma_1\in L^p([0,1],E_1)$ such that $\lVert \gamma-\gamma_1 \lVert_{L^p,\kappa}\leq \varepsilon/5$, we estimate
\begin{align*}  
\left\lVert \tilde{f}(\gamma_1, \eta_1)-\tilde{f}(\gamma,\eta) \right\lVert_{L^p,\beta} &
\leq \left\lVert \tilde{f}(\gamma_1,\eta_1) - \tilde{f}(\gamma_c,\eta_1)\right\lVert_{L^p,\beta}+\left\lVert \tilde{f}(\gamma_c,\eta_1)-\tilde{f}(\gamma,\eta)\right\lVert_{L^p,\beta} \\
& \leq \left\lVert \tilde{f}(\gamma_1-\gamma_c,\eta_1)\right\lVert_{L^p,\beta}+
\left\lVert \tilde{f}(\gamma_c,\eta_1)-\tilde{f}(\gamma_c,\eta )\right\lVert_{L^p,\beta} + \left\lVert \tilde{f}(\gamma_c,\eta)-\tilde{f}(\gamma ,\eta )\right\lVert_{L^p,\beta} \\
&= \left\lVert \tilde{f}(\gamma_1-\gamma_c,\eta_1)\right\lVert_{L^p,\beta}+\left\lVert \tilde{f}(\gamma_c ,\eta_1)-\tilde{f}(\gamma_c,\eta )\right\lVert_{L^p,\beta} + \left\lVert \tilde{f}(\gamma_c-\gamma  , \eta  ) \right\lVert_{L^p,\beta} \\
&\leq \Big\lVert \gamma_1-\gamma_c \Big\lVert_{L^p,\kappa}+\left\lVert \tilde{f}(\gamma_c ,\eta_1)-\tilde{f}(\gamma_c,\eta )\right\lVert_{L^p,\beta}+ \Big\lVert \gamma_c-\gamma  \Big\lVert_{L^p,\kappa}.
\end{align*}
Since $\lVert \gamma_c - \gamma\lVert_{L^p,\kappa} \leq \varepsilon/4$ and $\lVert \gamma_1-\gamma\lVert_{L^p,\kappa} \leq \varepsilon/4$ we have
\[ \lVert \gamma_1-\gamma_c \lVert_{L^p,\kappa} \leq 2\varepsilon/4. \]
Thus
\[ \lVert \tilde{f}(\gamma_1, \eta_1)-\tilde{f}(\gamma,\eta) \lVert_{L^p,\beta} \leq \varepsilon/4 + \varepsilon/4 + 2\varepsilon/4 = \varepsilon.\]
Thus $\tilde{f}$ is continuous in $(\gamma,\eta)\in L^p([0,1],E_1)\times C([0,1],U)$. By linearity in the first variable, the map $\overline{f}$ has a continuous differential in the first variable
\[ d_1 \tilde{f}: L^p([0,1],E_1)\times L^p([0,1],E_1)\times  C([0,1],E_2) \to L^p([0,1],F). \]
Let assume that the function $f$ as a $C^1$-map. Then, for $x\in E_1$, $y, y_1\in E_2$ and $t\in \R^{\times}$, we have
\[ \frac{1}{t}\left( f(x,y+t y_1)-f(x,y) \right) = \int_0^1 d_2 f(x,y+tsy_1,y_1)ds \]
whenever $y+[0,1]ty_1\subseteq U$. Given that the map 
\[ d_2 f:E_1\times U\times E_2 \to F,\quad (x,y,h_2)\mapsto D_{(0,h_2)}f(x,y)=d_2f(x,y,h_2)\] is continuous. We identify $C([0,1],E_1\times E_2)\cong C([0,1],E_1)\times C([0,1], E_2)$. Then, we have that the function
\[ \widetilde{d_2 f}:L^p([0,1],E_1)\times C([0,1],U)\times  C([0,1],E_2) \to L^p([0,1],F), \quad (\gamma,\eta, \eta_1)\mapsto d_2 f(\gamma,\eta,\eta_1) \]
is continuous. For each $t\in [0,1]$ fixed, we denote the evaluation function by 
\[g_t : C([0,1],F) \to F,\quad \gamma \mapsto \gamma(t).\]
Since the family of maps $\{ g_t :t\in [0,1]\}$ separate points on $C([0,1],F)$, we arrive to the equality
\[ \frac{1}{t}\left( f\circ (\gamma_c,\eta+t\eta_1)- f\circ (\gamma_c,\eta) \right) = \int_0^1 d_2 f(\gamma_c,\eta+ts\eta_1,\eta_1)ds \]
is valid for each $\gamma_c \in C([0,1],E_1)$, $\eta \in C([0,1],U)$, $\eta_1\in C([0,1],E_2)$ and $t\in \R^{\times}$ such that 
\[\eta+[0,1]t\eta_1 \in C([0,1],U).\] 
By density of $C([0,1],E_1)$ on $L^p([0,1],E_1)$ and continuity of $\tilde{f}$, for $\gamma\in L^p([0,1],E_1)$ the equation verifies
\[ \frac{1}{t}\left( \tilde{f} (\gamma,\eta+t\eta_1)- \tilde{f}(\gamma,\eta) \right) = \int_0^1 \widetilde{d_2 f}(\gamma,\eta+ts\eta_1,\eta_1)ds. \]
Let $\gamma$, $\eta$ and $\eta_0$ be fixed, then the map 
\[ [0,1]\times [0,1]\to F,\quad (t,s)\mapsto \widetilde{d_2 f}(\gamma,\eta+ts\eta_1,\eta_1)\] 
is continuous, including in $(0,s)$. Then, taken the limit $t\to 0$ in the equality we obtain
\[ d_2 \tilde{f}(\gamma, \eta,\eta_1) = \widetilde{d_2f}(\gamma,\eta,\eta_1).\]
Hence the continuity of $\widetilde{d_2f}$ implies that the map $\widetilde{f}$ is $C^1$. 
Proceeding by induction, if $f$ is a $C^k$-map, since $\tilde{f}$ is linear in the first variable we have
\begin{align*}
     d\tilde{f}(\gamma,\eta,\gamma_1,\eta_1)&=d_1\tilde{f}(\gamma,\gamma_1,\eta)+d_2\tilde{f}(\gamma,\eta,\eta_1) \\
     &= \tilde{f}(\gamma_1,\eta) + \widetilde{d_2 f}(\gamma,\eta,\eta_1).
\end{align*}
By the induction hypothesis  $\tilde{f}$ and $d_2 f$ are $C^{r-1}$, hence $\widetilde{d_2f}$ is $C^{r-1}$ with $\widetilde{d_2f}=d_2\tilde{f}$, thus $\tilde{f}$ is $C^r$.
\end{proof}
\noindent
For the case $L_{rc}^\infty$ we recall \cite[Proposition 2.3]{GL2}.
\begin{lemma}\label{lemmadif}
Let $E_1, E_2$ and $F$ be integral complete locally convex spaces and $U\subseteq E_2$ an open subset. If $r\in \N\cup\{0,\infty\}$ and $f:E_1\times U\to F$ is a $C^r$-map such that for each $y\in U$ the map $f(\cdot,y):E_1\to F$ is linear, then the map
\[ \tilde{f}:L_{rc}^\infty([0,1],E_1)\times C([0,1],U)\to L_{rc}^\infty([0,1], F),\quad (\gamma,\eta)\mapsto f\circ (\gamma,\eta)\]
is $C^r$.
\end{lemma}

\begin{theorem}\label{teocontinuity}
    Let $1\leq p < \infty$ and $G$ be a sequentially complete locally convex right half-Lie group which admits a local addition such that the map 
    \[ \sigma : T_eG\times G \to TG,\quad (v,g)\mapsto (g,T_e\rho_g(v)) \]
    is continuous. If $G$ is $L^p$-semiregular and the map evolution map with continuous values $\Evol_C:L^p([0,1],T_eG)\to C([0,1],G)$ is continuous, then the evolution map of $G$
    \[ \Evol:L^p([0,1],T_eG),\to AC_{L^p}([0,1],G),\quad \gamma\mapsto \Evol(\gamma)\]
    is continuous. 
\end{theorem}
\begin{proof}
Let $\gamma\in L^p([0,1],T_eG)$. Let $\mathcal{P}_{\gamma}\subseteq C([0,1],G)$ be a open neighborhood of $\Evol_C(\gamma)$, then there exists a partition $P=\{t_0,...,t_n\}$ of $[0,1]$ and a family of charts $\varphi_i:U_i\to V_i$ such that
\[ \mathcal{P}_{\gamma} = \bigcap_{i=1}^n \{ \eta\in C([t_{i-1},t_i],U_i) : \eta( [t_{i-1},t_i])\subseteq U_i\}.\]
Let $\mathcal{Q}_{\gamma} = \Evol_C^{-1}(\mathcal{P}_{\gamma})$, then $\mathcal{Q}_{\gamma}$ is an open neighborhood of ${\gamma}$ in $L^p([0,1],T_eG)$. We will show that the map
\[ \Evol|_{\mathcal{Q}_{\gamma}}: \mathcal{Q}_{\gamma}\to \mathcal{P}_{\gamma} \cap AC_{L^p}([0,1],G),\quad \gamma \mapsto \Evol(\gamma) \]
is continuous. The map
\[ \Phi_1: \mathcal{Q}_{\gamma} \to \prod_{i=1}^n L^p([t_{i-1},t_i],T_e G),\quad \beta \mapsto \left( [\beta|_{[t_{i-1},t_i]}]\right)_{i=1}^n \]
is a linear topological embedding (see \cite[Lemma 2.16]{Nin1}). For each $i\in\{1,...,n\}$ we have 
\[ \Evol_C(\beta|_{[t_{i-1},t_i]})=\Evol_C(\beta)|_{[t_{i-1},t_i]},\quad \text{ for all } \beta\in \mathcal{Q}_{\gamma}.\]
And each map 
\[ C([t_{i-1},t_i],U_i)\to C([t_{i-1},t_i],V_i),\quad \eta\mapsto \varphi_i\circ \eta\] 
is an homeomorphism. This allow us to define the continuous map
\[ \Phi_2: \mathcal{Q}_{\gamma} \to \prod_{i=1}^n C([t_{i-1},t_i],V_i),\quad \beta\mapsto \left(\varphi_i\circ \Evol_C(\beta)|_{[t_{i-1},t_i]}\right)_{i=1}^n. \]
We consider the map 
\[f_i:T_e G\times V_i \to E,\quad (v,x)\mapsto d\varphi_i \circ \tau(v,\varphi_i^{-1}(x)). \]
Then each $f_i$ is continuous and for each $x\in V_i$ fixed, the function 
\[f_i(\cdot,x):T_e G\to E,\quad v\mapsto f_i(v,x)\] 
is linear. Therefore, by Lemma \ref{lcfunction}, the map
\[ \tilde{f}_i:L^p([t_{i-1},t_i],T_e G)\times C([t_{i-1},t_i],V_i)\to L^p ([t_{i-1},t_i],E),\quad (\beta, \eta)\mapsto f_i \circ (\beta,\eta)\] 
is continuous. We define the function
\[  F:\prod_{i=1}^nL^p([t_{i-1},t_i],T_e G)\times  C([t_{i-1},t_i],V_i)\to  \prod_{i=1}^n L^p ([t_{i-1},t_i],E),\quad\left( [\beta_i],\eta_i \right)_{i=1}^n \mapsto  \left( f_i \circ ([\beta_i],\eta_i)\right)_{i=1}^n.\]
Then $F$ is continuous and for each  $i\in\{1,...,n\}$, we have
\[ f_i \circ ([\beta_i],\eta_i) = d\varphi_i \circ\tau([\beta_i],\varphi_i^{-1}\circ \eta_i). \]
This allow us to define the continuous map 
\begin{align*}
    \Phi_3 : \mathcal{Q}_{\gamma} &\to \prod_{i=1}^n  L^p ([t_{i-1},t_i],E)\times C([t_{i-1},t_i],V_i),\\
\beta &\mapsto \Big(f_i\circ \left(\left[\beta|_{[t_{i-1},t_i]}\right], \varphi_i\circ \Evol_C(\beta)|_{[t_{i-1},t_i]}\right), \varphi_i\circ \Evol_C(\beta)|_{[t_{i-1},t_i]}\Big)_{i=1}^n
\end{align*}
We notices that
\begin{align*}
    f_i \circ \left( \left[\beta|_{[t_{i-1},t_i]}\right], \varphi_i \circ \Evol_C|_{[t_{i-1},t_i]} \right) &=
    d\varphi_i \circ \tau \left( \left[\beta|_{[t_{i-1},t_i]}\right] ,\varphi_i^{-1} \circ  \varphi_i \circ \Evol_C ( \beta ) |_{[t_{i-1},t_i]} \right) \\
    &= d\varphi_i \circ  \tau \left( \left[\beta|_{[t_{i-1},t_i]}\right] , \Evol_C ( \beta ) |_{[t_{i-1},t_i]}\right) \\
    &= d \varphi_i \circ \left(\Evol_C(\beta)|_{[t_{i-1},t_i]}\right)^\prime \\
    &= \left(\varphi_i \circ \Evol_C(\beta)|_{[t_{i-1},t_i]}\right)^\prime.
\end{align*}
Hence
\[ \Phi_3 (\beta) = \Big(\left(\varphi_i \circ \Evol_C(\beta)|_{[t_{i-1},t_i]}\right)', \varphi_i \circ \Evol_C(\beta)|_{[t_{i-1},t_i]}\Big)_{i=1}^n,\quad \text{ for all }\beta\in\mathcal{Q}_{\gamma}. \]
We set the topological embedding (see Lemma \ref{actocl})
\[ \Psi_i : AC_{L^p}([t_{i-1},t_i],V_i)\to  L^p([t_{i-1},t_i],V_i)\times C([t_{i-1},t_i],V_i),\quad \alpha\mapsto (\alpha',\alpha),\]
And
\[ \Psi : \prod_{i=1}^n AC_{L^p}([t_{i-1},t_i],V_i)\to  \prod_{i=1}^n  L^p([t_{i-1},t_i],T_eG) \times C([t_{i-1},t_i],V_i),\quad (\alpha_i)_{i=1}^n \mapsto (\alpha_i^\prime,\alpha_i)_{i=1}^n.\]
Thus $\text{Im}(\Phi_3)\subseteq \text{Im}(\Psi)$. For each $\beta\in\mathcal{Q}_{\gamma}$, we have
\[  \left(\Psi_i|^{\text{Im}(\Psi_i)} \right)^{-1}\Big((\varphi_i \circ \Evol_C(\beta)|_{[t_{i-1},t_i]})', \varphi_i \circ \Evol_C(\beta)|_{[t_{i-1},t_i]}\Big) = \varphi_i \circ \Evol(\beta)|_{[t_{i-1},t_i]}\]
We set the continuous map
\[ \Phi_4 := \left(\Psi|^{\text{Im}(\Psi)} \right)^{-1}\circ \Phi_3 : \mathcal{Q}_{\gamma} \to \prod_{i=1}^n AC_{L^p}([t_{i-1},t_i],V_i),\quad \beta\mapsto \left(\varphi_i \circ \Evol(\beta)|_{[t_{i-1},t_i]}\right)_{i=1}^n. \]
And the map
\[ \Phi_5:AC_{L^p}([0,1],G)\to \prod_{i=1}^n AC_{L^p}([t_{i-1},t_i],G),\quad \eta\mapsto (\eta|_{[t_{i-1},t_i]})_{i=1}^n,\]
which is a homeomorphism onto its image (see \cite[Proposition 4.13]{Pin}) and the homeomorphism
\[ \Phi_6: \prod_{i=1}^n AC_{L^p}([t_{i-1},t_i], V_i)\to \prod_{i=1}^n AC_{L^p}([t_{i-1},t_i],U_i),\quad \tau_i\mapsto (\varphi_i^{-1}\circ \tau_i)_{i=1}^n.\]
We see that
\[ \Phi_6 \circ \Phi_4 : \mathcal{Q}_{\gamma}\to \prod_{i=1}^n AC_{L^p}([t_{i-1},t_i],U_i),\quad \beta\mapsto \left(\Evol(\beta)|_{[t_{i-1},t_i]}\right)_{i=1}^n. \]
Since each function $\Evol(\beta)$ is continuous, we have that $(\Phi_6 \circ \Phi_4) \left(\mathcal{Q}_{\gamma}\right)\subseteq \text{Im}(\Phi_5)$, hence
\[ \Evol|_{\mathcal{Q}_{\gamma}}= \left(\Phi_5|^{\text{Im}(\Phi_5)}\right)^{-1} \circ \Phi_6 \circ \Phi_4. \]
Thus, the evolution map $\Evol$ is continuous.
\end{proof}

\begin{coro}
    Let $G$ be a integral complete locally convex right half-Lie group which admits a local addition such that the map 
    \[ \sigma : T_eG\times G \to TG,\quad (v,g)\mapsto (g,T_e\rho_g(v)) \]
    is continuous. If $G$ is $L_{rc}^\infty$-semiregular and the map evolution map with continuous values $\Evol_C:L_{rc}^\infty([0,1],T_eG)\to C([0,1],G)$ is continuous, then the evolution map of $G$
    \[ \Evol:L_{rc}^\infty([0,1],T_eG),\to AC_{L^p}([0,1],G),\quad \gamma\mapsto \Evol(\gamma)\]
    is continuous.
\end{coro}
\begin{proof}
    By Lemma \ref{lemmadif}, the proof is analogous since the map 
\[ \tilde{f}_i:L_{rc}^\infty([t_{i-1},t_i],T_e G)\times C([t_{i-1},t_i],V_i)\to L_{rc}^\infty ([t_{i-1},t_i],E),\quad (\beta, \eta)\mapsto f_i \circ (\beta,\eta)\] is continous for each $i\in\{1,...,n\}$.
\end{proof}

\section{$L^p$-semiregularity of $\Dif(\Rn)$}
Let $n,r\in \N$, $1\leq p<\infty$ and $K$ be a compact subset of $\Rn$. In this section, we will show that the Banach half-lie group of all $C^r$-diffeomorphisms of $\Rn$ with support in $K$, denotated by $\Dif(\Rn)$, is $L^p$-semiregular with evolution map
\[ \Evol:L^p([0,1],T_e \Dif(\Rn))\to AC_{L^p}([0,1],\Dif(\Rn)),\quad \gamma\mapsto \Evol(\gamma)\]
continuous.
\begin{definition}
    Let $n, m\in \N$ and $r \in \N_0$. We consider the vector space
    \[ C(\Rn, \Rm)= \{\phi:\Rn\to \Rm \mid \text{ $\phi$ is continuous}\}.\]
    For each non empty compact set $L\subseteq \Rn$ we define the seminorm
    \[ \lVert \cdot \lVert_{L} : C(\Rn, \Rm)\to [0,\infty), \quad \phi\mapsto \sup_{x\in L} \lvert \phi(x)\lvert.\]
    Then the space $C(\Rn,\Rm)$, endowed with the locally convex topology generated by these family of seminorms, is a Fr\'echet space. Now, we consider the vector space
    \[ C^r(\Rn,\Rm)=\{\phi\in C(\Rn,\Rm) \mid \text{ $\phi$ is a $C^r$-map}\},\] 
    endowed with the initial topology with respect to the maps
    \[ D^\alpha:C^r(\Rn,\Rm)\to C(\Rn,\Rm),\quad \phi\mapsto \frac{\partial^\alpha \phi}{\partial x^\alpha}\]
    for $\alpha\in \N_0^n$ with $\lvert \alpha \lvert \leq r$. Then $C^r(\Rn,\Rm)$ is a Fr\'echet space. \\ 
    For a fixed compact subset $K\subset \Rn$, we define the vector space of $C^r$-maps supported in $K$ as
    \begin{equation*}
     C_K^r(\Rn,\Rm):= \{ \phi\in C^r(\Rn,\Rm) : \phi|_{\Rn\setminus K} = 0\}
    \end{equation*}
    endowed with the subspace topology. Then $C_K^r(\Rn,\Rm)$ is a Banach space with norm
    \[ \lVert \phi \lVert = \sum_{\lvert \alpha\lvert \leq r} \sup_{x\in K}\lvert D^\alpha(\phi)(x)\lvert,\quad \forall \phi\in C_K^r(\Rn,\Rm). \] 
\end{definition}

\begin{definition}\label{example}
Let $n\in \N$, $r\in \N_0$ and $K\subset \Rn$ be a fixed compact set. Under composition, we define the group of $C^r$-diffeomorphisms of $\Rn$ as
\[\Diff(\Rn)= \{ \phi:\Rn\to\Rn \mid \text{$\phi$ is a $C^r$-diffeomorphism}\}.\]
Denoting $\id:\Rn\to \Rn$ the identity map, we define the subgroup of all $C^r$-diffeomorphisms with support in $K$ by
\begin{equation*}
    \Dif(\Rn):= \{ \phi\in\Diff(\Rn) : \phi-\id \in C_K^r(\Rn,\Rn)\}.
\end{equation*}
We define the open set  
\[ \mathcal{V}_K := \{ \phi-\id \in C_K^r(\Rn,\Rn) : \phi\in \text{Diff}_K^r(\Rn) \},\] and the map
\[  \Phi:\Dif \to \mathcal{V}_K, \quad \phi\mapsto \phi-\id.\] 
Then $(\Phi,\Dif(\Rn))$ is a global chart of $\Dif(\Rn)$, turning it into a topological group and a Banach smooth manifold. 
\end{definition}

\begin{remark}
On $\mathcal{V}_k$ we define the group operation
\[ \phi\ast \psi := \Phi\left( \Phi^{-1}(\phi)\circ \Phi^{-1}(\psi)\right) =\psi+\phi\circ (\id + \psi),\]
for all $\phi,\psi\in \mathcal{V}_K$, with the constant function $0$ as neutral element. Therefore, the map $\Phi$ is a smooth group isomorphism, enabling us to do the identification 
\[ \Dif(\Rn)\cong\mathcal{V}_k\] as topological groups and smooth manifolds. Since $C_K^r(\Rn,\Rn)$ is a Banach space, we have
\[ T\mathcal{V}_K \cong \mathcal{V}_K \times C_K^r(\Rn,\Rn).\]
Let $\psi \in \mathcal{V}_K$ be fixed, the right translation is given by
\[ \rho_\psi : \mathcal{V}_K \to \mathcal{V}_K,\quad \phi \mapsto \psi+\phi\circ (\id+\psi). \]
Its tangent map is 
\[ T\rho_\psi: \mathcal{V}_K \times C_K^r(\Rn,\Rn) \to  \mathcal{V}_K \times C_K^r(\Rn,\Rn),\quad (\phi,f) \mapsto \Big( \psi+\phi\circ (\id+\psi), f\circ (\id+\psi )\Big).\]
Hence, $\mathcal{V}_k$ (and in consequence $\Dif(\Rn)$) is a Banach right half-Lie group. Moreover, identifying $T_0\mathcal{V}_K$ with $C_K^r(\Rn,\Rn)$, we notices that the map
\[ \sigma:C_K^r(\Rn,\Rn) \times \mathcal{V}_k \to T\mathcal{V}_K,\quad (f,\phi) \mapsto \Big( \phi, f\circ (\id+\phi )\Big).\]
is continuous (see e.g. \cite{GL3}). We state this fact as the following lemma.
\end{remark}   

\begin{lemma}\label{sigmadiffk}
    Let $n\in \N$, $r\in \N_0$ and $K\subset \Rn$ be a fixed compact set. Then the map
\[ \sigma:T_{\id}\Dif(\Rn)\times \Dif(\Rn) \to T\Dif(\Rn),\quad (X,\phi)\mapsto \Big(\phi, X\circ\phi\Big).\]
is continuous.
\end{lemma}

\begin{remark}\label{redif}
Let $\phi,\psi\in \mathcal{V}_K$ and $f\in C_K^r(\Rn,\Rn)$, then
\[ (\phi,f).\psi := T\rho_\psi (\phi,f) = (\phi\ast \psi, f\circ (\id+\psi ) ). \]
Let $\gamma\in L^p([0,1],C_K^r(\Rn,\Rn))$. We want to find a function $\eta \in AC_{L^p} ([0,1], \mathcal{V}_K)$ such that 
\[  \dot\eta=\gamma.\eta,\quad \eta(0)=0.\] 
In other words
\[ (\eta(t),\eta'(t))=\Big(\eta(t),\gamma(t)\circ (\id+\eta(t))\Big),\quad \text{ a.e. $t\in [0,1]$.} \]
Therefore
\[ \eta(t)=\int_{0}^t \gamma(s)\circ (\id+\eta(s))ds, \quad \forall t\in [0,1].\]
Setting $\zeta=\id+\eta(s)$, this is equivalent to the integral equation
\[ \zeta(t)=\id+\int_{0}^t \gamma(s)\circ \zeta (s) ds,\quad \forall t\in [0,1].\]
For $x\in \Rn$, let us consider the continuous linear map
\[ \varepsilon_x :C_K^r(\Rn,\Rn)\to \Rn,\quad f\mapsto f(x).\] 
be the point evaluation map.Then the family of maps $(\varepsilon_x)_{x\in \Rn}$ separate points on $C_K^r(\Rn,\Rn)$. Hence, $\zeta\in AC_{L^p}([0,1],\mathcal{V}_K)$ es solution of the integral equation 
\[ \zeta(t)=\id+\int_{0}^t \gamma(s)\circ \zeta (s) ds,\quad \forall t\in [0,1],\]
if, and only if, the functions $\zeta_x:=\varepsilon_x \circ \zeta \in AC_{L^1} ([0,1], \Rn)$ satisfy
\[ \zeta_x(t)=x+\int_{0}^t \gamma(s)\circ \zeta_x (s) ds, \quad\forall t\in [0,1].\]
\end{remark}

\begin{theorem}\label{text-main-3}
Let $n, r\in \N$ and $1\leq p<\infty$. Then the Banach right half-Lie group $\Dif(\Rn)$ is $L^p$-semiregular. Moreover, its evolution map
\[ \Evol: L^p([0,1],T_e\Dif(\Rn)) \to AC_{L^p}([0,1],\Dif(\Rn)),\quad  \gamma\mapsto \eta_\gamma.\]
is continuous.
\end{theorem}
\begin{proof}
Following the discussion in Remark \ref{redif}, we will show that $\mathcal{V}_K$ is locally $L^p$-semiregular.\\ 
We denote derivative map as
\begin{equation*}
D:C_K^r(\Rn,\Rn)\to C_K^{r-1}(\Rn,\R^{n\times n}),\quad  f\mapsto Df:=f'
\end{equation*}
where $f'(x)$ is the Jacobian matrix of $f$ and we define the continuous seminorm on $C_K^r(\Rn,\Rn)$ by
\[ \alpha: C_K^r(\Rn,\Rn)\to [0,\infty),\quad \alpha(f):= \lVert f\lVert_{\mathcal{L}^\infty, \lVert \cdot \lVert_{op}} = \sup_{x\in \Rn} \lVert Df(x) \lVert_{op}.\]
Let $0<L<1$. We denote the open ball centered in $0\in L^p([0,1],C_K^r(\Rn,\Rn))$ by
\[ B_L := \left\{ \gamma\in L^p([0,1],C_K^r(\Rn,\Rn)): \lVert \gamma \lVert_{L^p,\alpha} := \left(\int_0^1 (\alpha\circ\gamma)^p (t)dt\right)^{1/p} < L \right\}. \]
For $x\in \Rn$, we define the smooth map
\[ c:\Rn\to C([0,1],\Rn),\quad x\mapsto [t\mapsto x]. \]
Let $J:L^1 ([0,1],\Rn)\to C([0,1],\Rn)$ be the continuous linear operator
\[ J([\xi])(t):=\int_0^t \xi(s)ds, \text{ for all }t\in [0,1].\]
Let us consider the evaluation map 
\[ \varepsilon:C_K^r (\Rn,\Rn)\times \Rn \to \Rn,\quad (f,x)\mapsto f(x)\] 
which is a $C^{\infty,r}$-map (and hence, a $C^r$-map). By Lemma \ref{lcfunction}, the map
\[ \widetilde{\varepsilon} : L^p([0,1],C_K^r (\Rn,\Rn))\times C([0,1],\Rn)\to L^p([0,1],\Rn),\quad (\gamma,\zeta)\mapsto \varepsilon\circ (\gamma,\zeta) \]
is well defined and is a $C^r$-map. We define the operator 
\[ T:B_L\times \Rn\times C([0,1],\Rn)\to C([0,1],\Rn),\quad (\gamma,x,\zeta)\mapsto c(x)+(J\circ\widetilde{\varepsilon})(\gamma,\zeta).\] 
In other words, if $t\in [0,1]$, then
\[ T(\gamma,x,\zeta)(t) = x + \int_0^t \gamma(s)\left( \zeta(s)\right)ds. \]
Then $T$ is $C^{\infty,\infty,r}$ in the sense of \cite{Alz}. Let $\gamma\in B_L$, by the Mean Value Theorem, for all $t\in [0,1]$ and $x,y\in \Rn$ we have
\[ \gamma(t)(x)-\gamma(t)(y)=\int_0^1 D(\gamma(t))(x+s(y-x))(y-x)ds.\]
Therefore, for $\zeta_1, \zeta_2 \in C([0,1],\Rn)$ and $s\in [0,1]$, we have
\[ \gamma(s)\circ \zeta_2(s) -\gamma(s)\circ\zeta_1 (s) = \int_0^1 \Big( D(\gamma (s)) \Big)\Big(\zeta_1(s)+w(\zeta_2(s)-\zeta_1 (s))\Big)\Big(\zeta_2(s)-\zeta_1(s)\Big)dw. \]
Thus
\begin{align*}
\lVert T(\gamma,x,\zeta_2)-T(\gamma,x,\zeta_1)\lVert_\infty 
& = \sup_{t\in [0,1]} \left\lVert \int_0^t (\gamma(s)\circ \zeta_2(s) -\gamma(s)\circ\zeta_1 (s)) ds  \right\lVert_\infty  \\
&\leq \sup_{t\in [0,1]}\int_0^t\int_0^1 \left\lVert\Big( D(\gamma (s)) \Big)\Big(\zeta_1(s)+w(\zeta_2(s)-\zeta_1 (s))\Big)\Big(\zeta_2(s)-\zeta_1(s)\Big)\right\lVert_\infty dwds \\
& \leq \sup_{t\in [0,1]}\int_0^t \lVert D\gamma(t) \lVert_{\mathcal{L}^\infty, \lVert\cdot\lVert_{op}} \lVert \zeta_2 - \zeta_1 \lVert_{\infty}dt \\
& = \sup_{t\in [0,1]}\int_0^1 \alpha(\gamma(t)) dt \lVert \zeta_2 -  \zeta_1 \lVert_{\infty}\\
&\leq \lVert \gamma \lVert_{L^1, \alpha} \lVert \zeta_2-\zeta_1\lVert_{\infty} \\
&\leq \lVert \gamma \lVert_{L^p, \alpha} \lVert \zeta_2-\zeta_1\lVert_{\infty} \\
&\leq L \lVert \zeta_2-\zeta_1\lVert_{\infty}.
\end{align*}
Hence $T(\gamma,x,\cdot)$ is a $C^r$-map and an $L$-contraction. By Banach's Fixed Point Theorem, there exists $\zeta_{\gamma,x} \in C([0,1],\Rn)$ such that
\[ T(\gamma,x,\zeta_{\gamma,x})=\zeta_{\gamma,x}. \]
By \cite[Lemma 6.2]{GL2} the map
\[ F_1:B_L\times \Rn \to C([0,1],\Rn),\quad (\gamma,x)\mapsto \zeta_{\gamma,x} \]
is $C^r$. By the exponential laws, we have:
\begin{itemize}
\item[i)] $ F_2: \left( B_L\times \Rn\right) \times [0,1]\to \Rn$, $(\gamma,x,t)\mapsto \zeta_{\gamma,x}(t) $ is $C^{r,0}$. 
\item[ii)] $ F_3: B_L\times [0,1]\to C^r(\Rn,\Rn)$, $(\gamma,t)\mapsto \zeta_{\gamma,\cdot}(t)$ is $C^{r, 0}$.
\item[iii)] $ F_4: B_L \to C([0,1],C^r(\Rn,\Rn))$, $\gamma\mapsto \zeta_{\gamma} $ is continuous.
\end{itemize} 
 Let $x\in \Rn$ and $\zeta_{\gamma,x}\in C([0,1],\Rn)$ be the solution of the integral equation
\[ \zeta_{\gamma,x} (t) = x + \int_0^t \gamma(s) \left( \zeta_{\gamma,x} \right) (s) ds,\quad \forall t\in [0,1].\]
Since $\gamma:[0,1]\to C_K^r(\Rn,\Rn)$, if $x\in \Rn\setminus K$ then $\gamma(\cdot)(x)=0$ and 
\[ \zeta_{\gamma,x} (t) = x,\quad \forall t\in [0,1].\]
Therefore, the constant map $c_x:\Rn\to \Rn$ is also solution. Hence, by uniqueness of solutions, we have $\zeta_{\gamma,x} = c_x$. 
Let us consider the constant map $\text{Id}:[0,1]\to C^r(\Rn,\Rn)$, $t\mapsto \id$, we define
\[S:B_L\to C([0,1],C^r(\Rn,\Rn)), \quad \gamma\mapsto F_4(\gamma)-\text{Id}. \]
Then
\[ S(\gamma)(t)=\int_0^t \gamma(s) \left( \zeta_{\gamma,x} \right) (s) ds\in C_K^r(\Rn,\Rn),\quad \forall t\in [0,1].\]
Moreover, since $S$ is continuous and $\mathcal{V}_K$ is open, there is exists an 0-neighborhood $B \subseteq B_L$ such that $S(B)\subseteq  \mathcal{V}_K$. Hence, we consider the continuous map
\[\tilde{S}: B \to C([0,1],\mathcal{V}_K),\quad \gamma\mapsto S(\gamma).\] 
By Remark \ref{redif}, if $\gamma\in B$, then $\eta:=\tilde{S}(\gamma)$ is solution of the integral equation
\[ \eta(t)=\int_{0}^t \gamma(s)\circ (\id+\eta(s))ds,\quad \forall t\in [0,1].\]
Hence $\Evol(\gamma)=\eta$. In consequence, by Lemma \ref{resemiregular}, $\mathcal{V}_K$ is $L^p$-semiregular. Moreover, the evolution map restricted to the $0$-neighborhood $B$ is given by
\[\Evol|_B:B\subseteq L^p([0,1],C_K^r(\Rn,\Rn)) \to AC_{L^p}([0,1], \mathcal{V}_K), \quad \gamma\mapsto \tilde{S}(\gamma).\]
Since $\tilde{S}$ is continuous, by Lemma \ref{regularity}, the evolution map with continuous values $\Evol_C$ is continuous. Since $\sigma$ is continuous (see Lemma \ref{sigmadiffk}), by Theorem \ref{teocontinuity}, the evolution map of $\mathcal{V}_K$
\[\Evol: L^p([0,1],C_K^r(\Rn,\Rn)) \to AC_{L^p}([0,1], \mathcal{V}_K), \quad \gamma\mapsto \Evol(\gamma)\]
is continuous. Therefore, by identification, $\Dif(\Rn)$ is $L^p$-regular and its evolution map continuous.
\end{proof}

Proceeding exactly as the case $L^p$, for the case $L_{rc}^\infty$ we have the same result.

\begin{prop}\label{linftyreg}
If $r\in \N$, then the right half-Lie group $\Dif(\Rn)$ is $L_{rc}^\infty$-semiregular. Moreover, the evolution map
\[ \Evol: L_{rc}^\infty([0,1],T_e\Dif(\Rn)) \to AC_{L_{rc}^\infty}([0,1],\Dif(\Rn)),\quad  \gamma\mapsto \eta_\gamma \]
is continuous.
\end{prop}
\begin{proof}
For $0<L<1$, we define the open ball centered in $0\in L_{rc}^\infty([0,1],C_K^r(\Rn,\Rn))$ via
\[ B_L := \{ \gamma\in L_{rc}^\infty([0,1],C_K^r(\Rn,\Rn)): \lVert \gamma \lVert_{L_{rc}^\infty,\alpha} := \esssup_{t\in [0,1]} \alpha(\gamma(t)) < L \} \]
We will follow the same steps and notation as Proposition \ref{main-3}. Since the evaluation map $\varepsilon$ is $C^{r,\infty}$, by Lemma \ref{lemmadif}, the map 
\[ \Phi_n : L_{rc}^\infty([0,1],C_K^r (\Rn,\Rn))\times C([0,1],\Rn)\to L_{rc}^\infty([0,1],\Rn),\quad \Phi_n(\gamma,\zeta):= \varepsilon\circ ([\gamma,\zeta]) \]
is $C^r$. Moreover, the map 
\[ T:B_L\times \Rn\times C([0,1],\Rn)\to C([0,1],\Rn),\quad (\gamma,x,\zeta)\mapsto T(\gamma,x,\zeta),\] 
given by
\[ T(\gamma,x,\zeta)(t):= x+ \int_0^t \gamma(s)\left( \zeta(s)\right)ds, \quad \forall t\in [0,1].\]
is a $C^{\infty,\infty,r}$-map. Let $\gamma\in B_L$ and $x\in \Rn$ be fixed, for $\zeta_1, \zeta_2 \in C([0,1],\Rn)$, by the Mean Value Theorem we have
\[ \gamma(t)\circ \zeta_2(t) -\gamma(t)\circ\zeta_1 (t) = \int_0^1 \Big( D\gamma (t) \Big) \Big(\zeta_1(t)+s(\zeta_2(t)-\zeta_1 (t))\Big)\Big(\zeta_2(t)-\zeta_1(t)\Big)ds. \]
Therefore
\begin{align*}
\lVert T(\gamma,x,\zeta_2)(t) -   T(\gamma,x,\zeta_1)(t)\lVert_{\infty} 
& \leq \int_0^t\left\lVert \gamma(t)\circ \zeta_2(t) -\gamma(t)\circ\zeta_1 (t) \right\lVert_\infty dt\\
&\leq \int_0^t\int_0^1 \left\lVert \Big( D\gamma (t) \Big) \Big(\zeta_1(t)+s(\zeta_2(t)-\zeta_1 (t))\Big)\Big(\zeta_2(t)-\zeta_1(t)\Big)\right\lVert_\infty ds \\
& \leq \lVert D\gamma(t) \lVert_{\mathcal{L}^\infty, \lVert\cdot\lVert_{op}} \lVert \zeta_2 - \zeta_1 \lVert_{\mathcal{L}^\infty, \lVert\cdot\lVert_{\infty}} \\
& = \alpha(\gamma(t))\lVert \zeta_2 -  \zeta_1 \lVert_{\mathcal{L}^\infty, \lVert\cdot\lVert_{\infty}}.
\end{align*}
Hence 
\[ \lVert \Phi_n(\gamma,\zeta_2) -  \Phi_n(\gamma,\zeta_1)\lVert_{L_{rc}^\infty,\infty} \leq \lVert \gamma \lVert_{L_{rc}^\infty,\alpha}\lVert \zeta_2-\zeta_1\lVert_{\mathcal{L}^\infty,\lVert\cdot\lVert_{\infty}}.\]
Thus, them map $T(\gamma,x,\cdot)$ is a $C^r$-map and a $L$-contraction. Following proof of $L^p$-regularity of $\Dif(\Rn)$, we can show that $\Dif(\Rn)$ is also $L_{rc}^\infty$-semiregular. Moreover, there exists a $0$-neighborhood $B$ in $L_{rc}^\infty([0,1],T_e \Dif(\Rn))$ such that the restricted map evolution map
\[\Evol_C|_B : B\to C ([0,1],\Dif(\Rn)),\quad \gamma\mapsto \Evol_C(\gamma)\] 
is continuous. By Lemma \ref{regularity} the evolution map with continuous values $\Evol_C$ is continuous.Hence, by Theorem \ref{teocontinuity}, the evolution map
\[\Evol: L_{rc}^\infty([0,1],C_K^r(\Rn,\Rn)) \to AC_{L_{rc}^\infty}([0,1], \mathcal{V}_K), \quad \gamma\mapsto \Evol(\gamma)\]
is continuous.
\end{proof}

\section{$L^p$-Semiregularity of $\Diff(M)$}
Let $r\in \N$ and $1\leq p<\infty$ and $M$ be a compact smooth manifold. In this section we will show that the Banach half-lie group of all $C^r$-diffeomorphisms of $M$, denotated by $\Diff(M)$, is $L^p$-semiregular with evolution map
\[ \Evol:L^p([0,1],T_{\id} \Diff(M))\to AC_{L^p}([0,1],\Diff(M)),\quad \gamma\mapsto \Evol(\gamma)\]
continuous. 

\begin{remark}
    Let $r\in N$ and $M$ be a compact smooth manifold, we denote the identity map of $M$ by $\id:M\to M$, its tangent bundle by $\pi_{TM}:TM\to M$ and the space of all $C^r$-vector fields as
    \[ \mathcal{X}^r(M) = \{ X\in C^r(M,TM) : \pi_{TM}\circ X = \id\}.\]  If $g$ is a smooth Riemannian metric on $M$, we denote its exponential map by
\[ \exp:\mathcal{D}\subseteq TM\to M\times M,\quad (p,v)\mapsto \exp_p(v).\]
Let $\mathcal{W}\subseteq TM$ be an open neighborhood of the zero-section such that the map
\[ (\pi_{TM},\exp):\mathcal{W}\subseteq TM \to \mathcal{W}'\subseteq M\times M,\quad v\mapsto (\pi_{TM}(v),\exp(v) ),\]
is a $C^\infty$-diffeomorphism. We notice that if $\mathcal{W}_p:=\mathcal{W}\cap T_pM$, then the map
\[\exp_p|_{\mathcal{W}_p}:\mathcal{W}_p\subseteq T_pM \to \exp_p\left(\mathcal{W}_p\right)\subseteq M \]
is a $C^{\infty}$-diffeomorphism for each $p\in M$. Therefore, if $(p,q)\in \mathcal{W}'$, then
\[ (\pi_{TM},\exp)^{-1}(p,q)=\left(\exp_p|_{\mathcal{W}_p}\right)^{-1}(q).\]
\end{remark}

\begin{definition}\label{diffcharts}
    Let $r\in \N$ and $M$ be a compact smooth manifold without boundary endowed with a Riemannian metric $g$. We define the set of all $C^r$-diffeomorphisms of $M$ by
    \[ \Diff(M)=\{\phi:M\to \phi \mid \text{ $\phi$ is an $C^r$-diffeomorphism}\}.\] 
    Then, under the composition, $\Diff(M)$ is a group with neutral element given by the identity map $\id:M\to M$. For each $\phi \in \Diff(M)$, we define the vector space of $C^r$-sections
    \[\Gamma_{C^r}(\phi):=\{ X \in C^r(M,TM) : \pi_{TM} \circ X = \phi \}.\]
    If $\nabla$ is the Levi-Civita connection of $(M,g)$, for each $X\in \Gamma_{C^r}(\phi)$ we define
    \[ \nabla X:TM\to TM,\quad (p,v)\mapsto \nabla_v X(p). \] 
    For each $p\in M$ we define
    \[ \lVert X\lVert_p = \sup_{\substack{v \in T_pM \\ \|v\|_g = 1}} \lVert \nabla_v X(p)\lVert_g, \quad\forall X\in \Gamma_{C^r}(\phi).\] 
    This enable us to define the norm
    \[ \lVert X\lVert = \sum_{i=1}^r \sup_{p\in M} \lVert \nabla^{(r)}X\lVert_p,\quad \forall X\in \Gamma_{C^r}(\phi).\] 
    Therefore, $(\Gamma_{C^r}(\phi),\lVert\cdot\lVert)$ is a Banach space. We define the open $0$-neighborhood 
    \[ \mathcal{V}_\phi := \{ X\in \Gamma_{C^r}(\phi) : X(M)\subseteq \mathcal{W}\},\]
    and the set
    \[ \mathcal{U}_\phi := \{ \exp\circ X \in \Diff(M) : X\in \mathcal{V}_\phi \}.\]
    The map
    \[ \Psi_\phi:\mathcal{U}_\phi \to \mathcal{V}_\phi,\quad \Psi_\phi(\phi)(p):=\left(\exp_p|_{\mathcal{W}_p}\right)^{-1}(\phi(p)).\]
    is a bijection with inverse map given by
    \[ \Psi_\phi^{-1}:\mathcal{V}_\phi \to \mathcal{U}_\phi,\quad X\mapsto \exp\circ X.\]
    We endow $\Diff(M)$ with the topology such that the charts
    \[ \{ \Psi_\phi^{-1}:\mathcal{U}_\phi\to \mathcal{V}_\phi \mid \phi\in \Diff(M)\}\] 
    are smooth diffeomorphisms. This manifold structure makes $\Diff(M)$ a Banach smooth manifold and a topological group.
\end{definition}

\begin{remark}
    For $\phi\in \Diff(M)$, we consider the tangent space $T_\phi\Diff(M)$ as the set of equivalence class of smooth curves $\alpha:I\subseteq \R\to \Diff(M)$ with $\alpha(0)=\phi$.\\
    Let $p\in M$, we denote the evaluation map by
    \[ \varepsilon_p:\Diff(M)\to M,\quad \phi\mapsto \phi(p),\]
    which is $C^r$. Then
    \[ T_\phi\varepsilon_p:T_\phi\Diff(M)\to TM,\quad [\alpha]\mapsto T_\phi\varepsilon_p([\alpha])=[\varepsilon_p\circ \alpha].\]
    For $[\alpha]\in T_\phi\Diff(M)$, we define the $C^r$-section
    \[ X_{[\alpha]}:M\to TM,\quad p\mapsto [\varepsilon_p\circ \alpha].\]
    with $\pi_{TM}\circ X_{[\alpha]}= \phi$, hence $X_{[\alpha]}\in \Gamma_{C^r}(\phi)$. Therefore, the linear map
    \[ \Lambda_\phi:T_\phi \Diff(M) \to \Gamma_{C^r}(\phi),\quad [\alpha]\mapsto  X_{[\alpha]}\]
    defines a isomorphism of vector spaces. From now on, we identify $T_\phi\Diff(M)\cong \Gamma_{C^r}(\phi)$. In particular, 
    \[ T_{\id} \Diff(M) \cong \mathcal{X}^r(M).\]
\end{remark}

\begin{remark}
    For $\phi\in \Diff(M)$, the right translation is given by 
    \[ \rho_\phi:\Diff(M)\to \Diff(M),\quad \psi\mapsto \psi\circ \phi,\] 
    and its tangent map is
    \[ T\rho_\phi: T\Diff(M)\to T\Diff(M),\quad (\psi,X_\psi)\mapsto (\psi\circ\phi, X_\psi\circ\phi).\] 
    Then, the manifold $\Diff(M)$ is a Banach right half-Lie group (see e.g. \cite{GL4}). Moreover, we state the following lemma.
\end{remark}    
    
\begin{lemma}\label{sigmadiffm}
    Let $r\in \N$ and $M$ be a compact smooth manifold. Then the map
    \[ \sigma:T_{\id}\Diff(M)\times \Diff(M)\to T\Diff(M),\quad (X,\phi)\mapsto (\phi,X\circ \phi)\]
    is continuous.
\end{lemma}

\begin{remark}
    Let $r\in \N$ and $1\leq p<\infty$. If $\gamma\in L^p([0,1],\mathcal{X}^r(M))$ and $\eta\in AC_{L^p}([0,1],G)$, then 
    \[ \gamma(t).\eta(t) = \gamma(t)\circ \eta(t),\quad \forall t\in [0,1].\] 
    Since we are dealing with differential equation in Banach smooth manifolds, we need to recall some results.
\end{remark}

\begin{definition}
Let $(E,\lVert\cdot\lVert_{E})$ be a normed space and $U\subseteq E$ be a subset. We say that a function $f:[a,b]\times U\to E$ satisfies an $L^1$-Lipschitz condition if there exists a function $g\in L^1([a,b])$ such that
\[ \lVert f(t,x)-f(t,y)\lVert_E \leq g(t)\lVert x-y\lVert_E,\quad \forall x,y\in U.\]
\end{definition}
\begin{definition}
Let $M$ be a $C^1$-manifold modeled on a normed space $E$, $J\subseteq \R$ be a non-degenerate interval. We consider the function
\[ f:J\times M\to TM,\quad (t,p)\mapsto f(t,p)\in T_p M.\] 
We say that $f$ satisfies a local $L^1$-Lipschitz condition if for all $(t_0,p)\in J\times M$, there exists a chart $\kappa:U_\kappa\to V_\kappa$ of $M$ on $p$ and a relatively open subinterval $[a,b]\subseteq J$, which is a neighborhood of $t_0$ in $J$, such that the map
\[ f_\kappa :[a,b]\times V_\kappa\to E,\quad (t,x)\mapsto d\kappa\left( f(t,\kappa^{-1}(y))\right) \]
satisfies an $L^1$-Lipschitz condition.
\end{definition}

\begin{remark}\label{reuni}
Let $M$ be a smooth manifold modeled in the Banach space $E$. We consider the map
\[ f:[0,1]\times M\to TM,\quad (t,p)\mapsto f(t,p)\in T_p M,\]
satisfying a local $L^1$-condition. If $\tau, \eta:[0,1]\to M$ are two $AC_{L^1}$-Carath\'eodory solutions to
\[\dot{y}=f(t,y) \]
satisfying $\tau(t_0)=\eta(t_0)$ for some $t_0\in [0,1]$, then $\tau=\eta$ \cite[Proposition 10.5]{GL2}.\\
Let $p\in M$. If the initial value problem
\begin{align*}
    y'(t)&=f(t,y(t)),\quad t\in [0,1]\\
    y(0)&=p
\end{align*}
has a (necessarily unique) solution $\eta_p:[0,1]\to M$, then we say that $f$ admits a global flow 
\[ F_f:[0,1]\times M\to M,\quad (t,p)\mapsto \eta_p(t).\]
\end{remark}

In a compact smooth manifold, the $C^r$-analogue of \cite[Proposition 11.4]{GL2} holds with the same proof.
\begin{theorem}\label{flow}
Let $r\in \N$, $1\leq p <\infty$ and $M$ be a compact smooth manifold. If $\gamma\in L^p([0,1],\mathcal{X}^r(M))$, then map
\[ \overline{\gamma}:[0,1]\times M\to TM, \quad \overline{\gamma}(t,q):=\gamma(t)(q)\]
satisfies a local $L^1$-Lipschitz condition. Let $\eta\in AC_{L^p}([0,1],\Diff(M))$ with $\eta(0)=\text{id}_M$. Then $\eta=\Evol(\gamma)$ if, and only if, $\overline{\gamma}$ admits global flow $F_{\overline{\gamma}}$ for initial time $t_0=0$ and 
\[ \eta(t)(p)=F_{\overline{\gamma}} (t,q),\quad \forall t\in [0,1], \forall q\in M.\]
\end{theorem}

Analogously to the proof of $L^p$-regularity of the Banach Lie group $\Difff(M)$ (see \cite[Section 11]{GL2}), we have the following result.
\begin{theorem}\label{text-main-4}
Let $r\in N$, $1\leq p <\infty$ and $M$ be a compact smooth manifold. Then the right half-Lie group $\Diff(M)$ is $L^p$-semiregular with evolution map
\[ \Evol:  L^p([0,1],T_{\id}\Diff(M))\to AC_{L^p}([0,1],\Diff(M)),\quad \gamma\mapsto \eta_{\gamma} \]
is continuous.
\end{theorem}
\begin{proof}
The half-Lie group $\Diff(M)$ is $L^p$-semiregular if for each $\gamma\in L^p([0,1],\mathcal{X}^r(M)$, there exists $\eta_\gamma\in AC_{L^p}([0,1],\Diff(M))$ such that
\begin{align*}
    \dot\eta_\gamma(t)&=\gamma(t)\circ \eta_\gamma(t),\quad \forall t\in [0,1],\\
    \eta_\gamma(0)&=\id. 
\end{align*}
Let $\phi\in\Diff(M)$. We consider a finite cover $(U_i)_{i=1}^n$ of open and relatively compact subsets of $M$ and charts 
\[ \kappa_i:U_i\subseteq M \to B_5(0)\subseteq \R^n,\] 
such that the family of open sets $\kappa_i^{-1}(B_1(0))$ cover $M$ and $U_i \subseteq \phi^{-1}(U_{\psi_i})$ for some chart $\psi_i:U_{\psi_i}\to V_{\psi_i}$ of $M$ (see e.g. \cite{Lang}). For each $X\in \Gamma_{C^r}(\phi)$, we define the map
\[ X_i := d\psi_i \circ X \circ \kappa_i^{-1}:B_5(0)\to \R^m,\quad x\mapsto d\psi_i(X(\kappa_i(x))). \]
Then, for each $\ell\in [1,5]$ the map
\[ \rho_\ell:\Gamma_{C^r}(\phi)\to  \prod_{i=1}^n  C^r(B_\ell(0),\R^m),\quad X\mapsto \rho_\ell(X):=\left(X_i|_{B_\ell(0)}\right)_{i=1}^n \]
is a linear topological embedding with closed image (see e.g. \cite{AGS}).\\ 
Considering $\phi=\id$, by corresponding identifications of product spaces, we obtain the following linear topological embeddings with closed image
\[ R_5:=L^p([0,1], \rho_5):L^p([0,1],\mathcal{X}^r(M))\to \prod_{i=1}^n L^p([0,1],C^r(B_5(0),\R^m)), \quad \gamma\mapsto \rho_5\circ \gamma\]
and
\[ R_1:=C([0,1], \rho_1):C([0,1],\mathcal{X}^r(M))\to \prod_{i=1}^n C([0,1],C^r(B_1(0),\R^m)), \quad \gamma\mapsto \rho_1\circ \gamma.\]
First, we wil study the $L^p$-regularity of $C^r(B_5(0),\R^m)$. Let 
\[q:C^r(B_5(0),\R^m)\to [0,\infty),\] 
be the seminorm given by
\[ q(\phi):= \sup_{x\in \overline{B}_4(0)} \left( \lVert D\phi(x) \lVert_{op}+\lVert \phi(x)\lVert_\infty\right), \quad \forall\phi\in C^r(B_5(0),\Rn). \]
Let $0<L<1$, we consider the open ball with respect to $q$
\[B_L:=\{ \gamma\in L^p([0,1],C^r(B_5(0),\R^m)) : \lVert \gamma \lVert_{L^p,q} < L \}. \]
Following the case of $\Dif(\Rn)$, for $\gamma\in B_L$ and $x\in B_3(0)$ fixed, we need an $AC_{L^p}$-Carath\'{e}odory solution $\zeta:[0,1]\to B_4(0)$ that verifies the integral equation
\[ \zeta'(s)=x+\int_0^t \gamma(s)\left(\zeta(s)\right) ds,\quad \forall t\in [0,1].\]
We consider the continuous linear map $J:L^p([0,1],C^r(B_5(0),\R^m))\to C([0,1],C^r(B_5(0),\R^m))$ given by
\[ J(\xi)(t)=\int_0^t \xi(s)ds,\quad \forall t\in [0,1].\]
The evaluation map
\[\varepsilon:C^r(B_5(0),\R^m)\times B_5(0)\to\Rm,\quad (f,x)\mapsto f(x)\]
which is $C^r$. Therefore, By Lemma $\ref{lcfunction}$, the map
\[\widetilde{\varepsilon}:L^p([0,1],C^r(B_5(0),\R^m))\times C([0,1],B_5(0))\to L^p([0,1],\Rm),\quad (\gamma,\eta)\mapsto \varepsilon\circ (\gamma,\eta)\]
is $C^r$. This enable us to define the $C^r$-map 
\[T:B_L\times B_3(0)\times C([0,1],B_4(0))\to C([0,1],B_4(0))\] 
defined as
\[ T(\gamma, x, \zeta )(t):=x+\int_0^t \gamma(s)\left(\zeta(s)\right) ds,\quad t\in [0,1].\]
Moreover, following the steps of Theorem \ref{text-main-3}, for $\gamma\in B_L$ and $x\in B_3(0)$, we can show that
\[  \lVert T(\gamma,x,\zeta_2 ) - T(\gamma,x,\zeta_1 )\lVert_\infty\leq L\lVert \zeta_2-\zeta_1\lVert_\infty,\quad \forall \zeta_1,\zeta_2\in  C([0,1],B_4(0)).\]
Thus, by Banach's Fixed Point Theorem, for each $(\gamma,x)\in B_L\times B_3(0)$ the map $T(\gamma,x,\cdot) $ has a unique fixed point $\zeta_{\gamma,x}\in C([0,1],B_4(0))$ and
\[ F:B_L\times B_3(0) \to C([0,1],B_4(0)),\quad (\gamma,x)\mapsto \zeta_{\gamma,x}\]
is $C^{r}$ (see \cite[Lemma 6.2]{GL2}). By the exponential law, we have 
\begin{itemize}
\item[i)] $F_2: \left( B_L\times B_3(0)\right) \times [0,1] \to B_4(0)$, $(\gamma,t,x)\mapsto \zeta_{\gamma,x}(t)$ is $C^{r, 0}$. 
\item[ii)] $ F_3: B_L\times [0,1]\to C^r(B_3(0),B_4(0))$, $(\gamma,t)\mapsto \zeta_{\gamma}(t) $ is $C^{r, 0}$.
\item[iii)] $ F_4: B_L \to C([0,1],C^r(B_3(0),B_4(0)))$, $\gamma\mapsto \zeta_{\gamma} $ is continuous.
\end{itemize} 
Considering the continuous map
\[ \rho_2:C^r(B_3(0),B_4(0))\to C^r(B_2(0),B_4(0)),\quad \varphi\mapsto \varphi|_{B_2(0)}\]
we define continuous map
\[ H:=C([0,1],\rho_2)\circ F_4 : B_L \to C([0,1],C^r(B_2(0),B_4(0)), \quad \gamma\mapsto \rho_2\circ \zeta_{\gamma}. \]
Then
\[ H(\gamma)\in AC_{L^p}( [0,1],C^r(B_2(0),B_4(0)),\quad \forall \gamma\in B_L.\]
Indeed, let us consider the open set 
\[ \lfloor \overline{B_2(0)},B_4(0) \rfloor_r :=\{\varphi\in C^r(B_3(0),\R^m):\varphi(\overline{B_2(0)})\subseteq B_4(0)\},\] 
and the map
\[ S: C^r(B_5(0),\Rn) \times \lfloor \overline{B_2(0)},B_4(0) \rfloor_r  \to C^r(B_2(0),\R^m),\quad (\psi,\varphi)\mapsto \psi\circ \varphi \]
Then $S$ is continuous and linear in the first variable. By Lemma \ref{lcfunction}, the map
\begin{align*}
    \tilde{S}: L^p([0,1],C^r(B_5(0),\Rn))\times C([0,1], \lfloor\overline{B_2(0)},B_4(0) \rfloor_r)&\to L^p([0,1],C^r(B_2(0),\R^m))\\
    (\alpha,\beta)&\mapsto S\circ (\alpha,\beta)
\end{align*}
is continuous. If $\gamma\in B_L$, then
\[\widetilde{S}(\gamma,H(\gamma))=\gamma\circ H(\gamma)=\gamma.H(\gamma) \in L^p([0,1],C^r(B_2(0),\R^m)). \]
This allow us to define the $L^p$-absolutely continuous function $\tau:[0,1]\to C^r(B_2(0),\R^m)$ given by
\[ \tau(t) =\text{id}_{B_2(0)}+\int_0^t \gamma(s) \circ  H(\gamma)(s)ds,\quad \forall t\in[0,1]. \]
For each $x\in B_2(0)$, we denote the evaluation map 
\[ \varepsilon_x:C^r(B_2(0),\R^m)\to \R^m,\quad f\mapsto f(x).\]
Then $\{ \varepsilon_x :x\in B_2(0)\}$ separate points in $C^r(B_2(0),\R^m)$. Therefore, we have
\begin{align*}
    \varepsilon_x\left(\tau(t)\right) &= x+\int_0^t \gamma(s)\circ H(\gamma)(s)(x)ds\\
    &= x+\int_0^t \gamma(s)\circ \zeta_{\gamma,x}(s)ds \\
    &=\zeta_{\gamma,x}(t) \\
    &= H(\gamma)(x)(t).
\end{align*}
Thus, $\varepsilon_x(\tau)=\varepsilon_x(H(\gamma))$ and in consequence, $\tau=H(\gamma)$. Hence, the half-lie group $C^r(B_2(0),B_4(0))$ is $L^p$-semiregular and
\[  \Evol|_{B_L}:B_L \to AC_{L^p}([0,1],C^r(B_2(0),B_4(0)), \quad \gamma\mapsto H(\gamma).\]
Now we study the case for $L^p ([0,1], \mathcal{X}^r(M))$. We define the $0$-neighborhood of $L^p ([0,1], \mathcal{X}^r(M))$
\begin{align*}
\mathcal{B}_L&:= R_5^{-1} \left( \prod_{i=1}^n B_L \right) \\
&= \{ \gamma\in L^p ([0,1], \mathcal{X}^r(M)) : (\forall i\in \{1,...,n\})\text{ } \gamma_i \in B_L \}.
\end{align*}
Then $\mathcal{B}_L$ is open by continuity of $R_5$. Let $\gamma\in \mathcal{B}_L$, with $\gamma_i\in B_L$ for each $i\in\{1,...,n\}$. For each $m\in M$, there exists a chart $\kappa_j:U_j\to B_5(0)$ such that $m\in \kappa_j^{-1}(B_3(0))$ for some $j\in \{1,...,n\}$. Let $\zeta_{\gamma,\kappa_j(m)}\in AC_{L^P}([0,1], C^r(B_2(0),B_4(0))$ be the solution of the integral equation
\[ \zeta_{\gamma,\kappa_j(m)}(t)=\kappa_j(m)+\int_0^t \gamma_j(s)\left(  \zeta_{\gamma,\kappa_j(m)}(s)\right)ds,\quad \forall t\in [0,1]. \]
In other words, $ \zeta_{\gamma,\kappa_j(m)}$ is a $AC_{L^p}$-Carath\'eodory solution to
\begin{align*}
    x'(t)&=\gamma_j(t)\circ x(t),\quad t\in[0,1], \\
    x(0)&=\kappa_j(m).
\end{align*}
This implies that the function
\[\eta_{\gamma,m}:[0,1]\to M,\quad \eta_{\gamma,m}(t):=\kappa_j^{-1}\circ \zeta_{\gamma,\kappa_j(m)}(t)\] 
is a $AC_{L^p}$-Carath\'eodory solution of the equation
\begin{align*}
    y'(t)&=\gamma(t)\left( y(t)\right),\quad t\in[0,1], \\
    y(0)&=m.
\end{align*}
By Theorem \ref{flow}, the map $\tilde{\gamma}$ satisfies a local $L^1$-Lipschitz condition. By Remark \ref{reuni}, the solution of this equation is unique and therefore, $\eta_{\gamma,m}$ is well defined. Moreover, $\tilde{\gamma}$ admits a global flow for initial time $t_0=0$, given by
\[ F_{\overline{\gamma}}(t,m)=\eta_{\gamma,m}(t),\quad \forall t\in [0,1], \forall m\in M.\]
Let $i\in \{1,...,n\}$. We can construct an exponential map on $B_5(0)$
\[\exp_i:\mathcal{D}_i\subseteq TB_5(0)\to B_5(0)\]
such that
\[ \kappa_i^{-1}\circ \exp_i = \exp\circ T\kappa_i^{-1}|_{\mathcal{D}_i}.\]
Moreoveor, we consider an open set $\mathcal{O}_i\subseteq \mathcal{D}_i$ containing $\overline{B_4(0)}\times \{0\}$ such that $(\text{pr}_1,\exp_i)(O_i)$ is open in $B_5(0)\times B_5(0)$ and the map
\[ \Theta_i:=(\text{pr}_1, \exp_i)|_{\mathcal{O}_i}:\mathcal{O}_i\subseteq TB_5(0)\to B_5(0)\times B_4(0), \quad (p,v)\mapsto (p,\exp_i(p,v))\]
is a $C^\infty$-diffeomorphism onto its image. Assume that $T\kappa_i^{-1}(\mathcal{O}_i)\subseteq \mathcal{W}$. Since
\[ \{(x,x):x\in \overline{B_4(0)}\}\subseteq (\text{pr}_1,\exp_i)(\mathcal{O}_i),\]
there exists $0<s_i\leq 1$ such that
\[ \bigcup_{x\in B_4(0)} \{x\}\times B_{s_i}(x) \subseteq (\text{pr}_1,\exp_i)(\mathcal{O}_i). \]
Therefore, there exists a $0<s_i\leq 1$, such that $\Theta_i^{-1}$ restrict to a smooth diffeomorphism of the form 
\[(\text{Id}_{B_4(0)},\theta_i):\bigcup_{x\in B_4(0)} \{x\}\times B_{s_i}(x)\subseteq B_5(0)\times B_5(0) \to \Theta_i^{-1}\left( \bigcup_{x\in B_4(0)} \{x\}\times B_{s_i}(x)\right)\subseteq \mathcal{O}_i\]
with open image in $\mathcal{O}_i$. Let $x\in B_4(0)$ be fixed, we define the set
\[ \mathcal{O}_{i,x}:=\{ y\in \R^m : (x,y)\in \mathcal{O}_i \}.\]
Then, we notices that
\[ \theta_i(x,\cdot)=\left( \exp(x,\cdot)|_{\mathcal{O}_{i,x}} \right) ^{-1}|_{B_{s_i}(0)}.\] 
The set
\[Z_{i,x} := \left\{ \varphi \in C^r(B_2(0),\Rn) : \varphi\left(\overline{B_1(0)}\right)\subseteq \bigcup_{x\in \overline{B_1(0)}}  B_{s_i}(x) \right\}\]
is open in $C^r(B_2(0),\R^m)$ with the compact-open topology. The smoothness of $\theta_i$ implies that the map
\[ (\theta_i)_* : Z_{i,x}\subseteq C^r(B_2(0),\R^m) \to C^r(B_1(0),\R^m),\quad \varphi\mapsto (\theta_i)_*(\varphi):=\theta_i\circ (\text{id}_{B_1(0)},\varphi)\]
is smooth (see e.g. \cite[Proposition 4.23]{GL4}). Let
\[ \Psi:\mathcal{U}\to \mathcal{V},\quad \Psi(\phi)(p)=\left(\exp_p|_{\mathcal{W}_p}\right)^{-1}(\phi(p))\] 
be a chart on $\text{id}_M\in \Diff(M)$. Since $\rho_1$ is a topological embedding, we can assume that
\[ \mathcal{V}:=\rho_1^{-1}\left(\prod_{i=1}^n \mathcal{V}_i \right)\]
for suitable open $0$-neighborhoods $\mathcal{V}_i \subseteq C^r(B_1(0),\R^m)$. \\
Since $(\theta_i)_*(\text{id}_{B_2(0)})=0$, by continuity of $(\theta_i)_*$, there exists open $\text{id}_{B_2(0)}$-neighborhoods $\mathcal{Y}_i \subseteq Z_{i,x}$  such that \[ (\theta_i)_*(\mathcal{Y}_i)\subseteq \mathcal{V}_i.\] 
We notices that, if $\text{id}_{B_2(0)}:B_2(0)\to B_4(0)$ denotes the inclusion map, then
\[H(0):[0,1]\to C^r(B_2(0),B_4(0)),\quad t\mapsto \id|_{B_2(0)}.\] 
By continuity of $H$, there exists open $0$-neighborhoods $\mathcal{P}_i \subseteq \mathcal{B}_L$ such that
\[ H(\mathcal{P}_i)\subseteq C([0,1],\mathcal{Y}_i)\] 
Denoting
\[ \mathcal{P}:= R_5^{-1}\left(\prod_{i=1}^n \mathcal{P}_i \right), \]
we see that $\mathcal{P}$ is an open $0$-neighborhood in $ L^p([0,1],\mathcal{X}^r(M))$ with $\mathcal{P} \subseteq \mathcal{B}_L$.\\  
This allow us to do the composition
\[ (\theta_i)_*\circ H: \mathcal{P}_i \subseteq L^p([0,1],C^r(B_5(0),\R^m)\to C([0,1],C^r(B_1(0),\R^m). \]
Analogously to the case $\text{Diff}_c^\infty(M)$, the same proof of \cite[Points $11.18$ and $11.19$]{GL2}, we can show that for each $\gamma\in \mathcal{P}$, there exists an unique $\theta_{\gamma}\in C([0,1],\mathcal{X}^r(M))$ such that
\[ R_1(\theta_{\gamma})=\Big( (\theta_i)_* ( H([\gamma_i]))\Big)_{i=1}^n. \]
Then $\rho_1\left (\theta_{\gamma}(t)\right) \in \prod_{i=1}^n  \mathcal{V}_i$, whence $\theta_{\gamma}(t)\in \mathcal{V}$. In consequence  
\[ \Psi^{-1}(\theta_{\gamma}(t)) = (\exp\circ \theta_{\gamma})(t) \in \Diff(M),\quad \forall t\in [0,1].\]
Let $m\in M$, denoting $x:=\kappa_i(m)$, we notices that
\begin{align*}
    (\exp\circ \theta_{\gamma}) (t)(m) &= \left(\kappa_i^{-1}\circ \exp_i \circ T\kappa_i\right) \circ \theta_{\gamma}(t) (\kappa_i^{-1}(x))\\
    &= \kappa_i^{-1}\circ \exp_i \circ (\theta_i)_*( H([\gamma_i])) (t) (x) \\
    &= \kappa_i^{-1}\circ \exp_i \circ \left(\exp_i|_{\mathcal{W}_i}\right)^{-1}\circ H([\gamma_i]) (t) (x)\\
    &= \kappa_i^{-1}\circ H([\gamma_i]) (t) (x)\\
    &= \kappa_i^{-1}\circ \zeta_{[\gamma_i],x}(t)\\
    &= F_{\tilde{\gamma}}(t,m).
\end{align*}
This enable us to define the function
\[ \eta_{\gamma}:=\exp\circ \theta_{\gamma}:[0,1]\to \Diff(M).\]
We will show that $\eta_{\gamma}$ is absolutely continuous. By smoothness of the exponential map, is enough to show that $\theta_{\gamma}$ is absolutely continuous. This is equivalent to show that
\[ \rho_1\circ \theta_{\gamma}:[0,1]\to \prod_{i=1}^n C^r(B_1(0),\R^m),\quad t\mapsto \rho_1\circ \theta_{\gamma}(t).\]
is absolutely continuous. We notice that this map coincides with
\[ [0,1]\to \prod_{i=1}^n C^r(B_1(0),\R^m),\quad t\mapsto \left((\theta_i)_* \left(H([\gamma_i])\right)\right)_{i=1}^n (t)\]
Since $(\theta_i)_*$ is smooth and $H([\gamma_i])\in AC_{L^p}([0,1],C^r(B_2(0),B_4(0)))$, by Proposition \ref{propfon}, the composition and hence the product are absolutely continuous. Therefore
\[ \theta_{\gamma}\in AC_{L^p}([0,1],\mathcal{X}^r(M))\]
whence 
\[ \eta_{\gamma}\in AC_{L^p}([0,1],\Diff(M)).\]
In consequence, by Lemma \ref{resemiregular} and Remark \ref{reuni}, the right half-Lie group $\Diff(M)$ is $L^p$-semiregular. Since the map 
\[ B_L\to \prod_{i=1}^n C([0,1],C^r(B_1(0),\Rn)),\quad [\gamma_i]\mapsto \left((\theta_i)_*\left( H([\gamma_i])\right)\right)_{i=1}^n \]
is continuous, the restricted evolution map with continuous values is given by
\[ \Evol_C|_{\mathcal{P}}:\mathcal{P}\to AC_{L^p}([0,1],\Diff(M)),\quad \gamma\mapsto \eta_{\gamma}=\exp \circ \theta_{\gamma}.\]
Hence, by Lemma \ref{regularity} the evolution map $\Evol_C$ is continuous. Moreover, since $\sigma$ is continuous (see Lemma \ref{sigmadiffm}), by Theorem \ref{teocontinuity}, the evolution map
\[ \Evol:L^p([0,1],\mathcal{X}^r(M))\to AC_{L^p}([0,1],\Diff(M)),\quad \gamma\mapsto \Evol(\gamma)\]
is continuous.
\end{proof}

Now we will focus in the $L^p$-semiregularity of the case $\Diff(M)$ with $M$ a compact smooth manifold with boundary.
\begin{definition}
Let $M$ and $N$ be smooth manifolds with boundaries and suppose ${f:\partial M\to \partial N}$ is a diffeomorphism, We define adjunction space $M\cup_f N$ as the set formed identifying each $p\in \partial M$ with $f(p)\in \partial N$.
\end{definition}
\noindent
We recall \cite[Theorem 9.29]{JLee}.

\begin{theorem}\label{Lee 1}
Let $M$ and $N$ be a smooth manifolds with boundaries and ${f:\partial M\to \partial N}$ be a diffeomorphism. Then the adjunction space $M\cup_f N$ is a topological manifold without boundary which has a smooth manifold structure such that there are regular domains $M', N' \subseteq M\cup_f N$ diffeomorphic to $M$ and $N$, respectively, such that $M'\cup N'=M\cup_f N$ and ${M'\cap N'=\partial M'=\partial N'}$. Moreover, $M$ and $N$ are both compact if and only if $M\cup_f N$ is compact.
\end{theorem}

\begin{definition}\label{double}
Let $M$ be a smooth manifold with boundary. If $M'$ denotes a copy of $M$, we define the double of $M$ as the smooth manifold without boundary
\begin{equation}
 DM=M\sqcup_{\text{id}_\partial} M'
\end{equation}
where $\text{id}_\partial:\partial M \to \partial M'$ is the identity map.\\ 
For $p\in M$, we denote by $(p,0)$ and $(p,1)$ the corresponding element in $M$ and $M'$, respectively, and if $p\in \partial M$, then $(p,0)\sim (p,1)$. By Theorem \ref{Lee 1}, the map
\[ \rho: M\to DM,\quad p\mapsto [(p,0)] \]
is an embedding onto a regular domain of $DM$ which we identify with $M$.
\end{definition}

\begin{definition}
Let $r\in \N\cup\{0,\infty\}$ and $M$ be a smooth manifold with boundary. We define the vector space of $C^r$-stratified vector field on $M$ as
\[ \mathcal{X}_{str}^r(M):=\{ X\in \mathcal{X}^r(M) : X(\partial M)\subseteq T\partial M \}.\]
\end{definition}

\noindent
We recall \cite[Colorally 1.8]{GL5}.
\begin{prop}\label{exten prop}
    For each $k\in \N\cup\{0,\infty\}$, $n\in \N$, $m\in \{0,...,d\}$ and locally convex space $F$, the restriction map
   \[ \mathcal{E}:C^k(\R^d,F)\to C^k([0,\infty)^m\times \R^{d-m},F)\] 
    has a continuous linear right inverse. Moreover, the restriction map
    \[ C^k (\R^d,F)\to C^k([0,1]^d,F)\]
    has continuous linear right inverse.
\end{prop}

\begin{prop}\label{de alpha}
Let $r\in \N\cup\{0,\infty\}$, $m\in \N$ and $M$ be an $m$-dimensional compact smooth manifold with boundary. Then there exists an extension map 
\[ \alpha: \mathcal{X}_{str}^r(M)\to \mathcal{X}^r(DM),\quad X\mapsto \alpha(X)\] 
which is continuous linear.
\end{prop}
\begin{proof}
By compactness of $M$, we find charts $\varphi_i:U_i\to V_i$ of $M$ around points $p_i\in \partial M$ such that $(U_i)_{i=1}^k$ is a finite open cover of $\partial M$ which extends to charts $\widetilde{\varphi_i}:\widetilde{U_i}\to \widetilde{V_i}$ around $p_i$. For $X\in \mathcal{X}_{str}^r(M)$, we write
\[ Y_i:=d\varphi_i \circ X \circ \varphi_i^{-1}: V_i\subseteq [0,\infty)\times \R^{m-1}\to \R^m.\]
Without loss of generality, we assume that $V_i=[0,\infty)\times \R^{m-1}$ and we consider the extension $\widetilde{Y}_i:\Rm\to \Rm$ of $Y_i$ given by Proposition \ref{exten prop}. We define
\[ \widetilde{X}_i := T\widetilde{\varphi_i}^{-1}\circ \widetilde{Y}_i \circ \widetilde{\varphi_i} : \widetilde{U_i}\to T\widetilde{U_i}. \]
Then $\widetilde{X}_i \in \mathcal{X}^r(\tilde{U_i})$ and the map
\[ \Phi_1 : \mathcal{X}_{str}^r(M)\to \prod_{i=1}^k \mathcal{X}^r(\widetilde{U_i}),\quad X\mapsto (\widetilde{X}_i)_{i=1}^k \]
is continuous linear. Let us consider $\widetilde{U}_{k+1}:=M^\circ$, $\widetilde{U}_{k+2}:=(M')^\circ$ and the open cover of $DM$
\[\mathcal{A}:=\{\widetilde{U_1},...,\widetilde{U_k},\widetilde{U}_{k+1}, \widetilde{U}_{k+2}\}.\]  
Then there exists a partition of the unity subordinate to $\mathcal{A}$, denoted by $\{h_1,..., h_k, h_{k+1}, h_{k+2}\}$, such that $\text{supp}(h_i)\subseteq \widetilde{U_i}$ for each $i\in \{1,...,k,k+1,k+2\}$. \\
Denoting the space of all $C^k$-vector fields of $\widetilde{U}_i$ with support on $\text{supp}(h_i)$ as $\mathcal{X}_i^r(\widetilde{U}_i)$, we see that the map
\[ \mathcal{X}^r(\widetilde{U}_i)\to \mathcal{X}_i^r(\widetilde{U}_i),\quad Y\mapsto h_i Y \]
is continuous linear.\\ 
Moreover, if $Z\in \mathcal{X}_i^r(\widetilde{U}_i)$, we can obtain an extension $\mathcal{E}_i (Z)\in \mathcal{X}^r(DM)$ of $Z$ by extending it by $0$, and a continuous linear map
\[ \mathcal{E}_i:\mathcal{X}_i^r(\widetilde{U}_i)\to \mathcal{X}^r(DM),\quad Z\mapsto \mathcal{E}_i (Z). \]
For $X\in \mathcal{X}_{str}^r(M)$, we write $\widetilde{X}_{k+1}:=X|_{M^\circ}$. This enables us to define the extension map 
\[ \alpha: \mathcal{X}_{str}^r(M)\to \mathcal{X}^r(DM),\quad X\mapsto \alpha(X):=\sum_{i=1}^{k+1} \mathcal{E}_i(h_i \widetilde{X}_i) \]
which is continuous linear.
\end{proof}

\begin{remark}
Let $M$ be a compact smooth manifiold with boundary and $r\in \N\cup\{\infty\}$. By \cite[Proposition 1.3]{MILAN}, the set $\Diff(M)$ is a embedded submanifold of $C^r(M,DM)$. This allows us to consider the inclusion map restricted onto its image  
\begin{equation}\label{7 Jembe}
    J:\Diff(M)\to J(\Diff(M))\subset C^r(M,DM),\quad \phi\mapsto \phi
\end{equation}
then $J$ is a diffeomorphism. Since for each $g\in  C^r(M,M)$ fixed, the right translation map 
\[ \rho_{C^r}(g): C^r(M,DM)\to C^r(M,DM),\quad \phi\mapsto \phi\circ g\] 
is smooth. For each $g\in  \Diff(M)$ fixed, the right translation map 
\[ \rho(g) :\Diff(M)\to \Diff(M),\quad \phi\mapsto \phi\circ g\]
can be written as 
\[ \rho(g) = J^{-1}\circ \rho_{C^r}(J(g)) \circ J.\]
Hence $\Diff(M)$ is a right half-Lie group.
\end{remark}
\begin{theorem}
Let $M$ be a compact smooth manifold with boundary, $r\in \N\cup\{\infty\}$ and $1\leq p < \infty$. Then the right half-Lie group $\Diff(M)$ is $L^p$-semiregular and the evolution map with continuous values
\[ \Evol_C:L^p([0,1],T_e \Diff(M) ) \to C([0,1],\Diff(M)),\quad [\gamma]\mapsto \Evol_C(\gamma)\]
is continuous.
\end{theorem}
\begin{proof}
Let consider the map $\alpha:\mathcal{X}_{str}^r(M)\to \mathcal{X}^r(DM)$ as Proposition \ref{de alpha}, we define the map
\[ \tilde{\alpha}:=L^p([0,1],\alpha):L^p([0,1],\mathcal{X}_{str}^r(M))\to L^p([0,1],\mathcal{X}^r(DM)),\quad [\gamma]\mapsto [\alpha\circ \gamma];\]
which is linear and continuous. Since $DM$ is a compact smooth manifold without boundary, the right half-Lie group $\text{Diff}^r(DM)$ is $L^p$-semiregular with continuous evolution map denoted by
\[ \Evol_{DM}:L^p([0,1],\mathcal{X}^r(DM))\to AC_{L^p}([0,1],\text{Diff}^r(DM)),\quad [\gamma]\mapsto \Evol_{DM}(\gamma).\]
Let $[\gamma]\in L^p([0,1],\mathcal{X}_{str}^r(M))$, we define the absolutely continuous function ${\xi_\gamma :[0,1]\to \text{Diff}^r(DM)}$ by $\xi_\gamma := \Evol_{DM}(\tilde{\alpha}([\gamma]))$. Then $\xi_\gamma$ is the solution of the equation
\begin{align*}
   \dot{\xi}_\gamma(t)&= \tilde{\alpha}([\gamma])(t).\xi_\gamma(t),\quad t\in [0,1]\\ 
   \xi_\gamma(0)&=e.
\end{align*}
For $[\gamma]$ close to $0$, the proof of Theorem \ref{text-main-4} shows that, for each $p\in M$, the function $x_p:[0,1]\to DM$, given by ${x_p(t):=\xi_\gamma(t)(p)}$ is a solution of the equation
\begin{align*}
    \dot{ x_p}(t)&= \tilde{\alpha}([\gamma])(t)\circ x_p(t),\quad t\in [0,1]\\ 
    x_p(0)&=p,
\end{align*}
and this ODE satisfies local uniqueness of Caratheodory solutions.
Looking at the compact manifold $\partial M$ without boundary and the vector fields $\gamma(t)|_{\partial M}\in \mathcal{X}^r(\partial M)$ we likewise get a solution $y_p:[0,1]\to \partial M$ for each $p\in \partial M$, for the differential equation 
\begin{align*}
    \dot{ y_p}(t)&= \tilde{\alpha}([\gamma])(t)\circ y_p(t),\quad t\in [0,1]\\ 
    y_p(0)&=p.
\end{align*}
Then $y_p$ also solves the initial value problem for $x_p$, whence $x_p=y_p$ by local uniqueness. In consequence $x_p([0,1])\subseteq \partial M$.
This implies that for each $t\in [0,1]$ fixed, we have $\xi_\gamma(t)(M)\subseteq M$, $\xi_\gamma(t)(M')\subseteq M'$ and $\xi_\gamma(t)(\partial M)\subseteq \partial M$. Therefore, we obtain ${\xi_\gamma(t)|_M\in \Diff(M)}$. \\
Consider the smooth embedding $\iota:\text{Diff}^r(DM)\to C^r(DM,DM)$, $\phi\mapsto \phi$. Then the map
\[ \tilde{\iota}:=AC_{L^p}([0,1],\iota): AC_{L^p}([0,1],\text{Diff}^r(DM)) \to AC_{L^p}([0,1],C^r(DM,DM)),\quad \eta \mapsto \iota\circ \eta \]
is smooth. Let $\rho:M\to DM$ be the inclusion map, which is smooth and a diffeomorphism onto its image, then the maps
\[ \rho^*:=C^r(\rho, DM): C^r(DM,DM)\to C^r(M,DM),\quad \phi\mapsto \phi\circ \rho = \phi|_{M} \]
and  
\[ \tilde{\rho}:AC_{L^p}([0,1],C^r(DM,DM))\to AC_{L^p}([0,1],C^r(M,DM)),\quad \eta\mapsto \rho^*\circ \eta \]
are smooth. \\
Consider the restricted inclusion map $J$ as in equation (\ref{7 Jembe}), by Lemma \ref{pre1}, we define the map
\[ \tilde{j}:AC_{L^p}([0,1],J(\Diff(M)))\to AC_{L^p}([0,1],\Diff(M)),\quad \eta \mapsto J^{-1}\circ \eta. \]
The fact that $\xi_\gamma(t)|_M\in \Diff(M)$ for each $t\in [0,1]$, enables us to define the function
\[ \zeta_\gamma:[0,1]\to C^r(M,DM),\quad \zeta_\gamma(t):= (\tilde{j}\circ\tilde{\rho}\circ \tilde{\iota})(\xi_\gamma)(t)\]
which is in $AC_{L^p}([0,1], \Diff(M))$. Moreover, by definition of $\alpha$, we have
\[ \tilde{\alpha}([\gamma])(t)\circ \zeta_\gamma(t) = \gamma(t)\circ \zeta_\gamma(t). \]
Hence $\zeta_\gamma$ verifies the equation
\begin{align*}
   \dot{ \zeta_\gamma}(t)&= \gamma(t).\zeta_\gamma(t), \quad t\in [0,1]\\ 
   \zeta_\gamma(0)&=e.
\end{align*}
Therefore $\Diff(M)$ is $L^p$-semiregular and the evolution map is given by
\[ \Evol : L^p([0,1],\mathcal{X}_{str}^r(M))\to AC_{L^p}([0,1],\Diff(M)),\quad [\gamma]\mapsto \Big(\tilde{j}\circ\tilde{\rho}\circ\tilde{\iota}\circ \Evol_{DM} \circ \tilde{\alpha}\Big) ([\gamma]). \]
We consider the inclusion map
\[ \omega:AC_{L^p}([0,1],C^r(M,DM)) \to C([0,1],C^r(M,DM)),\quad \eta\mapsto \eta \]
which is smooth, and we define the smooth map
\[ \widetilde{\rho}_C:AC_{L^p}([0,1],C^r(DM,DM))\to C([0,1],C^r(M,DM)),\quad \eta\mapsto \omega(\rho^*\circ \eta).\]
We write
\[ \widetilde{j}_C:C([0,1],J(\Diff(M)))\to C([0,1],\Diff(M)),\quad \eta \to J^{-1}\circ \eta \]
then $\widetilde{j}_C$ is continuous. Therefore, the evolution map with continuous values $\Evol_C$ is given by
\[ \Evol_C : L^p([0,1],\mathcal{X}_{str}^r(M))\to C([0,1],\Diff(M)),\quad [\gamma]\mapsto \Big(\widetilde{j}_C\circ\widetilde{\rho}_C\circ\widetilde{\iota}\circ \Evol_{DM} \circ \widetilde{\alpha}\Big) ([\gamma]) \]
and is continuous since is composition of continuous maps.
\end{proof}
\newpage
\noindent
{\bf Acknowledgements}: The author would like to thank Helge Gl\"ockner for his guidance during
the development of this work. The author was partially supported by the FONDECYT Grant \#1241719 and by ANID and DAAD (DAAD/BecasChile 2020, ID:91762237/62190017).

\end{document}